\documentclass[11pt]{article}
\usepackage{epsfig}
\usepackage{amssymb,amsmath,amsthm,amscd}
\usepackage{latexsym}

\usepackage[colorlinks,
            linkcolor=red,
            anchorcolor=blue,
            citecolor=green
            ]{hyperref}

\newtheorem{thm}{Theorem}[section]
\newtheorem{prop}[thm]{Proposition}
\newtheorem{cor}[thm]{Corollary}
\newtheorem{lem}[thm]{Lemma}

\theoremstyle{definition}
\newtheorem{defn}[thm]{Definition}
\newtheorem{rem}[thm]{Remark}

\newcommand{\be}{\begin{equation}}
\newcommand{\ee}{\end{equation}}

\newcommand{\R}{\mathbb{R}}
\newcommand{\N}{\mathbb{N}}
\newcommand{\E}{\mathbb{E}}

\def \cald {{  {\mathcal{D}}  }}

 \def  \eps { { \varepsilon } }

\numberwithin{equation}{section}

\begin{document}

\baselineskip=1.2\baselineskip

\pagestyle{plain}
\title{ Invariant measures, periodic measures and
pullback measure attractors
of   McKean-Vlasov stochastic reaction-diffusion equations
on unbounded domains.
\footnote{This work was supported
 by National Natural Science Foundation of China (12471154, 12071384, 12090010, 12090013, 12071317, 12471170).}}

\author{{Lin Shi$^\text{a}$, 
Jun Shen$^\text{b}\footnote{Corresponding author.}$, 
Kening Lu$^\text{b}$, 
Bixiang Wang$^\text{c}$
}\\
  { \small\textsl{$^\text{a}$ School of Mathematical Sciences, University of Electronic Science and Technology of China, }}\\
  { \small \textsl{Chengdu 611731,   P.R. China}}\\
  { \small\textsl{$^\text{b}$ School of Mathematics,  Sichuan University,  Chengdu 610064, P.R. China}}\\
  { \small\textsl{$^\text{c}$ Department of Mathematics, New Mexico Institute of Mining and Technology, Socorro, NM 87801, USA} }
}
\footnotetext{
\emph{E-mail addresses}: shilinlavender@163.com(L. Shi),
junshen85@163.com(J. Shen),
keninglu2014@126.com(K. Lu),
Bixiang.Wang@nmt.edu(B. Wang).
}

\date{}
\maketitle

{ \bf Abstract.}
This paper deals with the long term dynamics
of the  non-autonomous
McKean-Vlasov stochastic reaction-diffusion equations
 on $\R^n$.  We first prove
the existence and uniqueness of pullback measure attractors
of the non-autonomous dynamical system generated by
the solution operators defined in  the space of probability measures. We then prove the
existence and uniqueness of invariant measures and
periodic measures of the equation
under further conditions. We finally establish the
upper semi-continuity of pullback measure attractors
as well as the  convergence of
invariant measures and
periodic measures when the distribution dependent
stochastic equations converge to a distribution
independent  system.



\section{Introduction}

 In this paper, we investigate the existence, uniqueness
 and limiting behavior of invariant measures, periodic
 measures and pullback measure attractors of the
   McKean-Vlasov stochastic
   reaction-diffusion  equation
 driven by nonlinear noise defined on $\mathbb{R}^n$:
$$
    d u(t,x)  -\Delta u(t,x) dt
    +\lambda u(t,x) dt
           + f(t, x, u (t,x), \mathcal{L}_{u (t)} ) dt
           $$
       \be \label{sde1}
             = g(t,x,u (t,x), \mathcal {L}_{u (t)} )dt
              +
                  \sum_{k=1}^{\infty}
                          \left( \theta_{k}(t,x) + \kappa(x) \sigma_{k}(t, u (t,x), \mathcal {L}_{u (t)})
                          \right) dW_k(t),
                          \quad t>\tau,
\ee
 with initial data
\begin{align}\label{sde2}
   u(\tau, x) = u_\tau (x),
   \quad x\in \R^n,
\end{align}
 where
  $\lambda>0$ is a constant,
 $\mathcal{L}_{u (t)}$
 is the distribution  of   $u (t)$,
 $f$ is a nonlinear  function with arbitrary growth rate,
 $g$ is a Lipschitz function,
 $\theta_k : \mathbb{R} \to L^2(\mathbb{R}^n)$
 is given,
 $\kappa \in H^1(\mathbb{R}^n) \bigcap
 W^{1, \infty} (\mathbb{R}^n)$,
 $\sigma_{k}$ is  a nonlinear diffusion term,
 and
 $\{W_k\}_{k\in \mathbb{N}}$ is a sequence of independent
 standard  real-valued Wiener processes
 on a complete filtered probability space $\left( \Omega, \mathcal{F}, \{\mathcal{F}_t\}_{t \geq 0}, \mathbb{P} \right)$.

   The  McKean-Vlasov stochastic
   differential  equations (MVSDEs)
   were first
     studied  in   \cite{McKean1966, Vlasov1968}
     which   often
     arise from the
     interacting particle systems
 \cite{BH2022,   DV1995, FG2015, S1991}.
 The striking feature of such differential equations
 lies in the fact that the equations
      depend  not only on the states of  the solutions,
 but also on the distributions of  the solutions.
 As a result,
 the Markov operators associated with
MVSDEs are no longer  semigroups
 (see, e.g., \cite{FWang2018}), and hence
 the  arguments to deal with the
 stochastic equations independent of distributions
 cannot be directly applied to  the MVSDEs.

 Currently, many publications are devoted to the study
 of the solutions of MVSDEs.
   For example,
   the existence of solutions of  the MVSDEs
   has been investigated in
   \cite{AD1995,  FHSY2022, GHL2022, HDS2021,  RZ2021,
   FWang2018},
    the Bismut formula in
   \cite{Banos2018, RW2019},
   the averaging principles
     in  \cite{LWX2022, RSX2021},
    and the large deviations in
     \cite{CW2024, HLL2021, LSZZ2023}
     and the references therein.
     In particular, the existence and ergodicity
     of invariant measures and periodic measures
     of the  {\it finite-dimensional}
     MVSDEs   have been
     explored in
     \cite{Bao2022,Hua2024,
     Hua2024b,  Lia2021, FWang2023, FWang2023b,
       Zha2023}
     and  in \cite{Ren2021}, respectively.
     However,
     As far as the authors are aware,
     the invariant measures and periodic measures
     of the {\it infinite-dimensional}
     MVSDEs have not been studied in the literature.
      In the present paper, we  will investigate the
      existence and  uniqueness
     of invariant measures and periodic measures
     of the  {\it infinite-dimensional}
     McKean-Vlasov stochastic
     reaction-diffusion equation
       as given by \eqref{sde1} defined
     on $\R^n$.

In the aforementioned references, the
      existence of invariant measures of the
     MVSDEs  was  obtained by the fixed
     point  theory.
     In this paper, we will apply  the theory
     of pullback measure attractors
     instead of the fixed point argument
     to prove the existence of invariant measures
     for the
McKean-Vlasov stochastic equation \eqref{sde1}.

 The concept of   measure attractor
 was first introduced in
 \cite{Schmalfuss91}
 for  the stochastic Navier-Stokes equations.
 Since then,   the
existence of measure attractors for stochastic
differential  equations has been further
investigated in  \cite{LW2024, MC1998,Morimoto1992,
Schmalfuss1997,
SWeng2024}.
Note that in all of these articles, the stochastic equations
do not depend on distributions of  solutions.
This paper is the first one to deal with measure attractors
of
the
McKean-Vlasov
stochastic  equations
which are dependent  on the laws of  solutions.

        To describe the main results of this paper,
        we denote by
        $\mathcal{P}(L^2(\R^n))$ the space of probability measures
        on $(L^2(\R^n), \mathcal{B}(L^2(\R^n)))$,
        where $\mathcal{B}(L^2(\R^n))$ is the
        Borel $\sigma$-algebra of $L^2(\R^n)$.
        The weak topology of
 $ \mathcal{P}(L^2(\R^n))$
 is metrizable, and the corresponding   metric
 is denoted by   $d_{\mathcal{P}(L^2(\R^n))}$.
 Set
$$
\mathcal P_4 \left( L^2(\R^n) \right)
= \left\{ {\mu  \in \mathcal
P\left(  L^2(\R^n) \right):
\int_{L^2(\R^n)}
 {\|\xi \|_{L^2(\R^n) }^4   d \mu \left(  \xi  \right) <    \infty } } \right\}.
$$
Then
$\left (
{\mathcal{P}}_4   ( L^2(\R^n)), \
d_{\mathcal{P}(L^2(\R^n))}
\right  )$ is a metric space.
 Given  $r>0$, denote by
$$
 B_ {
 \mathcal{P}_4   ( L^2(\R^n))
 }  (r)
 = \left\{ {\mu  \in {\mathcal
P}_4\left(  L^2(\R^n) \right):
\int_{L^2(\R^n)}
 {\|\xi \|_{L^2(\R^n) }^4   d \mu \left(  \xi  \right)
 \le r^4  } } \right\}.
$$

 Given $   \tau \le t  $
 and  $\mu\in {\mathcal{P}}_4   ( L^2(\R^n))$,
 let
 $P^*_{\tau, t} \mu$ be the law of
 the solution
  of \eqref{sde1}-\eqref{sde2}
with initial law $\mu$
  at initial time $\tau$.
  If $\phi:  L^2(\R^n) \to \R$ is a bounded Borel function,
  then we write
  $$
  P_{\tau, t} \phi (u_\tau)
  =\E \left (
  \phi (u(t, \tau, u_\tau))
  \right ),\quad \forall \ u_\tau \in  L^2(\R^n),
  $$
  where $u(t, \tau, u_\tau)$
  is the solution of \eqref{sde1}-\eqref{sde2}
  with initial value $u_\tau$
  at initial time $\tau$.
  Note that for the
  McKean-Vlasov stochastic
     equations like \eqref{sde1},
     $ P^*_{\tau, t} $ is not the dual of
     $P_{\tau, t} $
     (see \cite{FWang2018})  in the sense that
     \be\label{intr3}
     \int_{L^2(\R^n)}
     P_{\tau, t}\phi (\xi)\
     d\mu (\xi)
     \neq   \int_{L^2(\R^n)}
     \phi (\xi)  \
     d P^*_{\tau, t} (\xi),
    \ee
     where  $\mu\in {\mathcal{P}}_4   ( L^2(\R^n))$
  and  $\phi:  L^2(\R^n) \to \R$ is a bounded Borel function.

To prove the existence of   pullback measure attractors
for \eqref{sde1}-\eqref{sde2}, we  first  need to show
$ \{P^*_{\tau, t}\}_{\tau \le t} $
is a continuous non-autonomous dynamical system
on $\left (
{\mathcal{P}}_4   ( L^2(\R^n)), \
d_{\mathcal{P}(L^2(\R^n))}
\right  )$.
In general,
if  a    stochastic equation does not depend
on distributions of the solutions,
then the continuity
of
 $\{P^*_{\tau, t}\}_{\tau \le t} $
 follows from the  Feller property
 of  $\{P_{\tau, t}\}_{\tau \le t} $
 and the duality relation between
$P^*_{\tau, t}  $ and $P_{\tau, t}  $.
However, this method does not apply
to the
McKean-Vlasov stochastic
     reaction-diffusion equation  \eqref{sde1},
     because
$P^*_{\tau, t}  $ is no longer  the dual
of $P_{\tau, t}  $ as demonstrated by
\eqref{intr3}.
It seems that it is very  difficult to
prove the continuity
$\{P^*_{\tau, t}\}_{\tau \le t} $
for \eqref{sde1}
on  the space $\left (
{\mathcal{P}}_4   ( L^2(\R^n)), \
d_{\mathcal{P}(L^2(\R^n))}
\right  )$.

In order to circumvent this obstacle,
we take the advantage of the regularity
of ${\mathcal{P}}_4   ( L^2(\R^n))$
and apply the Vitali theorem to prove
the continuity of
 $\{P^*_{\tau, t}\}_{\tau \le t} $
on the   subspace
$\left (
B_ {
 \mathcal{P}_4   ( L^2(\R^n))
 }  (r)
, \
d_{\mathcal{P}(L^2(\R^n))}
\right  )$
  instead of the entire space
$\left (
{\mathcal{P}}_4   ( L^2(\R^n)), \
d_{\mathcal{P}(L^2(\R^n))}
\right  )$ (see Lemma \ref{ma1}).
It is fortunate that the
continuity of    $\{P^*_{\tau, t}\}_{\tau \le t} $
on the   subspace
$\left (
B_ {
 \mathcal{P}_4   ( L^2(\R^n))
 }  (r)
, \
d_{\mathcal{P}(L^2(\R^n))}
\right  )$
is sufficient for the existence of pullback
measure attractors for \eqref{sde1}-\eqref{sde2}
provided $\{P^*_{\tau, t}\}_{\tau \le t} $
has a closed pullback absorbing set and is pullback
asymptotically compact.

Since the stochastic equation \eqref{sde1}
is defined on the entire space
$\R^n$, and the Sobolev embeddings are not
compact on unbounded domains,
the pullback asymptotic compactness
of
$\{P^*_{\tau, t}\}_{\tau \le t} $
on
$\left (
{\mathcal{P}}_4   ( L^2(\R^n)), \
d_{\mathcal{P}(L^2(\R^n))}
\right  )$
does not follow from the uniform estimates
of solutions and the  Sobolev embeddings.
This is  another  challenge to establish
the existence of measure attractors for
\eqref{sde1}-\eqref{sde2} on $\R^n$.
We will use the argument of uniform
tail-ends estimates of solutions  to circumvent the
non-compactness of Sobolev embeddings on
$\R^n$
as in the deterministic case
\cite{Wang99} by showing the
mean square of the solutions of \eqref{sde1}-\eqref{sde2}
is
uniformly small outside a sufficiently large ball
 (see Lemma \ref{est4}).
By the uniform smallness of the tails of solutions
and the regularity of solutions, we then obtain
the pullback asymptotic compactness
of
$\{P^*_{\tau, t}\}_{\tau \le t} $
on
$\left (
{\mathcal{P}}_4   ( L^2(\R^n)), \
d_{\mathcal{P}(L^2(\R^n))}
\right  )$ (see Lemma \ref{ma3}),
and hence the existence and uniqueness of
pullback measure attractors of
$\{P^*_{\tau, t}\}_{\tau \le t} $.
Furthermore, if the functions $f$, $g$,
$\theta_k$  and $\sigma_k$
are all periodic functions in time
with period $T>0$, then   the
pullback measure attractor
is also periodic in time with period $T$
(see
 Theorem \ref{main_e}).

 Under further dissipative conditions,
 we prove that any two solutions of
 \eqref{sde1}-\eqref{sde2}
 pullback converge to each other, and hence
 the pullback measure attractor
 is a singleton in this case.
 By the invariance of the pullback
 measure attractor, we infer that
 the stochastic equation \eqref{sde1}
 has a unique periodic probability measure
 if
all the functions $f$, $g$,
$\theta_k$  and $\sigma_k$
are $T$-periodic in time.
Analogously, \eqref{sde1}
 has a unique invariant  probability measure
 if
all the functions $f$, $g$,
$\theta_k$  and $\sigma_k$
are time independent  (see Theorem
\ref{main_s}).

The last goal of this paper is to
investigate the limiting behavior  of
pullback measure attractors,
periodic measures and invariant measures
when
 the
McKean-Vlasov stochastic
      equation \eqref{sde1}
      converges to  a stochastic equation
      independent of distributions of solutions.
      More precisely,
      for every $\varepsilon \in (0,1)$,
      consider
      the
McKean-Vlasov stochastic
      equation in $\R^n$:
$$
    d u^\eps (t,x)  -\Delta u^\eps(t,x) dt
    +\lambda u^\eps(t,x) dt
             + f^\eps(t, x, u^\eps (t,x), \mathcal{L}_{u^\eps (t)} )
             $$
             \be  \label{sdep1}
             = g^\eps(t,x,u^\eps (t,x), \mathcal {L}_{u^\eps (t)} )dt
                +
                  \sum_{k=1}^{\infty}
                          \left( \theta_{k}(t,x) + \kappa(x)
                          \sigma^\eps_{k}(t, u^\eps (t,x), \mathcal {L}_{u^\eps (t)})
                          \right) dW_k(t),
                          \quad t>\tau.
\ee
Under certain conditions,
we   prove that, as $\varepsilon \to 0$,
 the solutions of
\eqref{sdep1}
converge to that of  the following equation:
 $$
    d u  (t,x)  -\Delta u (t,x) dt
    +\lambda u (t,x) dt
             + f (t, x, u  (t,x)  )
             $$
             \be  \label{sdel1}
             = g (t,x,u  (t,x)  )dt
                +
                  \sum_{k=1}^{\infty}
                          \left( \theta_{k}(t,x) + \kappa(x)
                          \sigma _{k}(t, u  (t,x) )
                          \right) dW_k(t),
                          \quad t>\tau.
\ee
Note that \eqref{sdel1}
does not depend on the laws of solutions.
Based on the convergence of solutions
of \eqref{sdep1}, we then
prove that
the pullback measure attractors
of \eqref{sdep1}
converge to that of \eqref{sdel1}
as $\varepsilon\to 0$
in terms of the Hausdorff semi-distance
(see Theorem \ref{main_up}).
In the special case where
the  pullback measure attractor
of \eqref{sdep1} is a singleton
for every $\varepsilon \in (0,1)$,
we infer that
the unique periodic  (invariant) probability measure
of \eqref{sdep1} converges to that
of \eqref{sdel1} when all the functions
in \eqref{sdep1}
and \eqref{sdel1}
are periodic  (independent)  in time
(see Corollary \ref{ip_con}).

This paper is organized as follows.
In section 2, we recall  a well-known result
on the existence and uniqueness
of pullback measure attractors.
 In section 3, we
 discuss the well-posedness
 of \eqref{sde1}-\eqref{sde2}.
 Section 4 is devoted to the uniform estimates
 of solutions including the uniform tail-ends estimates.
 In Section 5, we prove the existence and uniqueness
 of pullback measure attractors of \eqref{sde1}-\eqref{sde2}.
 In Section 6, we show the existence and uniqueness
 of periodic measures and invariant measures
 of \eqref{sde1}-\eqref{sde2}.
 In the last section, we
 prove the upper semi-continuity of pullback measure
 attractors of \eqref{sdep1}  as well as
  the convergence of
 periodic measures and invariant measures
 of \eqref{sdep1} as $\varepsilon \to 0$.

   For convenience,
   throughout this paper,
   we  denote the space
   $L^2(\R^n)$ by $H$ with
   norm $\| \cdot \|$ and inner product
   $(\cdot, \cdot)$.

\section{Preliminaries}

In this section,
for the reader's convenience,
we review basic results on
   pullback measure attractors
   for   dynamical systems defined in
   probability spaces

    Let $X$ be        a
    separable Banach space with norm  $\|\cdot\|_X$
    and
     Borel $\sigma$-algebra  $\mathcal B(X)$.
       Denote by
$C_b(X)$    the space of bounded continuous functions
on $X$.
Let  $L_b(X)$
be the space of bounded Lipschitz
continuous functions on $X$
with norm $\|\cdot\|_{L_b}$
$$
\|\xi \|_{L_b}
=\sup_{x\in X} |\xi (x)|
+    \mathop {\sup }\limits_{x_1 ,x_2  \in X,x_1\neq x_2} \frac{{\left| {\xi \left( {x_1 } \right)}  -
{\xi \left( {x_2 } \right)} \right|}}{{\|x_1-x_2\|_X}} .
$$

Let   $\mathcal P (X)$
be  the space  of probability measures on $(X,\mathcal B(X))$.
If $\mu, \mu_n \in
  \mathcal P(X)$ and for every
  $\xi \in   C_b(X)$,
  $\int_X \xi (x)  \mu_n (dx)
  \to \int_X \xi (x) \mu (dx), $
 then we say $\{\mu_n\}_{n=1}^\infty$
 converges to $\mu$ weakly.
 The weak topology of
 $\mathcal P(X)$
 is metrizable with metric given by
 $$
  d_{\mathcal{P} (X)}
  (\mu_1, \mu_2)
  =\sup_{\xi \in L_b(X), \ \|\xi \|_{L_b} \le 1}
  \left |\int_X \xi d\mu_1 -\int_X \xi d\mu_2 \right |,
 \ \forall \  \mu_1,\mu_2\in \mathcal P(X).
$$
Recall that
    $(\mathcal P(X),  d_{\mathcal{P} (X)})$
    is a polish space.
For every $p\ge 1$, let
$(\mathcal P_p(X), \mathbb{W} _p )$ be the
polish
space   given by
$$
\mathcal P_p \left( X \right) = \left\{ {\mu  \in \mathcal
P\left( X \right):
\int_X {\|x\|_X^p \mu \left( {dx} \right) <    \infty } } \right\},
$$
and
$$
\mathbb{W} _p ( \mu , \nu ) =
\inf\limits_{ \pi \in \Pi ( \mu, \nu ) }
\Big (
\int_{ X \times X}
    \| x-y  \|_X^p \pi (dx, dy)
\Big )^{ \frac{1}{p} },\quad
\forall \mu, \nu \in \mathcal{P}_p ( X ),
$$
where $ \Pi ( \mu, \nu ) $ is the  set of all
couplings of $\mu$ and $\nu$.
The metric $\mathbb{W} _p$
is called the
Wasserstein distance.

   Given  $r>0$, denote by
$$
 B_ {\mathcal P_p(X)} (r)  = \left\{ {\mu  \in \mathcal P_p \left( X \right)
 	:
 	\left(\int_X \|x\|_X^p \mu \left( {dx} \right)\right)^{\frac{1}{p}} \leq r } \right\}.
$$
 A  subset $S\subseteq {\mathcal P}_p \left( X \right)$ is
  bounded  if there is $r>0$ such that
  $S\subseteq B_
{{\mathcal P}_p \left( X \right)
}
(r)$. If   $S$ is bounded in   $ {\mathcal P}_p ( X )$,
then we set
$$
\|S\|_{{\mathcal P}_p ( X )}
=
\sup_{
\mu \in S
} \left(\int_X \|x\|_X^p \mu
 ( {dx}  )\right)^{\frac{1}{p}}.
$$

    Note that $(\mathcal P_p(X), \mathbb{W} _p )$
    is a polish space, but
    $(\mathcal {P}_p (X), d_{\mathcal{P}(X)})$
    is not complete. Since
    for every
    $r>0$,
    $B_ {\mathcal P_p(X)} (r)$ is a closed subset
    of  $\mathcal{P}(X)$ with respect to the
     metric  $d_{\mathcal{P}(X)}$,
     we know that
     the space
    $(B_ {\mathcal P_p(X)} (r), \ d_{\mathcal{P}(X)})$
    is complete for every $r>0$.

    We will consider a non-autonomous dynamical
    system defined in the metric
    space
    $(\mathcal {P}_p (X), d_{\mathcal{P}(X)})$.

\begin{defn}
A family  $\Phi=\{\Phi(t,\tau): t\in \R^+,\,\tau\in \R\}$ of mappings from
$(\mathcal {P}_p (X), d_{\mathcal{P}(X)})$
 to $(\mathcal {P}_p (X), d_{\mathcal{P}(X)})$
 is called  a  non-autonomous dynamical system  on
 $(\mathcal {P}_p (X), d_{\mathcal{P}(X)})$
 if  $\Phi$ satisfies: for all $\tau\in \R$ and $t, s\in \R^+$:

(i)  $\Phi(0,\tau)=I$,  where $I $ is the identity operator on
  $(\mathcal {P}_p (X), d_{\mathcal{P}(X)})$;

(ii) $\Phi(t+s,\tau)=\Phi(t,s+\tau)\circ \Phi(s,\tau)$.

If, in addition, there exists $T>0$ such that
$\Phi (t, \tau +T)=\Phi (t, \tau)$
for all $t\in \R^+$   and $\tau \in \R$,
then we say $\Phi$ is periodic with period $T$.
\end{defn}

If
$\Phi$ is
 a  non-autonomous dynamical system  on
 $(\mathcal {P}_p (X), d_{\mathcal{P}(X)})$,
 and  for every bounded subset $D$
of  $\mathcal {P}_p (X)$,
$\Phi$
is continuous from
$ (D,  d_{\mathcal{P}(X)})$
to
$(\mathcal {P}_p (X), d_{\mathcal{P}(X)})$,
then we say $\Phi$ is continuous on
bounded subsets of
 $\mathcal {P}_p (X)$.

      Since $(\mathcal {P}_p (X), d_{\mathcal{P}(X)})$
      is a metric space,  by
    Theorem 2.23 (more specifically Proposition
    3.6)  and Theorem 2.25
     in \cite{wang12}, we have the
  following
   criterion on the existence and uniqueness   of
   pullback measure attractors.

\begin{prop}\label{exatt}
	 Let $\mathcal D$ be an
	  inclusion-closed collection of
 families of  nonempty bounded subsets of
 $ \mathcal {P}_p (X),$
 and
  $\Phi$
 be  a  non-autonomous dynamical system
 on
 $(\mathcal {P}_p (X), d_{\mathcal{P}(X)})$.
 Suppose $\Phi$ is continuous on
bounded subsets of
 $\mathcal {P}_p (X)$.
  If  $\Phi$ has a
 closed  $\mathcal D$-pullback  absorbing
set $K\in \mathcal D$ and   is
$\mathcal D$-pullback asymptotically
 compact in
 $(\mathcal {P}_p (X), d_{\mathcal{P}(X)})$,
 then
$ \Phi$   has a unique  $\mathcal D$-pullback
  measure attractor $\mathcal A \in   \mathcal D$
  in $(\mathcal {P}_p (X), d_{\mathcal{P}(X)})$.

  Furthermore,
  If $\Phi$  and $K$  are periodic with periodic $T$,
  then so is the measure attractor $\mathcal{A}$.
 \end{prop}

       \begin{rem}
           Note  that
            the dynamical system
            $\Phi$ in
       \cite{wang12} is required to be
       continuous.
       However, by examining the proof
       of Theorem 2.23
      and Theorem 2.25
      (especially Lemma 2.17 with $D$ replaced by
      a closed $\mathcal{D}$-pullback
      absorbing K)
     in \cite{wang12},    we find that
     the main results  of
     \cite{wang12}
     are still valid
       for $\Phi$ on
         $(\mathcal {P}_p (X), d_{\mathcal{P}(X)})$
         provided
   $\Phi$ is continuous on
bounded subsets of
 $\mathcal {P}_p (X)$
 and  $\mathcal D$ is
	  a collection of
 families of  nonempty bounded subsets of
 $ \mathcal {P}_p (X).$
 \end{rem}

\section{Existence and uniqueness
of solutions}

Throughout this paper,
we use $\delta_0$
for the Dirac probability measure at $0$,
and we assume
$f: {\mathbb{R}} \times {\mathbb{R}}^n
\times {\mathbb{R}}
\times \mathcal{P}_2(H)
\to {\mathbb{R}}$ is
  continuous
  and differentiable with respect to
  the second and third arguments,   which further
satisfies the   conditions:

  \noindent
 $({\bf H1})$. \
  For all
  $t, u, u_1, u_2  \in \mathbb{R}$,
          $x \in \mathbb{R}^n$
           and
          $\mu, \mu_1, \mu_2 \in \mathcal{P}_2(H)$,
  \be\label{f1}
  f(t,x, 0, \delta_0)=0,
  \ee
\begin{align}\label{f2}
   f(t,x,u,\mu) u \geq \alpha_1 |u|^p
   - \phi_1(t,x)(1 + |u|^2 )
   - \psi_1 (x)   \mu( \|\cdot\|^2),
\end{align}
\begin{align}\label{f3}
   | f(t,x,u_1, \mu_1) - f(t,x,u_2, \mu_2) |
   \leq & \alpha_2 \left( \phi_2(t,x) + |u_1|^{p-2} + |u_2|^{p-2} \right) |u_1 - u_2|  \nonumber \\
        & + \phi_3(t,x) \mathbb{W}_2 (\mu_1, \mu_2) ),
\end{align}
\begin{align}\label{f4}
   \frac{\partial f}{\partial u}
    (t,x,u,\mu)
   \geq - \phi_4 (t,x),
\end{align}
\begin{align}\label{f5}
  \left |  \frac{\partial f}{\partial x}
    (t,x,u,\mu)
    \right |
   \leq   \phi_5 (t,x)
   \left (1+ |u| +
   \sqrt{ \mu (\|\cdot \|^2} )
   \right ),
\end{align}
 where
 $p\ge 2$,  $\alpha_1>0$,
 $\alpha_2>0$, $\psi_1\in L^1(\R^n)
 \cap L^\infty (\R^n)$  and
       $\phi_i \in
        L ^\infty( \mathbb{R},
        L^\infty(\mathbb{R}^n)
                    \cap  L^1(\mathbb{R}^n) ) $
                   for
       $i = 1, 2,3,4,5$.

It follows from \eqref{f1} and \eqref{f3} that
for all
  $t, u   \in \mathbb{R}$,
          $x \in \mathbb{R}^n$ and
          $\mu  \in \mathcal{P}_2(H)$,
 \begin{align}\label{f6}
   | f(t,x,u, \mu) |
   \leq \alpha_3 |u|^{p-1} + \phi_6(t,x)
   \left( 1 + \sqrt{ \mu( \|\cdot\|^2) } \right),
\end{align}
for some
$\alpha_3>0$
 and
       $\phi_6\in
        L ^\infty( \mathbb{R},
        L^\infty(\mathbb{R}^n)
                   \cap  L^1(\mathbb{R}^n) )$.

For the nonlinear term $g$,
we assume
$ g: \mathbb{R} \times \mathbb{R}^n \times \mathbb{R} \times \mathcal {P}_2(H)
\rightarrow \mathbb{R}$
is continuous
and differentiable with respect to
  the second and third arguments
  such that:

  \noindent
 $({\bf  H2})$. \
For all $t, u, u_1, u_2 \in \mathbb{R},
  x \in \mathbb{R}^n$
and $\mu, \mu_1, \mu_2 \in \mathcal{P}_2(H)$,
 \be \label{g3}
     |g(t,x, u, \mu)|
     \le  \phi_g(
  t,x )
 +\phi_7(t,x)
  |u| + \psi_g (x) \sqrt{ \mu( \|\cdot\|^2) },
\ee
 \be\label{g2}
     |g(t,x, u_1, \mu_1) - g(t,x, u_2, \mu_2)|
     \leq \phi_7 (t,x)
          \left( | u_1 - u_2 | + \mathbb{W}_2(\mu_1, \mu_2)
          \right),
\ee
\be\label{g3a}
\left | {\frac {\partial g}{\partial x}}
(t,x, u, \mu) \right |
\le \phi_8 (t,x) +\phi_7(t,x)
\left ( |u| + \sqrt{\mu (\| \cdot \|^2)}
\right ),
\ee
  where
   $\phi_7  \in L^\infty (\mathbb{R}, L^\infty (\mathbb{R}^n)
                    \cap L^1 (\mathbb{R}^n) )$,
                     $\psi_g  \in   L^\infty (\mathbb{R}^n)
                    \cap L^1 (\mathbb{R}^n )$
                    and
   $\phi_g, \phi_8
  \in L^2_{loc}
  (\R, L^2(\R^n))$.

   By    \eqref{g2}
   we find that
   for all
    $t, u  \in \mathbb{R},
   x \in \mathbb{R}^n$
   and $\mu  \in \mathcal{P}_2(H)$,
     \be\label{g4}
\left | {\frac {\partial g}{\partial u}}
(t,x, u, \mu) \right |
\le  \phi_7(t,x).
\ee

For the
  diffusion terms
  $\{ \theta_{k} \}_{k=1}^\infty$
  and
  $\{ \sigma_{k} \}_{k=1}^\infty$
  we assume the following conditions.

  \noindent
 $({\bf H3})$. \
 The function
   $\theta = \{ \theta_{k} \}_{k=1}^\infty : \mathbb{R} \rightarrow L^2(\mathbb{R}^n, l^2)$ is continuous
   which further satisfies
   \be\label{theta1}
   \sum_{k=1}^\infty
    \int_t^{t+1}
   \|\nabla\theta_k  (s) \|^2ds
   <\infty,
   \quad \forall \  t\in \R.
   \ee

   \noindent
 $({\bf H4})$. \
   For every $k\in \mathbb{N}$,
   $\sigma_{k}:
    \mathbb{R}   \times \mathbb{R} \times \mathcal {P}_2(H)
   \rightarrow \mathbb{R}$
   is continuous
   such that for all  $t, u
    \in \mathbb{R}$ and $\mu \in \mathcal {P}_2(H)$,
   \begin{align}\label{s1}
       | \sigma_{k}(t,u,\mu ) | \leq \beta_{k} \left( 1 + \sqrt{ \mu( \|\cdot\|^2) } \right) + \gamma_{k} |u|,
   \end{align}
   where
   $\beta=\{\beta_{k}\}_{k=1}^\infty$ and $\gamma=\{\gamma_{k}\}_{k=1}^\infty$
   are    nonnegative sequences
   with
     $ \sum_{k=1}^\infty
                  \left( \beta_{k}^2 + \gamma_{k}^2
                  \right)
        < \infty$.
   Furthermore, we assume
     $\sigma_{k}(t,u,\mu)$ is
     differentiable in $u$ and  Lipschitz continuous in
     both
   $u$ and $\mu$
   uniformly for $t \in \mathbb{R}$
   in the sense that
   for all $t, u_1, u_2 \in \mathbb{R}$ and $\mu_1, \mu_2 \in \mathcal {P}_2(H)$,
  \begin{align} \label{s2}
      | \sigma_{k}(t, u_1, \mu_1) - \sigma_{k}(t, u_2, \mu_2) |
      \leq L_{\sigma,k} \left( | u_1 - u_2 | + \mathbb{W}_2 (\mu_1, \mu_2) \right),
  \end{align}
  where   $L_\sigma=
  \{L_{\sigma,k}\}_{k=1}^\infty
   $ is a sequence of nonnegative numbers
      such that  $ \sum_{k=1}^\infty
                   L_{\sigma,k}^2
         < \infty$.

It follows from \eqref{s2}
that
   for all $t, u  \in \mathbb{R}$ and $\mu
    \in \mathcal {P}_2(H)$,
  \begin{align} \label{s3}
     \left  |{\frac { \partial \sigma_{k}}
      {\partial u}}
      (t, u , \mu )
      \right   |
      \leq L_{\sigma,k} .
  \end{align}

Denote by
     $l^2$
      the  space of square summable sequences of real numbers.
For every
 $t \in \mathbb{R} $, $u \in H$ and $\mu \in \mathcal {P}_2(H)$,
 define  a map
 $\sigma(t,u,\mu): l^2 \rightarrow H$ by
\begin{align}\label{sidef}
  \sigma (t,u,\mu)(\eta)(x)
   = \sum_{k=1}^{\infty}
                       \left( \theta_k  (t,x)
                              + \kappa(x) \sigma_{ k} (t, u(x), \mu)
                       \right) \eta_k,\ \
                                 \forall\ \eta =
                                 \{ \eta_k\}_{k=1}^{\infty}\in l^2,
                                           x \in \mathbb{R}^n.
\end{align}
 Let  $L_2(l^2, H)$
be the space of Hilbert-Schmidt operators from
$l^2$ to $H$ with norm $\| \cdot \|_{L_2(l^2, H)}$.
Then
by \eqref{theta1} and \eqref{s1}
we infer that
the operator $\sigma(t,u,\mu)$
belongs to   $L_2(l^2, H)$
 with norm:
\begin{align} \label{s4}
        & \|\sigma(t,u,\mu)\|_{L_2(l^2, H)}^2
      =   \sum_{k=1}^\infty
      \int_{\R^n}
      | \theta_k  (t,x) + \kappa(x) \sigma_{k} (t, u(x), \mu ) |^2
      dx
           \nonumber \\
    \leq &
    2 \sum_{k=1}^\infty
    \| \theta_k  (t) \|^2
                 + 8 \| \kappa \|^2 \| \beta \|_{l^2}^2
                  \left( 1  + \mu( \|\cdot\|^2) \right)
                 + 4 \| \kappa \|_{L^\infty(\mathbb{R}^n)}^2 \| \gamma\|_{l^2}^2
   \| u \|^2  .
\end{align}
  Moreover, by \eqref{s2} we see that
     for all $t \in \mathbb{R}$, $u_1, u_2 \in H$ and $\mu_1, \mu_2 \in \mathcal {P}_2(H)$,
$$
    \| \sigma (t, u_1, \mu_1) -\sigma (t, u_2, \mu_2)\|^2 _{L_2(l^2, H)}
   =   \sum_{k=1}^\infty
          \int_{\R^n}
| \kappa(x)|^2
                      \left |  \sigma_{k} (t, u_1(x) , \mu_1)
                              - \sigma_{k} (t, u_2(x), \mu_2)
                      \right |^2 dx
      $$
      \be\label{s5}
     \le   2  \|
        L_{\sigma}  \|^2_{l^2}
                    \left (  \| \kappa \|_{L^\infty(\mathbb{R}^n)}^2
                    \| u_1 - u_2 \|^2
                           +
                           \| \kappa\|^2
                            \mathbb{W}^2_2 (\mu_1, \mu_2)
                      \right ).
    \ee

We now reformulate
 problem  \eqref{sde1}-\eqref{sde2}
as follows: for $t>\tau$,
\be \label{sde3}
       d u(t) -\Delta  u (t) dt
       +\lambda u(t) dt
         + f(t, \cdot,  u  (t),
          \mathcal{L}_{u (t)} ) dt
   =   g(t, \cdot, u (t),
   \mathcal{L}_{u  (t)} ) dt
       +
                  \sigma(t, u  (t),
                  \mathcal{L}_{u  (t)} ) dW(t),
\ee
with initial data
\be  \label{sde4}
    u (\tau) = u_\tau.
\ee

Under conditions
  ({\bf H1})-({\bf H4}),
 by the arguments
  of \cite{CW2024},
  one can verify that
  for every  $u_\tau \in L^2(\Omega, \mathcal{F}_\tau; H)$,
  system \eqref{sde3}-\eqref{sde4} has a unique solution
         $u $ defined
         on $[\tau, \infty)$
  in the sense that
  $u$ is
  a continuous $H$-valued $\mathcal{F}_t$-adapted
   process  such that
   for all $T>0$,
 $$ u
    \in L^2(\Omega, C([\tau, \tau + T], H) )
        \bigcap
        L^2(\Omega, L^2(\tau, \tau +T; H^1(\R^n) ))
        \bigcap
        L^p(\Omega, L^p(\tau,
        \tau +T; L^p(\mathbb{R}^n ) ) ),
 $$
  and for all $t\ge \tau$,  $\mathbb{P}$-almost surely,
   $$ u  (t)
      - \int_0^t  \Delta  u (s) ds
      +\lambda
      \int_0^t u (s) ds
       + \int_0^t f( s, \cdot, u  (s), \mathcal{L}_{u  (s)} )ds
        $$
        $$
  =   u_\tau
      + \int_0^t g(s, \cdot, u  (s), \mathcal{L}_{u  (s)}) ds
      +   \int_0^t \sigma (s, u (s), \mathcal{L}_{u (s)} ) dW(s),
$$
 in
 $\left(  H^1(\R^n)  \bigcap L^p(\mathbb{R}^n) \right)^\ast$.
 In addition, the following uniform estimates are valid:
\be\label{apri_est}
           \| u \|_{L^2(\Omega, C([\tau,
           \tau +T], H) )}^2
           + \| u  \|_{L^2(\Omega, L^2( \tau,
           \tau +T;  H^1(\R^n) ) )}^2
     \leq  C
              \left( 1 + \| u_\tau \|_{L^2(\Omega,  H)}^2
              \right),
\ee
 where $C=C(\tau, T) > 0$ is
 a constant  independent of $u_\tau$.

In the sequel, we also assume that
the coefficient $\lambda$ is sufficiently
large
such that
$$
    \lambda
    >  12 \|\kappa\|^2 \| \beta \|^2_{l^2}
         +
        6 \|\kappa\|^2_{L^\infty(\R^n)}
          \| \gamma \|^2_{l^2}
               +
            \| \phi_1  \|_{L^\infty
            (\R,
            L^\infty (\R^n)  )}
            +  \|\psi_1\|_{L^1(\R^n)}
            $$
     \be\label{lamc}
            +{\frac 12} \|\psi_g\|_{L^1(\R^n)\cap
            L^\infty(\R^n)}
            +
            \| \phi_7  \|_{L^\infty
            (\R,
            L^\infty  (\R^n) \cap L^1(\R^n)
            )} .
          \ee
     It follows from \eqref{lamc} that
         there exists
         a sufficiently small number $\eta\in (0,1)$
         such that
$$
    2 \lambda - 3\eta
    >   24 \|\kappa\|^2 \| \beta \|^2_{l^2}
         +
        12 \|\kappa\|^2_{L^\infty(\R^n)}
          \| \gamma \|^2_{l^2}
               +
           2  \| \phi_1  \|_{L^\infty
            (\R,
            L^\infty (\R^n)  )}
            + 2 \|\psi_1\|_{L^1(\R^n)}
            $$
     \be\label{etac}
            +\|\psi_g\|_{L^1(\R^n)\cap
            L^\infty(\R^n)}
            +2
            \| \phi_7  \|_{L^\infty
            (\R,
            L^\infty  (\R^n) \cap L^1(\R^n)
            )}
         .
          \ee

Let $ D_1 = \{  D(\tau):
\tau\in\mathbb{R}, \
              D(\tau) \ \text{is
              a bounded
              nonempty subset of  } \ \mathcal{P}_2 ( H )
               \}$,
               and
               $ D_2 = \{  D(\tau):
\tau\in\mathbb{R}, \
              D(\tau) \ \text{is
              a bounded
              nonempty subset of  } \ \mathcal{P}_4 ( H )
               \}$,
  Denote by $ \mathcal{D}_0 $ the collection of all
such  families  $D_1$   which further satisfy:
            $$     \lim\limits_{ \tau \rightarrow - \infty }
                     e^{ \eta  \tau }
                     \| D ( \tau )  \|_{ \mathcal{P}_2 (H  ) }^2
                     =0,
 $$
where $ \eta $
is the same number as in \eqref{etac}.
Similarly denote by
  $ \mathcal{D} $ the collection of all
such  families  $D_2$   which further satisfy:
            $$     \lim\limits_{ \tau \rightarrow - \infty }
                     e^{ 2\eta  \tau }
                     \| D ( \tau )  \|_{ \mathcal{P}_4 (H  ) }^4
                     =0.
 $$
It is evident
$ \mathcal{D} \subseteq   \mathcal{D}_0 $.

We also assume  the following condition
 in order to deriving uniform estimates of solutions:
\begin{equation}\label{phi_g}
\int_{-\infty}^\tau
     e^{  \eta  s}
     \left( \| \theta ( s) \|^2_{L^2(\R^n,
     l^2) }
           +  \|    \phi_g ( s)    \| ^2
     \right)  ds <  \infty, \ \ \forall \tau \in \mathbb{R},
\end{equation}
and
\begin{equation}\label{phi_ga}
\int_{-\infty}^\tau
     e^{  2\eta  s}
     \left( \| \theta ( s) \|^4_{L^2(\R^n,
     l^2) }
           +  \|    \phi_g ( s)    \| ^4
     \right)  ds <  \infty, \ \ \forall \tau \in \mathbb{R}.
\end{equation}

\section{Uniform Estimates of solutions}

This section is devoted to
the  uniform estimates of the solutions of problem
\eqref{sde3}-\eqref{sde4} which are needed to show the
existence and uniqueness of pullback measure attractors.

\begin{lem}\label{est1}
If  {\bf  (H1)}-{\bf (H4)},
\eqref{lamc} and \eqref{phi_g}
hold,
 then for every $ \tau \in \mathbb{R} $
and
$ D = \{  D ( t ) : t \in \mathbb{R}   \}     \in \mathcal{D}_0   $,
there exists
$ T = T ( \tau, D ) \ge 1 $
such that for all $ t \ge T $,
the solution   $u $   of   \eqref{sde3}-\eqref{sde4} satisfies,
for all $r\in [\tau -1, \tau]$,
$$
\E  \left (  \|  u  ( r, \tau - t, u_{\tau -t}  )  \|^2
    \right )
\le   M_1
     + M_1   \int_{ -\infty }^\tau
     e^{  \eta  ( s - \tau ) }
  \left (
     \| \phi_g (s)\|^2     +
       \| \theta (s) \|^2_{L^2(\R^n, l^2)}
       \right )  ds
 $$
and
$$
  \int_{ \tau - t }^\tau  e^{
  \eta  ( s- \tau  ) }
    \E \left (
     \|  u   ( s, \tau - t, u_{ \tau -t }  )
                \|_{ H^1 (\R^n ) } ^2
               + \|  u  ( s, \tau - t, u_{ \tau -t }  )
                 \|_{ L^p ( \R^n ) }^p
        \right ) ds
        $$
         \be\label{est1 2}
 \le
  M_1
     + M_1   \int_{ -\infty }^\tau
     e^{  \eta  ( s - \tau ) }
  \left (
     \| \phi_g (s)\|^2     +
       \| \theta (s) \|^2_{L^2(\R^n, l^2)}
       \right )  ds,
 \ee
where
$ u_{ \tau - t }  \in L^2  ( \Omega, \mathcal{F}_{ \tau - t },
H )  $
with
$ \mathcal{L}_{ u_{ \tau - t } } \in    D ( \tau - t )$,
  $\eta>0$ is the same number as in \eqref{etac},
  and
$ M_1>0 $  is a   constant independent of $ \tau $
 and $D$.
\end{lem}

\begin{proof}
 Given
  $u_\tau \in  L^2  ( \Omega, \mathcal{F}_{\tau},
H )$,
let  $u(t)= u(t, \tau, u_\tau)$ be the solution
of \eqref{sde3}-\eqref{sde4}
with initial data $u_\tau$
at initial time $ \tau$.

By \eqref{sde3} and It\^{o}'s formula, we get for
all  $ t\ge \tau $,
 $$
     e^{\eta t} \| u  (t) \|^2
      + 2 \int_{ \tau} ^t
                 e^{\eta s}  \| \nabla    u (s)
                 \|^2 ds
                 + (2 \lambda-\eta)
                  \int_ { \tau}^t e^{\eta s}
                  \|  u  ( s )  \|^2 ds
                  $$
                  $$
      + 2 \int_ { \tau} ^t\int_{\mathbb{R} ^n}
      e^{\eta s}
        f( s, x, u  (s,x), \mathcal{L}_{u  (s)} ) u  (s,x) dx ds
      $$
      $$
  =   e^{\eta \tau}\| u_ { \tau} \|^2
      + 2 \int_ { \tau}^t
      e^{\eta s}
      (g(s, u  (s), \mathcal{L}_{u (s)} ),  u  (s) ) ds
      +   \int_ { \tau}^t
      e^{\eta s}
      \| \sigma(s, u  (s), \mathcal{L}_{u  (s)} ) \|_{L_2(l^2,H)}^2 ds
      $$
        \be \label{est1 p1}
     + 2   \int_ { \tau}^t e^{\eta s}
          \left( u  (s),
         \sigma(s, u  (s), \mathcal{L}_{u  (s)}) dW(s)
                                      \right),
 \ee
  $    \mathbb{P}$-almost surely.
For every  $m\in \N$, define a stopping time
$\tau_m$ by
$$
\tau_m
=\inf
\{
t\ge \tau:  \| u(t) \| > m
\}.
$$
As usual, $\inf \emptyset =+\infty$.
By \eqref{est1 p1} we have
for all  $ t\ge \tau $,
$$
    \mathbb{E}
    \left (  e^{\eta (t\wedge \tau_m) } \| u  (t\wedge \tau_m ) \|^2
      + 2 \int_{ \tau} ^{t\wedge \tau_m}
                 e^{\eta s}  \| \nabla    u (s)
                 \|^2 ds
     \right )
     $$
     $$=
         \mathbb{E} \left (e^{\eta \tau} \| u_ { \tau} \|^2
         \right )
                 + (\eta - 2 \lambda )
                     \mathbb{E}
                     \left (
                  \int_ { \tau}^ {t\wedge \tau_m} e^{\eta s}
                  \|  u  ( s )  \|^2 ds
                  \right )
                  $$
                  $$
      - 2
       \mathbb{E}
                     \left (
      \int_ { \tau} ^ {t\wedge \tau_m}
      \int_{\mathbb{R} ^n}
      e^{\eta s}
        f( s, x, u  (s,x), \mathcal{L}_{u  (s)} ) u  (s,x) dx ds
        \right )
      $$
      $$
      + 2  \mathbb{E}
                     \left (
                     \int_ { \tau}^ {t\wedge \tau_m}
      e^{\eta s}
      (g(s, u  (s), \mathcal{L}_{u (s)} ),  u  (s) ) ds
      \right )
      $$
              \be \label{est1 p2}
      +   \mathbb{E}
                     \left (
                      \int_ { \tau}^ {t\wedge \tau_m}
      e^{\eta s}
      \| \sigma(s, u  (s), \mathcal{L}_{u  (s)} ) \|_{L_2(l^2,H)}^2 ds
      \right ).
  \ee

Next, we derive
the  uniform estimates
for the terms
on the right-hand side
of \eqref{est1 p2}.
 For the
 third  term on the right-hand side of
\eqref{est1 p2} , by \eqref{f2}, we have
$$
- 2
       \mathbb{E}
                     \left (
      \int_ { \tau} ^ {t\wedge \tau_m}
      \int_{\mathbb{R} ^n}
      e^{\eta s}
        f( s, x, u  (s,x), \mathcal{L}_{u  (s)} ) u  (s,x) dx ds
        \right )
        $$
$$
\le
-2\alpha_1
 \mathbb{E}
      \left (
      \int_ { \tau} ^ {t\wedge \tau_m}
      e^{\eta s}
      \| u(s) \|^p_{L^p(\R^n)} ds
      \right )
      +
      2\mathbb{E}
      \left (
      \int_ { \tau} ^ {t\wedge \tau_m}
      e^{\eta s} \| \phi_1 (s)\|_{L^\infty
      (\R^n)}
      \| u(s) \|^2  ds
      \right )
      $$
      $$
      +2
      \mathbb{E}
      \left (
      \int_ { \tau} ^ {t\wedge \tau_m}
      e^{\eta s} \left (
      \| \phi_1 (s)\|_{L^1
      (\R^n)}
       +
        \| \psi_1 \|_{L^1
      (\R^n)}
       \mathbb{E}
      \left (
      \| u(s) \|^2
      \right )
      \right )  ds
      \right )
      $$
      $$
\le
-2\alpha_1
 \mathbb{E}
      \left (
      \int_ { \tau} ^ {t\wedge \tau_m}
      e^{\eta s}
      \| u(s) \|^p_{L^p(\R^n)} ds
      \right )
      +
      2\mathbb{E}
      \left (
      \int_ { \tau} ^ {t }
      e^{\eta s} \| \phi_1 (s)\|_{L^\infty
      (\R^n)}
      \| u(s) \|^2  ds
      \right )
      $$
      $$
      +2
      \mathbb{E}
      \left (
      \int_ { \tau} ^ {t }
      e^{\eta s} \left (
      \| \phi_1 (s)\|_{L^1
      (\R^n)}
       +
        \| \psi_1 \|_{L^1
      (\R^n)}
       \mathbb{E}
      \left (
      \| u(s) \|^2
      \right )
      \right )  ds
      \right )
      $$
        $$
\le
-2\alpha_1
 \mathbb{E}
      \left (
      \int_ { \tau} ^ {t\wedge \tau_m}
      e^{\eta s}
      \| u(s) \|^p_{L^p(\R^n)} ds
      \right )
      +2
       \int_ { \tau} ^ {t }
      e^{\eta s}
      \| \phi_1 (s)\|_{L^1
      (\R^n)}     ds
      $$
    \be\label{est1 p3}
      +
      2
      \int_ { \tau} ^ {t }
      e^{\eta s}
      \left ( \| \phi_1 (s)\|_{L^\infty
      (\R^n)}
      +\| \psi_1  \|_{L^1
      (\R^n)}
      \right )
      \mathbb{E}
      \left (
      \| u(s) \|^2
      \right )   ds .
   \ee
   For the fourth term on the
   right-hand side of
   \eqref{est1 p2},
   by \eqref{g3} we have
   $$
       2  \mathbb{E}
                     \left (
                     \int_ { \tau}^ {t\wedge \tau_m}
      e^{\eta s}
      (g(s, u  (s), \mathcal{L}_{u (s)} ),  u  (s) ) ds
      \right )
      $$
      $$
      \le
      2  \mathbb{E}
                     \left (
                     \int_ { \tau}^ {t\wedge \tau_m}
      e^{\eta s}
      \int_{\R^n}
      \phi_g (s, x)
       |u(s,x)| dxds
       \right )
       $$
       $$
       +
       2  \mathbb{E}
                     \left (
                     \int_ { \tau}^ {t\wedge \tau_m}
      e^{\eta s}
      \int_{\R^n}
      \left (
      \phi_7 (s, x)
      |u(s,x)|^2
      +
       \psi_g   ( x)
      |u(s,x)| \sqrt{
      \mathbb{E}
                     \left (
                     \| u(s)\|^2
                     \right )
                     }
                     \right )
          dx ds
     \right )
      $$
       $$
      \le
        \mathbb{E}
                     \left (
                     \int_ { \tau}^ {t\wedge \tau_m}
      e^{\eta s}
      \left ( \eta \| u(s)\|^2
      + \eta^{-1} \|\phi_g (s) \|^2
      \right ) ds
      \right )
       $$
 $$
     +
         \mathbb{E}
                     \left (
                     \int_ { \tau}^ {t\wedge \tau_m}
      e^{\eta s}
      \int_{\R^n}
     \left (
      (2| \phi_7 (s,x)|+ |\psi_g(x)| ) |u(s,x)|^2
      +   |\psi_g (x)|
      \mathbb{E}
                     \left (
                     \| u(s)\|^2
                     \right )
                      \right )
          dx ds
     \right )
      $$
       $$
      \le
        \mathbb{E}
                     \left (
                     \int_ { \tau}^ {t\wedge \tau_m}
      e^{\eta s}
      \left ( \eta \| u(s)\|^2
      + \eta^{-1} \|\phi_g (s) \|^2
      \right ) ds
      \right )
       $$
       $$
      +
        \mathbb{E}
                     \left (
                     \int_ { \tau}^ {t\wedge \tau_m}
      e^{\eta s}
     \| \psi_g \|_{L^1(\R^n)}
      \mathbb{E}
                     \left (
                     \| u(s)\|^2
                     \right )
            ds
     \right )
     $$
     $$
      +
        \mathbb{E}
                     \left (
                     \int_ { \tau}^ {t\wedge \tau_m}
      e^{\eta s}
     (2\| \phi_7 (s )\|_{L^\infty (\R^n)}
     +\|\psi_g\|_{L^\infty (\R^n)})
           \| u(s)\|^2
            ds
     \right )
      $$
       $$
      \le
        \mathbb{E}
                     \left (
                     \int_ { \tau}^ {t }
      e^{\eta s}
      \left ( \eta \| u(s)\|^2
      + \eta^{-1} \|\phi_g (s) \|^2
      \right ) ds
      \right )
       $$
       $$
      +
        \mathbb{E}
                     \left (
                     \int_ { \tau}^ {t }
      e^{\eta s}
     \| \psi_g \|_{L^1(\R^n)}
      \mathbb{E}
                     \left (
                     \| u(s)\|^2
                     \right )
            ds
     \right )
     $$
     $$
      +
        \mathbb{E}
                     \left (
                     \int_ { \tau}^ {t }
      e^{\eta s}
     (2\| \phi_7 (s )\|_{L^\infty (\R^n)}
     +\|\psi_g\|_{L^\infty (\R^n)})
           \| u(s)\|^2
            ds
     \right )
      $$
         \be\label{est1 p4}
      \le
               \eta^{-1}      \int_ { \tau}^ {t }
      e^{\eta s}
     \| \phi_g (s)\|^2  ds
     +
               \int_ { \tau}^ {t }
      e^{\eta s}
      \left (\eta +  \| \psi_g \|_{L^1 (\R^n)}
      + 2
     \| \phi_7 (s)\|_{L^\infty (\R^n)}
     +\|\psi_g \|_{L^\infty (\R^n)}
          \right )
        \mathbb{E}
                     \left ( \| u(s)\|^2
                     \right )
            ds .
            \ee
     For the last term on the right-hand side
     of \eqref{est1 p2},
     by \eqref{s4} we have
      $$
         \mathbb{E}
                     \left (
                      \int_ { \tau}^ {t\wedge \tau_m}
      e^{\eta s}
      \| \sigma(s, u  (s), \mathcal{L}_{u  (s)} ) \|_{L_2(l^2,H)}^2 ds
      \right )
      $$
         $$
         \le
          2
          \mathbb{E}
                     \left (
                      \int_ { \tau}^ {t\wedge \tau_m}
      e^{\eta s}
       \| \theta (s) \|^2_{L^2(\R^n, l^2)} ds
       \right )
      + 8 \|\kappa\|^2 \| \beta \|^2_{l^2}
        \mathbb{E}
                     \left (
                      \int_ { \tau}^ {t\wedge \tau_m}
      e^{\eta s}
      \left  (1+  \mathbb{E}
                     \left ( \|u(s) \|^2
                     \right )  \right ) ds \right )
                     $$
      $$
         + 4 \|\kappa\|^2_{L^\infty(\R^n)}
          \| \gamma \|^2_{l^2}
           \mathbb{E}
                     \left (
                      \int_ { \tau}^ {t\wedge \tau_m}
      e^{\eta s}
          \| u(s) \|^2 ds
          \right  )
           $$
        $$
         \le
          2
                      \int_ { \tau}^ {t }
      e^{\eta s}
       \| \theta (s) \|^2_{L^2(\R^n, l^2)} ds
      + 8 \|\kappa\|^2 \| \beta \|^2_{l^2}
                \int_ { \tau}^ {t }
      e^{\eta s}
      \left (1+  \mathbb{E}
                     \left ( \|u(s) \|^2
                     \right )
                     \right ) ds
                     $$
      $$
         + 4 \|\kappa\|^2_{L^\infty(\R^n)}
          \| \gamma \|^2_{l^2}
                      \int_ { \tau}^ {t }
      e^{\eta s}             \mathbb{E}
                     \left (
          \| u(s) \|^2
          \right ) ds
           $$
  $$
         \le
          2
                      \int_ { \tau}^ {t }
      e^{\eta s}
       \| \theta (s) \|^2_{L^2(\R^n, l^2)} ds
      + 8 \|\kappa\|^2 \| \beta \|^2_{l^2}
      \eta^{-1} e^{\eta t}
      $$
    \be\label{est1 p5}
         +
         \left( 8 \|\kappa\|^2 \| \beta \|^2_{l^2}
         +
         4 \|\kappa\|^2_{L^\infty(\R^n)}
          \| \gamma \|^2_{l^2}
          \right )
                      \int_ { \tau}^ {t }
      e^{\eta s}             \mathbb{E}
                     \left (
          \| u(s) \|^2
          \right ) ds .
        \ee
 It follows from
 \eqref{est1 p2}-\eqref{est1 p5}
 that for all $t\ge \tau$,
 $$
    \mathbb{E}
    \left (  e^{\eta (t\wedge \tau_m) } \| u  (t\wedge \tau_m ) \|^2
      + 2 \int_{ \tau} ^{t\wedge \tau_m}
                 e^{\eta s}  \| \nabla    u (s)
                 \|^2 ds
     \right )
     $$
     $$
     +      \mathbb{E}
       \left (
        (  2 \lambda -\eta )
                  \int_ { \tau}^ {t\wedge \tau_m} e^{\eta s}
                  \|  u  ( s )  \|^2 ds
                  +
     2\alpha_1
      \int_ { \tau} ^ {t\wedge \tau_m}
      e^{\eta s}
      \| u(s) \|^p_{L^p(\R^n)} ds
      \right )
     $$
     $$\le
         \mathbb{E} \left (e^{\eta \tau} \| u_ { \tau} \|^2
         \right )
  +2
       \int_ { \tau} ^ {t }
      e^{\eta s}
      \| \phi_1 (s)\|_{L^1
      (\R^n)}     ds
      $$
      $$
      +
        \eta^{-1}      \int_ { \tau}^ {t }
      e^{\eta s}
     \| \phi_g (s)\|^2  ds
             +
          2
                      \int_ { \tau}^ {t }
      e^{\eta s}
       \| \theta (s) \|^2_{L^2(\R^n, l^2)} ds
      + 8 \|\kappa\|^2 \| \beta \|^2_{l^2}
      \eta^{-1} e^{\eta t}
      $$
  $$
      +
      2
      \int_ { \tau} ^ {t }
      e^{\eta s}
      \left ( \| \phi_1 (s)\|_{L^\infty
      (\R^n)}
      +\| \psi_1 \|_{L^1
      (\R^n)}
      \right )
      \mathbb{E}
      \left (
      \| u(s) \|^2
      \right )   ds
 $$
  $$
     +
               \int_ { \tau}^ {t }
      e^{\eta s}
      \left (  \eta + \| \psi_g \|_{L^1 (\R^n)}
      + 2
     \| \phi_7 (s )\|_{L^\infty (\R^n)}
     +\| \psi_g   \|_{L^\infty (\R^n)}
          \right )
        \mathbb{E}
                     \left ( \| u(s)\|^2
                     \right )
            ds
            $$
     \be\label{est1 p6}
         +
         \left( 8 \|\kappa\|^2 \| \beta \|^2_{l^2}
         +
         4 \|\kappa\|^2_{L^\infty(\R^n)}
          \| \gamma \|^2_{l^2}
          \right )
                      \int_ { \tau}^ {t }
      e^{\eta s}             \mathbb{E}
                     \left (
          \| u(s) \|^2
          \right ) ds .
        \ee
      Taking the limit of \eqref{est1 p6}
      as $m \to \infty$, by Fatou's lemma we obtain
      for all $t\ge \tau$,
      $$e^{\eta t}
    \mathbb{E}
    \left (     \| u  (t  ) \|^2  \right )
      + 2 \int_{ \tau} ^{t }
                 e^{\eta s}  \mathbb{E}
    \left ( \| \nabla    u (s)
                 \|^2 ds
     \right )
     $$
     $$
     +
      \eta
                  \int_ { \tau}^ {t } e^{\eta s}
                  \mathbb{E}
       \left ( \|  u  ( s )  \|^2
       \right ) ds
                  +
     2\alpha_1
      \int_ { \tau} ^ {t }
      e^{\eta s} \mathbb{E}
       \left (
      \| u(s) \|^p_{L^p(\R^n)}
      \right ) ds
     $$
     $$\le
         \mathbb{E} \left (e^{\eta \tau} \| u_ { \tau} \|^2
         \right )
  +2
       \int_ { \tau} ^ {t }
      e^{\eta s}
      \| \phi_1 (s)\|_{L^1
      (\R^n)}     ds
      $$
      $$
      +
        \eta^{-1}      \int_ { \tau}^ {t }
      e^{\eta s}
     \| \phi_g (s)\|^2  ds
             +
          2
                      \int_ { \tau}^ {t }
      e^{\eta s}
       \| \theta (s) \|^2_{L^2(\R^n, l^2)} ds
      + 8 \|\kappa\|^2 \| \beta \|^2_{l^2}
      \eta^{-1} e^{\eta t}
      $$
  $$
      +
      2
      \int_ { \tau} ^ {t }
      e^{\eta s}
      \left ( \| \phi_1 (s)\|_{L^\infty
      (\R^n)}
      +\| \psi_1 \|_{L^1
      (\R^n)}
      \right )
      \mathbb{E}
      \left (
      \| u(s) \|^2
      \right )   ds
 $$
  $$
     +
               \int_ { \tau}^ {t }
      e^{\eta s}
      \left (  \eta + \| \psi_g \|_{L^1 (\R^n)}
      + 2
     \| \phi_7 (s )\|_{L^\infty (\R^n)}
     +\| \psi_g   \|_{L^\infty (\R^n)}
          \right )
        \mathbb{E}
                     \left ( \| u(s)\|^2
                     \right )
            ds
            $$
     \be\label{est1 p6a}
         +
         \left( 2\eta  -2\lambda +
         8 \|\kappa\|^2 \| \beta \|^2_{l^2}
         +
         4 \|\kappa\|^2_{L^\infty(\R^n)}
          \| \gamma \|^2_{l^2}
          \right )
                      \int_ { \tau}^ {t }
      e^{\eta s}             \mathbb{E}
                     \left (
          \| u(s) \|^2
          \right ) ds .
        \ee

 By \eqref{etac}
         and \eqref{est1 p6a}
         we get
            for all $t\ge \tau$,
            $$
        \mathbb{E}
       \left ( \| u  (t ,\tau, u_\tau  ) \|^2
       \right )
      + 2 \int_{ \tau} ^{t }
                 e^{\eta (s-t) } \mathbb{E}
       \left (
        \| \nabla    u (s)
                 \|^2
     \right )ds
     $$
     $$
     +\eta
     \int_{ \tau} ^{t }
                 e^{\eta (s-t) } \mathbb{E}
       \left (
        \|     u (s)
                 \|^2
     \right )ds
              +
     2\alpha_1
      \int_ { \tau} ^ {t }
      e^{\eta  (s-t) } \mathbb{E}
       \left (
      \| u(s) \|^p_{L^p(\R^n)}
      \right ) ds
     $$
     $$\le
     e^{\eta (\tau-t) }
         \mathbb{E} \left ( \| u_ { \tau} \|^2
         \right )
  +2
       \int_ { \tau} ^ {t }
      e^{\eta  (s-t)  }
      \| \phi_1 (s)\|_{L^1
      (\R^n)}     ds
      $$
      \be\label{est1 p7}
+  \eta^{-1}      \int_ { \tau}^ {t }
      e^{\eta (s-t) }
     \| \phi_g (s)\|^2  ds
             +
          2
                      \int_ { \tau}^ {t }
      e^{\eta  (s-t) }
       \| \theta (s) \|^2_{L^2(\R^n, l^2)} ds
      + 8 \|\kappa\|^2 \| \beta \|^2_{l^2}
      \eta^{-1}  .
    \ee

   Replacing $\tau$ and $t$
   in \eqref{est1 p7} by $\tau -t$
   and  $r$, respectively ,
         we get
            for all $t\ge 1$
            and $r\in [\tau -1,
            \tau]$,
            $$
        \mathbb{E}
       \left ( \| u  (r ,\tau -t , u_{\tau -t} ) \|^2
       \right )
      + 2 \int_{  \tau -t } ^{r }
                 e^{\eta (s-r) } \mathbb{E}
       \left (
        \| \nabla    u (s)
                 \|^2
     \right )ds
     $$
     $$
     +\eta \int_{  \tau -t } ^{r }
                 e^{\eta (s-r) } \mathbb{E}
       \left (
        \|   u (s)
                 \|^2
     \right )ds
              +
     2\alpha_1
      \int_ { \tau -t} ^ {r }
      e^{\eta  (s-r) } \mathbb{E}
       \left (
      \| u(s) \|^p_{L^p(\R^n)}
      \right ) ds
     $$
     $$\le
     e^{\eta (\tau-t-r) }
         \mathbb{E} \left ( \| u_ { \tau -t} \|^2
         \right )
  +2
       \int_ { \tau -t } ^ {r}
      e^{\eta  (s-r)  }
      \| \phi_1 (s)\|_{L^1
      (\R^n)}     ds
      $$
     $$
      +
        \eta^{-1}      \int_ { \tau -t }^ {r }
      e^{\eta (s-r) }
     \| \phi_g (s)\|^2  ds
             +
          2
                      \int_ { \tau -t }^ {r }
      e^{\eta  (s-r) }
       \| \theta (s) \|^2_{L^2(\R^n, l^2)} ds
      + 8 \|\kappa\|^2 \| \beta \|^2_{l^2}
      \eta^{-1}
    $$
    $$\le
     e^{\eta (\tau-t-r) }
         \mathbb{E} \left ( \| u_ { \tau -t} \|^2
         \right )
  +2\eta^{-1}
      \| \phi_1 \|_{L^\infty(\R, L^1
      (\R^n) ) }
     $$
   \be\label{est1 p8}
      +
        \eta^{-1}      \int_ { -\infty}^ {\tau}
      e^{\eta (s-\tau +1) }
     \| \phi_g (s)\|^2  ds
             +
          2
                      \int_ { -\infty }^ {\tau}
      e^{\eta  (s-\tau +1) }
       \| \theta (s) \|^2_{L^2(\R^n, l^2)} ds
      + 8 \|\kappa\|^2 \| \beta \|^2_{l^2}
      \eta^{-1}  .
      \ee

Note that $ \mathcal{L}_{ u_{\tau-t}  }  \in   D ( \tau-t )$
and $D\in \mathcal{D}_0$,
and hence for  all $r\in [\tau -1, \tau]$,
 $$
\lim_{t \to \infty}
 e^{\eta (\tau-t-r) }
         \mathbb{E} \left ( \| u_ { \tau -t} \|^2
         \right )
           \le  e^{\eta (1-\tau)}
          \lim_{t \to \infty}
 e^{\eta (\tau-t) }\| D(\tau -t)\|^2_{\mathcal{P}_2 ( H ) } =0.
         $$
         Then
  there exists $T=T(\tau, D ) \ge 1$ such that for all $t\ge T$
  and $r\in [\tau -1,
  \tau]$,
$$
e^{\eta (\tau-t-r) }
         \mathbb{E} \left ( \| u_ { \tau -t} \|^2
         \right )   \le 1,
         $$
         which along with
         \eqref{est1 p8}
         implies that
         for all $t\ge T$
  and $r\in [\tau -1,
  \tau]$,
  $$
        \mathbb{E}
       \left ( \| u  (r ,\tau -t , u_{\tau -t} ) \|^2
       \right )
      + 2  \int_{  \tau -t } ^{r }
                 e^{\eta (s-r) } \mathbb{E}
       \left (
        \| \nabla    u (s)
                 \|^2
     \right )ds
     $$
     $$
     +\eta \int_{  \tau -t } ^{r }
                 e^{\eta (s-r) } \mathbb{E}
       \left (
        \|   u (s)
                 \|^2
     \right )ds
              +
     2\alpha_1
      \int_ { \tau -t} ^ {r }
      e^{\eta  (s-r) } \mathbb{E}
       \left (
      \| u(s) \|^p_{L^p(\R^n)}
      \right ) ds
     $$
     $$\le 1
  +2\eta^{-1}
      \| \phi_1 \|_{L^\infty(\R, L^1
      (\R^n) ) }
       +
        \eta^{-1}e^\eta      \int_ { -\infty}^ {\tau}
      e^{\eta (s-\tau) }
     \| \phi_g (s)\|^2  ds
     $$
     $$
             +
          2  e^\eta
                      \int_ { -\infty }^ {\tau}
      e^{\eta  (s-\tau ) }
       \| \theta (s) \|^2_{L^2(\R^n, l^2)} ds
      + 8 \|\kappa\|^2 \| \beta \|^2_{l^2}
      \eta^{-1} .
$$
  This completes the proof.
    \end{proof}

As an immediate consequence of Lemma
\ref{est1}, we have
the following estimates.

\begin{cor}\label{est2}
If  {\bf  (H1)}-{\bf (H4)}, \eqref{lamc} and \eqref{phi_g}  hold,
 then for every $ \tau \in \mathbb{R} $
and
$ D = \{  D ( t ) : t \in \mathbb{R}   \}     \in \mathcal{D}_0   $,
there exists
$ T = T ( \tau, D ) \ge 1 $
such that for all $ t \ge T $,
the solution   $u $   of   \eqref{sde3}-\eqref{sde4} satisfies,
  $$
  \int_{ \tau - 1 }^\tau
    \E \left (
     \|  u   ( s, \tau - t, u_{ \tau -t }  )
                \|_{ H^1 (\R^n ) } ^2
               + \|  u  ( s, \tau - t, u_{ \tau -t }  )
                 \|_{ L^p ( \R^n ) }^p
        \right ) ds
        $$
        $$
 \le
  e^\eta M_1
     + e^\eta M_1   \int_{ -\infty }^\tau
     e^{  \eta  ( s - \tau ) }
  \left (
     \| \phi_g (s)\|^2     +
       \| \theta (s) \|^2_{L^2(\R^n, l^2)}
       \right )  ds,
 $$
where
$ u_{ \tau - t }  \in L^2  ( \Omega, \mathcal{F}_{ \tau - t },
H )  $
with
$ \mathcal{L}_{ u_{ \tau - t } } \in    D ( \tau - t )$,
and
$\eta>0$  and
$ M_1>0 $  are the same   constants as
in \eqref{etac} and  \eqref{est1 2},
respectively.
\end{cor}

Next, we establish the uniform
estimates of solutions in $H^1(\R^n)$.

\begin{lem}\label{est3}
If  {\bf  (H1)}-{\bf (H4)}, \eqref{lamc} and \eqref{phi_g}  hold,
 then for every $ \tau \in \mathbb{R} $
and
$ D = \{  D ( t ) : t \in \mathbb{R}   \}     \in \mathcal{D}_0   $,
there exists
$ T = T ( \tau, D ) \ge 1 $
such that for all $ t \ge T $,
the solution   $u $   of   \eqref{sde3}-\eqref{sde4} satisfies,
  $$
    \E \left (
     \|  u   ( \tau, \tau - t, u_{ \tau -t }  )
                \|_{ H^1 (\R^n ) } ^2
        \right )
        $$
        $$
 \le
   M_2
     +   M_2   \int_{ -\infty }^\tau
     e^{  \eta  ( s - \tau ) }
  \left (
     \| \phi_g (s)\|^2     +
       \| \theta (s) \|^2_{L^2(\R^n, l^2)}
       \right )  ds,
 $$
where
$ u_{ \tau - t }  \in L^2  ( \Omega, \mathcal{F}_{ \tau - t },
H )  $
with
$ \mathcal{L}_{ u_{ \tau - t } } \in    D ( \tau - t )$,
$\eta>0$     is  the same   constant as
in \eqref{etac},
 and $M_2=M_2(\tau)>0$
 depends on $\tau$ but not on $D$.
  \end{lem}

\begin{proof}
We just formally derive the uniform
estimates of solutions, which can be justified by
a limiting process.

By \eqref{sde3} and  Ito's formula, we
get for $ \tau \in \mathbb{R} $,  $ t \ge  1 $ and $ r \in ( \tau-1, \tau ) $,
 $$
  \|  \nabla u   ( \tau, \tau-t, u_{ \tau -t } )   \|^2
  +  2 \int_r^\tau \| \Delta  u  ( s, \tau-t, u_{ \tau -t } )     \|^2    ds
  + 2  \lambda  \int_r ^\tau
\|\nabla u   ( s, \tau-t, u_{ \tau -t } )
\|^2     ds
$$
$$
 =  \|  \nabla u   ( r, \tau-t, u_{ \tau -t } )   \|^2
   - 2 \int_r ^\tau
   \int_{\R^n}
     \nabla f   ( s, x, u   ( s, \tau-t, u_{ \tau -t },
     \mathcal{L}_{u(s)} )
     \cdot \nabla     u   ( s, \tau-t, u_{ \tau -t } )
                     dx  ds
$$
$$
 +2 \int_r ^\tau
   \int_{\R^n}
     \nabla g   ( s, x, u   ( s, \tau-t, u_{ \tau -t },
     \mathcal{L}_{u(s, \tau-t, u_{ \tau -t } )} )
     \cdot \nabla     u   ( s, \tau-t, u_{ \tau -t } )
                     dx  ds
$$
$$
 +  \int_r^ \tau
\| \nabla           \sigma
  (  s, u   ( s, \tau-t, u_{ \tau -t } ), \mathcal{L}_{u(s,
  \tau-t, u_{ \tau -t } )} )
  \|^2_{L_2(l^2, H)} ds
 $$
 \be \label{est3 p1}
  + 2 \int_r^ \tau
\left ( \nabla  u   ( s, \tau-t, u_{ \tau -t } ),
\nabla   \sigma   ( s,  u   ( s, \tau-t, u_{ \tau -t } )
                           , \mathcal{L}_{u
            ( s, \tau-t, u_{ \tau -t } ) }  ) dW(s)
            \right )       .
\ee
For the  second term
on the right-hand side
of  \eqref{est3 p1}, by \eqref{f4} and \eqref{f5}, we have
for
$ r \in ( \tau-1, \tau ) $,
$$
  - 2 \int_r^ \tau
   \int_{\R^n}
     \nabla f   ( s, x, u   ( s, \tau-t, u_{ \tau -t },
     \mathcal{L}_{u(s, \tau-t, u_{ \tau -t } )} )
     \cdot \nabla     u   ( s, \tau-t, u_{ \tau -t } )
                     dx  ds
$$
$$
= - 2 \int_r^ \tau    \int_{\R^n}
            \frac{ \partial f  }{ \partial x }
               ( s, x, u  ( s, \tau-t, u_{ \tau -t } ),
                \mathcal{L}_{u(s, \tau-t, u_{ \tau -t } )}        )
                        \cdot  \nabla u  ( s, \tau-t, u_{ \tau -t } )  dx
                        ds
$$
$$
             - 2 \int_r^ \tau
   \int_{\R^n}
           \frac{ \partial f  }{ \partial u  }
            ( s, x, u  ( s, \tau-t, u_{ \tau -t } ),
             \mathcal{L}_{u(s, \tau-t, u_{ \tau -t } )}
                   )
                         |  \nabla u
                          ( s, \tau-t, u_{ \tau -t } ) |^2  dx
                   ds
$$
$$
\le
2 \int_r^ \tau    \int_{\R^n}
\phi_5 (s,x)
\left (
1+  | u  ( s, \tau-t, u_{ \tau -t } )|
+ \sqrt{
\E (\| u  ( s )\|^2)
}
\right ) |  \nabla u  ( s, \tau-t, u_{ \tau -t } ) | dx
                        ds
$$
$$
           +2 \int_r^ \tau
   \int_{\R^n}
   \phi_4(s,x)   |  \nabla u
                          ( s, \tau-t, u_{ \tau -t } ) |^2  dx
                   ds
$$
$$
\le
2 \int_r^ \tau    \int_{\R^n}
|\phi_5 (s,x)|
|  \nabla u  ( s, \tau-t, u_{ \tau -t } ) | dx
                        ds
      $$
   $$
                        +
                          \int_r^ \tau    \int_{\R^n}
|\phi_5 (s,x)|
\left (
  | u  ( s, \tau-t, u_{ \tau -t } )|^2
+
\E (\| u  ( s )\|^2)
 +2 |  \nabla u  ( s, \tau-t, u_{ \tau -t } ) |^2
 \right )  dx
                        ds
$$
$$
           +2 \int_r^ \tau
   \int_{\R^n}
   \phi_4(s,x)   |  \nabla u
                          ( s, \tau-t, u_{ \tau -t } ) |^2  dx
                   ds
$$
$$
\le
 \int_r^ \tau
\left (
\| \phi_5 (s)\|^2 +
\| \nabla u  ( s, \tau-t, u_{ \tau -t } ) \|^2
\right )
                        ds
      $$
   $$
                        +
                          \int_r^ \tau
\|\phi_5 (s)\|_{L^\infty (\R^n)}
\left (
  \| u  ( s, \tau-t, u_{ \tau -t } )\|^2
  +2 \|  \nabla u  ( s, \tau-t, u_{ \tau -t } ) \|^2
 \right )    ds
$$
$$
 +        \int_r^ \tau
\|\phi_5 (s)\|_{L^1 (\R^n)}
 \E (\| u  ( s )\|^2) ds
 +
            2 \int_r^ \tau
   \| \phi_4(s)   \|_{L^\infty
   (\R^n)}
   \|  \nabla u
                          ( s, \tau-t, u_{ \tau -t } ) \|^2
                   ds
$$
$$
\le
 \left (
 1 + 2 \| \phi_5\|_{L^\infty
 (\R, L^\infty (\R^n))}
 +
 2 \| \phi_4\|_{L^\infty
 (\R, L^\infty (\R^n))}
 \right )
   \int_{\tau -1}^ \tau
\|   u  ( s  ) \|^2_{H^1(\R^n)}
     ds
     $$
     \be\label{est3 p2}
     +
\| \phi_5  \|^2_{L^\infty
(\R, L^2(\R^n)) }
     +
     \| \phi_5\|_{L^\infty
 (\R, L^1(\R^n))}  \int_{\tau -1}^ \tau
    \E (\| u  ( s , \tau-t, u_{ \tau -t } )\|^2) ds.
\ee

For the third term on the right-hand side of \eqref{est3 p1},
by \eqref{g3a} and \eqref{g4}  we have
for all $r\in (\tau -1, \tau)$,
 $$
  2 \int_r ^\tau
   \int_{\R^n}
     \nabla g   ( s, x, u   ( s, \tau-t, u_{ \tau -t },
     \mathcal{L}_{u(s, \tau-t, u_{ \tau -t } )} )
     \cdot \nabla     u   ( s, \tau-t, u_{ \tau -t } )
                     dx  ds
$$
$$
  = 2 \int_r ^\tau
   \int_{\R^n}
     {\frac {\partial g}{\partial x}}
         ( s, x, u   ( s, \tau-t, u_{ \tau -t },
     \mathcal{L}_{u(s, \tau-t, u_{ \tau -t } )} )
     \cdot \nabla     u   ( s, \tau-t, u_{ \tau -t } )
                     dx  ds
$$
$$
 + 2 \int_r ^\tau
   \int_{\R^n}
     {\frac {\partial g}{\partial u}}
         ( s, x, u   ( s, \tau-t, u_{ \tau -t },
     \mathcal{L}_{u(s, \tau-t, u_{ \tau -t } )} )
     \cdot \nabla     u   ( s, \tau-t, u_{ \tau -t } )
                     dx  ds
$$
$$
 \le
 2 \int_r ^\tau
   \int_{\R^n}
  \left (
  \phi_8 (s,x)
  +\phi_7(s,x)
  (|u (     s, \tau-t, u_{ \tau -t })|
  +\sqrt{
  \E (\| u(s)\|^2)
  } \  ) \right )
     |\nabla     u   ( s  )|
                     dx  ds
$$
$$
 + 2 \int_r ^\tau
   \int_{\R^n}
   \phi_7 (s,x)
    | \nabla     u   ( s, \tau-t, u_{ \tau -t } )|
                     dx  ds
$$
$$
 \le
   \int_r ^\tau
  \left (
  \|\phi_8 (s)\|^2
  + \| \nabla u (     s, \tau-t, u_{ \tau -t }) \|^2
  \right ) ds
  $$
  $$
   + \int_r ^\tau
   \|\phi_7(s)\|_{L^\infty(\R^n)}
   \left (
 \|u (     s, \tau-t, u_{ \tau -t })\|^2
 + 2 \|\nabla u (     s, \tau-t, u_{ \tau -t })\|^2
 \right ) ds
 $$
 $$ + \int_r ^\tau    \|\phi_7(s)\|_{L^1(\R^n)}
   \E (\| u( s, \tau-t, u_{ \tau -t })\|^2)
    ds
 $$
$$
 +  \int_r ^\tau
   \left (
   \| \phi_7 (s)  \|^2 +
    \| \nabla     u   ( s, \tau-t, u_{ \tau -t } )\|^2
    \right )  ds
$$
$$
\le
\| \phi_8  \|^2_{L^2
(\tau -1, \tau, L^2(\R^n)) }
+\| \phi_7\|_{L^\infty
 (\R, L^2 (\R^n))}
$$
$$
 + 2
 \left (
 1 +   \| \phi_7\|_{L^\infty
 (\R, L^\infty (\R^n))}
  \right )
   \int_{\tau -1}^ \tau
\|   u  ( s  ) \|^2_{H^1(\R^n)}
     ds
     $$
     \be\label{est3 p3}
     +
     \| \phi_7\|_{L^\infty
 (\R, L^1(\R^n))}  \int_{\tau -1}^ \tau
    \E (\| u  ( s , \tau-t, u_{ \tau -t } )\|^2) ds.
\ee

For the  fourth term on the right-hand side of \eqref{est3 p1},
  we   have
for all $r\in (\tau -1, \tau)$,
  $$
   \int_r^ \tau
\| \nabla           \sigma
  (  s, u   ( s, \tau-t, u_{ \tau -t } ), \mathcal{L}_{u(s,
  \tau-t, u_{ \tau -t } )} )
  \|^2_{L_2(l^2, H)} ds
 $$
 $$
 =\sum_{k=1}^\infty
 \int_r^ \tau
 \| \nabla \theta_k (s)
 +\nabla \left (
 \kappa  \sigma_k (s, u(s, \tau -t, u_{ \tau -t }),
 \mathcal{L}_{u(s,
  \tau-t, u_{ \tau -t } ) })
 \right )\|^2 ds
 $$
  $$
    \le
   2
   \sum_{k=1}^\infty
    \int_r^ \tau
   \|\nabla \theta_k (s) \|^2ds
   +
  4
   \sum_{k=1}^\infty
    \int_r^ \tau
    \|(\nabla \kappa)
    \sigma_k (s, u(s, \tau -t, u_{ \tau -t }),
 \mathcal{L}_{u(s,
  \tau-t, u_{ \tau -t } ) })
  \|^2 ds
  $$
\be\label{est3 p4}
 +4
   \sum_{k=1}^\infty
    \int_r^ \tau
 \| \kappa  \nabla u(s, \tau -t, u_{ \tau -t })
 \cdot
 {\frac {\partial \sigma_k}{\partial u}}
 (s, u(s, \tau -t, u_{ \tau -t }),
 \mathcal{L}_{u(s,
  \tau-t, u_{ \tau -t } ) })
  \|^2 ds .
  \ee
 For the second term on the right-hand side
 of \eqref{est3 p3}, by \eqref{s1}  we have
 for $r\in(\tau -1, \tau)$,
  $$
  4
   \sum_{k=1}^\infty
    \int_r^ \tau
    \|(\nabla \kappa)
    \sigma_k (s, u(s, \tau -t, u_{ \tau -t }),
 \mathcal{L}_{u(s,
  \tau-t, u_{ \tau -t } ) })
  \|^2 ds
  $$
 $$
 \le
 16
   \sum_{k=1}^\infty \beta_k^2
    \int_r^ \tau
    \int_{\R^n}
    |\nabla \kappa (x)|^2
    \left (1+ \E( \|u(s,
  \tau-t, u_{ \tau -t } )    \|^2)
  \right )   dx ds
  $$
  $$
  + 8
   \sum_{k=1}^\infty \gamma_k^2
    \int_r^ \tau
    \int_{\R^n}
     |\nabla \kappa (x)|^2
    |u(s,
  \tau-t, u_{ \tau -t } )|^2 dx ds
   $$
 $$
 \le
 16 \|\beta\|^2_{l^2}
 \|\nabla \kappa \|^2
 \left (1 + \int_{\tau -1}^\tau
 \E( \|u(s,
  \tau-t, u_{ \tau -t } )    \|^2) ds
  \right )
  $$
  \be\label{est3 p5}
  + 8 \|\gamma\|^2_{l^2}
  \|\nabla  \kappa \|_{L^\infty (\R^n)}
  \int_{\tau -1}^\tau
  \|u(s,
  \tau-t, u_{ \tau -t } )   \|^2 ds.
  \ee
  For the last term on the right-hand side
  of \eqref{est3 p4}, by
  \eqref{s3}  we have for $r\in (\tau -1, \tau)$,
   $$
   4
   \sum_{k=1}^\infty
    \int_r^ \tau
 \| \kappa  \nabla u(s, \tau -t, u_{ \tau -t })
 \cdot
 {\frac {\partial \sigma_k}{\partial u}}
 (s, u(s, \tau -t, u_{ \tau -t }),
 \mathcal{L}_{u(s,
  \tau-t, u_{ \tau -t } ) })
  \|^2 ds
  $$
\be\label{est3 p6}
  \le 4
   \sum_{k=1}^\infty L_{\sigma, k}^2
   \|\kappa \|^2_{L^\infty(\R^n)}
    \int_r^ \tau
    \| \nabla u(s, \tau -t, u_{ \tau -t })\|^2 ds.
 \ee

  By
  \eqref{theta1} and
   \eqref{est3 p4}-\eqref{est3 p6}
  we get for all
  $r\in (\tau -1, \tau)$,
 $$
   \int_r^ \tau
\| \nabla           \sigma
  (  s, u   ( s, \tau-t, u_{ \tau -t } ), \mathcal{L}_{u(s,
  \tau-t, u_{ \tau -t } )} )
  \|^2_{L_2(l^2, H)} ds
 $$
\be\label{est3 p7}
 \le
 c_1 +c_1
 \int_{\tau -1}^\tau
 \left (
 \E \left (
 \|   u(s, \tau -t, u_{ \tau -t })\|^2
 \right )
 +
 \|   u(s, \tau -t, u_{ \tau -t })\|^2 _{H^1(\R^n)}
 \right ) ds,
 \ee
 where $c_1=c_1(\tau)>0$ depends on $\tau$, but
 not on $D$.
  It follows from
 \eqref{est3 p1}-\eqref{est3 p3}
 and \eqref{est3 p7} that
   for  all $ \tau \in \mathbb{R} $,  $ t \ge  1 $
 and $ r \in ( \tau-1, \tau ) $,
 $$
  \E \left (\|  \nabla u   ( \tau, \tau-t, u_{ \tau -t } )   \|^2
    \right )
  \le
 \E \left (
 \|  \nabla u   ( r, \tau-t, u_{ \tau -t } )   \|^2
 \right )
 $$
 \be\label{est3 p8}
  +
 c_2 +c_2
 \int_{\tau -1}^\tau
  \E   \left (
 \|   u(s, \tau -t, u_{ \tau -t })\|^2 _{H^1(\R^n)}
 \right ) ds,
 \ee
 where $c_2=c_2(\tau)>0$ depends on $\tau$, but
 not on $D$.
 Integrating \eqref{est3 p8} with respect to
 $r$ on $(\tau -1, \tau)$, we get
  for  all $ \tau \in \mathbb{R} $ and  $ t \ge  1 $,
 $$
  \E \left (\|  \nabla u   ( \tau, \tau-t, u_{ \tau -t } )   \|^2
    \right )
  \le
 c_2 +(1+ c_2)
 \int_{\tau -1}^\tau
  \E   \left (
 \|   u(s, \tau -t, u_{ \tau -t })\|^2 _{H^1(\R^n)}
 \right ) ds,
$$
 where $c_2=c_2(\tau)>0$ depends on $\tau$, but
 not on $D$,
 which along with Lemma \ref{est1}
 and Corollary  \ref{est2}
 concludes the proof.
\end{proof}

   Next, we derive the uniform estimates on the tails
   of solutions which will be used
   to
   overcome the non-compactness
   of Sobolev embeddings on unbounded domains
   when proving
     the asymptotic compactness
   of distributions of solutions.

\begin{lem}\label{est4}
If  {\bf  (H1)}-{\bf (H4)}, \eqref{lamc} and \eqref{phi_g}  hold,
 then for every
 $\delta>0$,
 $ \tau \in \mathbb{R} $
and
$ D = \{  D ( t ) : t \in \mathbb{R}   \}
  \in \mathcal{D} _0  $,
there exist
$ T = T ( \delta, \tau, D ) \ge 1 $
and $N=N(\delta, \tau)\in \N$
such that for all $ t \ge T $ and $n\ge N$,
the solution   $u $   of   \eqref{sde3}-\eqref{sde4} satisfies,
  $$
    \E \left (
    \int_{|x|\ge n}
    |    u   ( \tau, \tau - t, u_{ \tau -t }  ) (x)|^2
    dx \right )
    <\delta,
    $$
where
$ u_{ \tau - t }  \in L^2  ( \Omega, \mathcal{F}_{ \tau - t },
H )  $
with
$ \mathcal{L}_{ u_{ \tau - t } } \in    D ( \tau - t )$.
 \end{lem}

   \begin{proof}
   Let $\rho: \mathbb{R}^n \rightarrow [0,1]$ be a
   continuously differentiable function
   such that
   \be\label{rho}
   \rho(x)=0  \  \text{ for }
   |x| \le \frac{1}{2};
   \quad
   \rho(x)=1 \ \text{ for } \   |x| \ge 1.
   \ee
 Given $n\in \mathbb{N}$,
   denote by  $\rho_n(x) = \rho( \frac{x}{n} )$.
   For convenience,
   we write
  $u(t)= u(t, \tau, u_\tau)$
  for  the solution
of \eqref{sde3}-\eqref{sde4}
with initial data $u_\tau$
at initial time $ \tau$.

   By \eqref{sde3}   we  have
   $$
       d \rho_n u(t) -\rho_n \Delta  u (t) dt
       +\lambda \rho_n u(t) dt
         + \rho_n f(t, \cdot,  u  (t),
          \mathcal{L}_{u (t)} ) dt
          $$
          \be \label{est4 p1a}
   =   \rho_n g(t, \cdot, u (t),
   \mathcal{L}_{u  (t)} ) dt
       +
                 \rho_n  \sigma(t, u  (t),
                  \mathcal{L}_{u  (t)} ) dW(t).
\ee
  By \eqref{est4 p1a} and It\^{o}'s formula, we get for
all  $ t\ge \tau $,
 $$
     e^{\eta t} \| \rho_n u  (t) \|^2
      + 2 \int_{ \tau} ^t
                 e^{\eta s}   (\nabla (\rho_n^2 u(s)),
                  \nabla    u (s) )
                  ds
                 + (2 \lambda-\eta)
                  \int_ { \tau}^t e^{\eta s}
                  \| \rho_n  u  ( s )  \|^2 ds
                  $$
                  $$
      + 2 \int_ { \tau} ^t\int_{\mathbb{R} ^n}
      e^{\eta s}  \rho_n^2(x)
        f( s, x, u  (s,x), \mathcal{L}_{u  (s)} ) u  (s,x) dx ds
      $$
      $$
  =   e^{\eta \tau}\| \rho_n u_ { \tau} \|^2
      + 2 \int_ { \tau}^t
      e^{\eta s}
      (g(s, u  (s), \mathcal{L}_{u (s)} ),  \rho_n^2 u  (s) ) ds
      +   \int_ { \tau}^t
      e^{\eta s}
      \|\rho_n  \sigma(s, u  (s), \mathcal{L}_{u  (s)} ) \|_{L_2(l^2,H)}^2 ds
      $$
        \be \label{est4 p1}
     + 2   \int_ { \tau}^t e^{\eta s}
          \left(\rho_n^2  u  (s),
         \sigma(s, u  (s), \mathcal{L}_{u  (s)}) dW(s)
                                      \right),
 \ee
  $    \mathbb{P}$-almost surely.
Given  $m\in \N$,  denote by
$$
\tau_m
=\inf
\{
t\ge \tau:  \| u(t) \| > m
\}.
$$
By \eqref{est4 p1} we have
for all  $ t\ge \tau $,
$$
    \mathbb{E}
    \left (  e^{\eta (t\wedge \tau_m) }
    \| \rho_n u  (t\wedge \tau_m ) \|^2
      + 2 \int_{ \tau} ^{t\wedge \tau_m}
                 e^{\eta s}   (\nabla (\rho_n ^2 u(s)),
                   \nabla    u (s)
                 )  ds
     \right )
     $$
     $$=
         \mathbb{E} \left (e^{\eta \tau} \| \rho_n u_ { \tau} \|^2
         \right )
                 + (\eta - 2 \lambda )
                     \mathbb{E}
                     \left (
                  \int_ { \tau}^ {t\wedge \tau_m} e^{\eta s}
                  \|  \rho_n  u  ( s )  \|^2 ds
                  \right )
                  $$
                  $$
      - 2
       \mathbb{E}
                     \left (
      \int_ { \tau} ^ {t\wedge \tau_m}
      \int_{\mathbb{R} ^n}
      e^{\eta s}   \rho_n ^2(x)
        f( s, x, u  (s,x), \mathcal{L}_{u  (s)} ) u  (s,x) dx ds
        \right )
      $$
      $$
      + 2  \mathbb{E}
                     \left (
                     \int_ { \tau}^ {t\wedge \tau_m}
      e^{\eta s}
      (g(s, u  (s), \mathcal{L}_{u (s)} ), \rho_n ^2  u  (s) ) ds
      \right )
      $$
              \be \label{est4 p2}
      +   \mathbb{E}
                     \left (
                      \int_ { \tau}^ {t\wedge \tau_m}
      e^{\eta s}
      \|\rho_n   \sigma(s, u  (s), \mathcal{L}_{u  (s)} ) \|_{L_2(l^2,H)}^2 ds
      \right ).
  \ee
   Note that
   $$
    -2 \mathbb{E}
    \left (
        \int_{ \tau} ^{t\wedge \tau_m}
                 e^{\eta s}   (\nabla (\rho_n ^2 u(s)),
                   \nabla    u (s)
                 )  ds
     \right )
    =
    -2 \mathbb{E}
    \left (
        \int_{ \tau} ^{t\wedge \tau_m}
                 e^{\eta s}
                 \int_{\R^n}
                 \rho_n ^2
                  | \nabla    u (s,x )|^2 dx
                   ds
     \right )
     $$
     $$
      -2 \mathbb{E}
    \left (
        \int_{ \tau} ^{t\wedge \tau_m}
                 e^{\eta s}
                 \int_{\R^n}
                 2 n^{-1} u(s,x) \rho_n (x)  \nabla
                 \rho  (x/n) \cdot
                   \nabla    u (s,x )  dx
                   ds
     \right )
     $$
     $$
     \le
   4 n^{-1} \|\nabla \rho \|_{L^\infty(\R^n) } \mathbb{E}
    \left (
        \int_{ \tau} ^{t\wedge \tau_m}
                 e^{\eta s}
                 \|  u(s) \| \|
                   \nabla    u (s)\|
                   ds
     \right )
     $$
   \be\label{est4 p2a}
     \le
    2 n^{-1} \|\nabla \rho \|_{L^\infty(\R^n) } \mathbb{E}
    \left (
        \int_{ \tau} ^{t }
                 e^{\eta s}
                \|
                       u (s)\|^2_{H^1(\R^n)}
                   ds
     \right ).
     \ee

 For the
 third  term on the right-hand side of
\eqref{est4 p2} , by \eqref{f2}, we have
$$
- 2
       \mathbb{E}
                     \left (
      \int_ { \tau} ^ {t\wedge \tau_m}
      \int_{\mathbb{R} ^n}
      e^{\eta s}   \rho_n^2(x)
        f( s, x, u  (s,x), \mathcal{L}_{u  (s)} )   u  (s,x) dx ds
        \right )
        $$
$$
\le
-2\alpha_1
 \mathbb{E}
      \left (
      \int_ { \tau} ^ {t\wedge \tau_m}
      e^{\eta s}\int_{\R^n}
      \rho_n ^2(x)
       | u(s,x)  |^p  dx  ds
      \right )
      +
      2\mathbb{E}
      \left (
      \int_ { \tau} ^ {t\wedge \tau_m}
      e^{\eta s} \| \phi_1 (s)\|_{L^\infty
      (\R^n)}
      \|\rho_n  u(s) \|^2  ds
      \right )
      $$
      $$
      +2
      \mathbb{E}
      \left (
      \int_ { \tau} ^ {t\wedge \tau_m}
      e^{\eta s}
      \int_{\R^n}  \rho_n^2 (x)  |\phi_1 (s,x) |dx
       ds
      \right )
      +
      2
      \mathbb{E}
      \left (
      \int_ { \tau} ^ {t\wedge \tau_m}
      e^{\eta s}   \mathbb{E}
      \left (
      \| u(s) \|^2
      \right )
       \int_{\R^n}  \rho_n^2 (x) \psi_1(x)
    dxds   \right )
      $$
      $$
\le
-2\alpha_1
 \mathbb{E}
      \left (
      \int_ { \tau} ^ {t\wedge \tau_m}
      \int_{\R^n} e^{\eta s}
      \rho_n ^2(x)
       | u(s,x)  |^p  dx  ds
      \right )
      +
      2\mathbb{E}
      \left (
      \int_ { \tau} ^ {t}
      e^{\eta s} \| \phi_1 (s)\|_{L^\infty
      (\R^n)}
      \|\rho_n  u(s) \|^2  ds
      \right )
      $$
      $$
      +2
      \mathbb{E}
      \left (
      \int_ { \tau} ^ {t }
      \int_{\R^n} e^{\eta s}
       \rho_n^2 (x)  |\phi_1 (s,x) |dx
       ds
      \right )
      +
      2
      \mathbb{E}
      \left (
      \int_ { \tau} ^ {t }
      e^{\eta s}   \mathbb{E}
      \left (
      \| u(s) \|^2
      \right )
       \int_{\R^n}  \rho_n^2 (x) \psi_1(x)
    dxds   \right )
      $$
       $$
\le
-2\alpha_1
 \mathbb{E}
      \left (
      \int_ { \tau} ^ {t\wedge \tau_m}
      \int_{\R^n} e^{\eta s}
      \rho_n ^2(x)
       | u(s,x)  |^p  dx  ds
      \right )
      +
      2
      \int_ { \tau} ^ {t}
      e^{\eta s} \| \phi_1 (s)\|_{L^\infty
      (\R^n)}
      \mathbb{E}
      \left (
      \|\rho_n  u(s) \|^2
      \right )ds
      $$
         \be\label{est4 p3}
 +2
      \int_ { \tau} ^ {t }
      \int_{\R^n} e^{\eta s}
       \rho_n^2 (x)  |\phi_1 (s,x) |dx
       ds
      +
      2  \int_{\R^n}  \rho_n^2 (x) \psi_1(x)
    dx
      \int_ { \tau} ^ {t }
      e^{\eta s}   \mathbb{E}
      \left (
      \| u(s) \|^2
      \right ) ds.
      \ee

   For the fourth term on the
   right-hand side of
   \eqref{est1 p2},
   by \eqref{g3} we have
   $$
       2  \mathbb{E}
                     \left (
                     \int_ { \tau}^ {t\wedge \tau_m}
      e^{\eta s}
      (g(s, u  (s), \mathcal{L}_{u (s)} ), \rho_n^2 u  (s) ) ds
      \right )
      $$
      $$
      \le
      2  \mathbb{E}
                     \left (
                     \int_ { \tau}^ {t\wedge \tau_m}
      e^{\eta s}
      \int_{\R^n} \rho_n^2(x)
      \phi_g (s, x)
       |u(s,x)| dxds
       \right )
       $$
       $$
       +
       2  \mathbb{E}
                     \left (
                     \int_ { \tau}^ {t\wedge \tau_m}
      e^{\eta s}
      \int_{\R^n} \rho_n^2(x)
      \left (  \phi_7 (s, x)
      |u(s,x)|^2
      + \psi_g (x) |u(s,x)| \sqrt{
      \mathbb{E}
                     \left (
                     \| u(s)\|^2
                     \right )
                     }
                     \right )
          dx ds
     \right )
      $$
       $$
      \le
       \mathbb{E}
                     \left (
                     \int_ { \tau}^ {t\wedge \tau_m}
      e^{\eta s}
      \int_{\R^n} \rho_n^2(x)
      \left (\eta  |u(s,x)|^2
      +\eta^{-1}
      |\phi_g (s, x)|^2
      \right )
         dxds
       \right )
       $$
       $$
       +
          \mathbb{E}
                     \left (
                     \int_ { \tau}^ {t\wedge \tau_m}
      e^{\eta s}
      \int_{\R^n} \rho_n^2(x)
      \left (  (2 |\phi_7 (s, x)|
      + |\psi_g (x)|)
      |u(s,x)|^2
      + |\psi_g (x)|
      \mathbb{E}
                     \left (
                     \| u(s)\|^2
                     \right )
                     \right )
          dx ds
     \right )
      $$
       $$
      \le
        \mathbb{E}
                     \left (
                     \int_ { \tau}^ {t\wedge \tau_m}
      e^{\eta s}
      \left ( \eta \| \rho_n u(s)\|^2
      + \eta^{-1} \|\rho_n \phi_g (s) \|^2
      \right ) ds
      \right )
       $$
       $$
      +
        \mathbb{E}
                     \left (
                     \int_ { \tau}^ {t\wedge \tau_m}
      e^{\eta s}
     \| \rho_n^2 \psi_g \|_{L^1(\R^n)}
      \mathbb{E}
                     \left (
                     \| u(s)\|^2
                     \right )
            ds
     \right )
     $$
     $$
      +
        \mathbb{E}
                     \left (
                     \int_ { \tau}^ {t\wedge \tau_m}
      e^{\eta s}
    (2 \| \phi_7 (s )\|_{L^\infty (\R^n)}
    +\|\psi_g \|_{L^\infty(\R^n)} )
           \|\rho_n  u(s)\|^2
            ds
     \right )
      $$
      $$
      \le
        \mathbb{E}
                     \left (
                     \int_ { \tau}^ {t }
      e^{\eta s}
      \left ( \eta \| \rho_n u(s)\|^2
      + \eta^{-1} \|\rho_n \phi_g (s) \|^2
      \right ) ds
      \right )
       $$
       $$
      +
        \mathbb{E}
                     \left (
                     \int_ { \tau}^ {t }
      e^{\eta s}
     \| \rho_n^2 \psi_g \|_{L^1(\R^n)}
      \mathbb{E}
                     \left (
                     \| u(s)\|^2
                     \right )
            ds
     \right )
     $$
     $$
      +
        \mathbb{E}
                     \left (
                     \int_ { \tau}^ {t }
      e^{\eta s}
    (2 \| \phi_7 (s )\|_{L^\infty (\R^n)}
    +\|\psi_g \|_{L^\infty(\R^n)} )
           \|\rho_n  u(s)\|^2
            ds
     \right )
      $$
$$
      \le
               \eta^{-1}      \int_ { \tau}^ {t }
      e^{\eta s}
     \| \rho_n \phi_g (s)\|^2  ds
     +
               \int_ { \tau}^ {t }
      e^{\eta s}
      \left (\eta
      +
     2\| \phi_7 (s)\|_{L^\infty (\R^n)}
     +\|\psi_g \|_{L^\infty (\R^n)}
          \right )
        \mathbb{E}
                     \left ( \| \rho_n u(s)\|^2
                     \right )
            ds
            $$
   \be\label{est4 p4}
      +
      \|\rho_n ^2 \psi_g \|_{L^1(\R^n)}
        \int_ { \tau}^ {t }
      e^{\eta s}
        \mathbb{E}
                     \left ( \|   u(s)\|^2
                     \right )
            ds       .
     \ee

     For the last term on the right-hand side
     of \eqref{est4 p2},
     by \eqref{s1} and \eqref{sidef}  we have
      $$
         \mathbb{E}
                     \left (
                      \int_ { \tau}^ {t\wedge \tau_m}
      e^{\eta s}
      \| \rho_n \sigma(s, u  (s), \mathcal{L}_{u  (s)} ) \|_{L_2(l^2,H)}^2 ds
      \right )
      $$
         $$
         \le
          2
          \mathbb{E}
                     \left (
                      \int_ { \tau}^ {t\wedge \tau_m}
      e^{\eta s}
       \| \rho_n \theta (s) \|^2_{L^2(\R^n, l^2)} ds
       \right )
       $$
       $$
      + 8 \|\rho_n \kappa\|^2 \| \beta \|^2_{l^2}
        \mathbb{E}
                     \left (
                      \int_ { \tau}^ {t\wedge \tau_m}
      e^{\eta s}
      \left  (1+  \mathbb{E}
                     \left ( \|u(s) \|^2
                     \right )  \right ) ds \right )
                     $$
      $$
         + 4 \|\kappa\|^2_{L^\infty(\R^n)}
          \| \gamma \|^2_{l^2}
           \mathbb{E}
                     \left (
                      \int_ { \tau}^ {t\wedge \tau_m}
      e^{\eta s}
          \| \rho_n u(s) \|^2 ds
          \right  )
           $$
        $$
         \le
          2
                      \int_ { \tau}^ {t }
      e^{\eta s}
       \| \rho_n \theta (s) \|^2_{L^2(\R^n, l^2)} ds
      + 8 \|\rho_n \kappa\|^2 \| \beta \|^2_{l^2}
                \int_ { \tau}^ {t }
      e^{\eta s}
      \left (1+  \mathbb{E}
                     \left (\|  u(s) \|^2
                     \right )
                     \right ) ds
                     $$
      $$
         + 4 \|\kappa\|^2_{L^\infty(\R^n)}
          \| \gamma \|^2_{l^2}
                      \int_ { \tau}^ {t }
      e^{\eta s}             \mathbb{E}
                     \left (
          \| \rho_n u(s) \|^2
          \right ) ds
           $$
  $$
         \le
          2
                      \int_ { \tau}^ {t }
      e^{\eta s}
       \| \rho_n \theta (s) \|^2_{L^2(\R^n, l^2)} ds
      + 8 \|\rho_n \kappa\|^2 \| \beta \|^2_{l^2}
      \eta^{-1} e^{\eta t}
      $$
    \be\label{est4 p5}
         +
         8 \|\rho_n \kappa\|^2 \| \beta \|^2_{l^2}
                 \int_ { \tau}^ {t }
      e^{\eta s}             \mathbb{E}
                     \left (
          \|  u(s) \|^2
          \right ) ds
          +
           4 \|\kappa\|^2_{L^\infty(\R^n)}
          \| \gamma \|^2_{l^2}
            \int_ { \tau}^ {t }
      e^{\eta s}             \mathbb{E}
                     \left (
          \|  \rho_n u(s) \|^2
          \right ) ds .
        \ee
 It follows from
 \eqref{est4 p2}-\eqref{est4 p5}
 that for all $t\ge \tau$,
 $$
    \mathbb{E}
    \left (  e^{\eta (t\wedge \tau_m) } \|
    \rho_n u  (t\wedge \tau_m ) \|^2
      +
        (  2 \lambda -\eta )
                  \int_ { \tau}^ {t\wedge \tau_m} e^{\eta s}
                  \| \rho_n  u  ( s )  \|^2 ds
        \right )
     $$
     $$\le
         \mathbb{E} \left (e^{\eta \tau} \|
         \rho_n u_ { \tau} \|^2
         \right )
   +2
       \int_ { \tau} ^ {t }
      e^{\eta s}
      \| \rho_n^2  \phi_1 (s)\|_{L^1
      (\R^n)}     ds
       $$
         $$
     +  2 n^{-1} \|\nabla \rho \|_{L^\infty(\R^n) }
         \int_{ \tau} ^{t }
                 e^{\eta s}
               \mathbb{E}
    \left (  \|
                       u (s)\|^2_{H^1(\R^n)}
        \right )             ds
      $$
      $$
      +
        \eta^{-1}      \int_ { \tau}^ {t }
      e^{\eta s}
     \| \rho_n \phi_g (s)\|^2  ds
             +
          2
                      \int_ { \tau}^ {t }
      e^{\eta s}
       \|\rho_n  \theta (s) \|^2_{L^2(\R^n, l^2)} ds
      + 8 \|\rho_n \kappa\|^2 \| \beta \|^2_{l^2}
      \eta^{-1} e^{\eta t}
      $$
  $$
      +
      \int_ { \tau} ^ {t }
      e^{\eta s}
      \left (
      4 \|\kappa\|^2_{L^\infty(\R^n)}
          \| \gamma \|^2_{l^2}
       +  2 \| \phi_1 (s)\|_{L^\infty
      (\R^n)}
        \right )
      \mathbb{E}
      \left (
      \| \rho_n u(s) \|^2
      \right )   ds
 $$
  $$
     +
               \int_ { \tau}^ {t }
      e^{\eta s}
      \left (  \eta
      + 2
     \| \phi_7 (s )\|_{L^\infty (\R^n)}
     + \| \psi_g \|_{L^\infty (\R^n)}
          \right )
        \mathbb{E}
                     \left ( \| \rho_n u(s)\|^2
                     \right )
            ds
            $$
     \be\label{est4 p6}
         +
         \left(
          8 \|\rho_n \kappa\|^2 \| \beta \|^2_{l^2}
         + 2\| \rho_n^2 \psi_1 \|_{L^1(\R^n)}
         +
         \| \rho_n^2 \psi_g \|_{L^1(\R^n)}
          \right )
                      \int_ { \tau}^ {t }
      e^{\eta s}             \mathbb{E}
                     \left (
          \| u(s) \|^2
          \right ) ds .
        \ee
      Taking the limit of \eqref{est4 p6}
      as $m \to \infty$, by Fatou's lemma we obtain
      for all $t\ge \tau$,
 $$
    \mathbb{E}
    \left (  e^{\eta t} \|
    \rho_n u  (t  ) \|^2
      +
        (  2 \lambda -\eta )
                  \int_ { \tau}^ {t } e^{\eta s}
                  \| \rho_n  u  ( s )  \|^2 ds
        \right )
     $$
     $$\le
         \mathbb{E} \left (e^{\eta \tau} \|
         \rho_n u_ { \tau} \|^2
         \right )
   +2
       \int_ { \tau} ^ {t }
      e^{\eta s}
      \| \rho_n^2  \phi_1 (s)\|_{L^1
      (\R^n)}     ds
       $$
         $$
     +  2 n^{-1} \|\nabla \rho \|_{L^\infty(\R^n) }
         \int_{ \tau} ^{t }
                 e^{\eta s}
               \mathbb{E}
    \left (  \|
                       u (s)\|^2_{H^1(\R^n)}
        \right )             ds
      $$
      $$
      +
        \eta^{-1}      \int_ { \tau}^ {t }
      e^{\eta s}
     \| \rho_n \phi_g (s)\|^2  ds
             +
          2
                      \int_ { \tau}^ {t }
      e^{\eta s}
       \|\rho_n  \theta (s) \|^2_{L^2(\R^n, l^2)} ds
      + 8 \|\rho_n \kappa\|^2 \| \beta \|^2_{l^2}
      \eta^{-1} e^{\eta t}
      $$
  $$
      +
      \int_ { \tau} ^ {t }
      e^{\eta s}
      \left (
      4 \|\kappa\|^2_{L^\infty(\R^n)}
          \| \gamma \|^2_{l^2}
       +  2 \| \phi_1 (s)\|_{L^\infty
      (\R^n)}
        \right )
      \mathbb{E}
      \left (
      \| \rho_n u(s) \|^2
      \right )   ds
 $$
  $$
     +
               \int_ { \tau}^ {t }
      e^{\eta s}
      \left (  \eta
      + 2
     \| \phi_7 (s )\|_{L^\infty (\R^n)}
     + \| \psi_g \|_{L^\infty (\R^n)}
          \right )
        \mathbb{E}
                     \left ( \| \rho_n u(s)\|^2
                     \right )
            ds
            $$
     \be\label{est4 p6}
         +
         \left(
          8 \|\rho_n \kappa\|^2 \| \beta \|^2_{l^2}
         + 2\| \rho_n^2 \psi_1 \|_{L^1(\R^n)}
         +
         \| \rho_n^2 \psi_g \|_{L^1(\R^n)}
          \right )
                      \int_ { \tau}^ {t }
      e^{\eta s}             \mathbb{E}
                     \left (
          \| u(s) \|^2
          \right ) ds .
        \ee

        By \eqref{etac}
         and \eqref{est4 p6}
         we get
            for all $t\ge \tau$,
 $$
    \mathbb{E}
    \left (   \|
    \rho_n u  (t  ) \|^2
        \right )
     \le e^{\eta  (\tau-t) }
         \mathbb{E} \left ( \|
         \rho_n u_ { \tau} \|^2
         \right )
   +2
       \int_ { \tau} ^ {t }
      e^{\eta (s-t) }
      \| \rho_n^2  \phi_1 (s)\|_{L^1
      (\R^n)}     ds
       $$
         $$
     +  2 n^{-1} \|\nabla \rho \|_{L^\infty(\R^n) }
         \int_{ \tau} ^{t }
                 e^{\eta  (s-t) }
               \mathbb{E}
    \left (  \|
                       u (s)\|^2_{H^1(\R^n)}
        \right )             ds
      $$
      $$
      +
        \eta^{-1}      \int_ { \tau}^ {t }
      e^{\eta  (s-t)}
     \| \rho_n \phi_g (s)\|^2  ds
             +
          2
                      \int_ { \tau}^ {t }
      e^{\eta  (s-t) }
       \|\rho_n  \theta (s) \|^2_{L^2(\R^n, l^2)} ds
      + 8 \|\rho_n \kappa\|^2 \| \beta \|^2_{l^2}
      \eta^{-1}
      $$
      \be\label{est4 p7}
         +
         \left(
          8 \|\rho_n \kappa\|^2 \| \beta \|^2_{l^2}
         + 2\| \rho_n^2 \psi_1 \|_{L^1(\R^n)}
         +
         \| \rho_n^2 \psi_g \|_{L^1(\R^n)}
          \right )
                      \int_ { \tau}^ {t }
      e^{\eta  (s-t)}             \mathbb{E}
                     \left (
          \| u(s) \|^2
          \right ) ds .
        \ee

   Replacing $\tau$ and $t$
   in \eqref{est4 p7} by $\tau -t$
   and  $\tau$, respectively ,
         we get
            for all $t\ge 0$,
            $$
    \mathbb{E}
    \left (   \|
    \rho_n u  (\tau, \tau -t, u_{\tau -t}   ) \|^2
        \right )
     \le e^{-\eta  t   }
         \mathbb{E} \left ( \|
         \rho_n u_ { \tau -t } \|^2
         \right )
   +2
       \int_ { \tau -t} ^ {\tau }
      e^{\eta (s-\tau) }
      \| \rho_n^2  \phi_1 (s)\|_{L^1
      (\R^n)}     ds
       $$
         $$
     +  2 n^{-1} \|\nabla \rho \|_{L^\infty(\R^n) }
         \int_ { \tau -t} ^ {\tau }
      e^{\eta (s-\tau) }
               \mathbb{E}
    \left (  \|
                       u (s,\tau-t, u_{\tau -t}
                       )\|^2_{H^1(\R^n)}
        \right )             ds
      $$
      $$
      +
        \eta^{-1}       \int_ { \tau -t} ^ {\tau }
      e^{\eta (s-\tau) }
     \| \rho_n \phi_g (s)\|^2  ds
             +
          2
                      \int_ { \tau -t} ^ {\tau }
      e^{\eta (s-\tau) }
       \|\rho_n  \theta (s) \|^2_{L^2(\R^n, l^2)} ds
      + 8 \|\rho_n \kappa\|^2 \| \beta \|^2_{l^2}
      \eta^{-1}
      $$
      \be\label{est4 p8}
         +
         \left(
          8 \|\rho_n \kappa\|^2 \| \beta \|^2_{l^2}
         + 2\| \rho_n^2 \psi_1 \|_{L^1(\R^n)}
         +
         \| \rho_n^2 \psi_g \|_{L^1(\R^n)}
          \right )
                       \int_ { \tau -t} ^ {\tau }
      e^{\eta (s-\tau) }
                 \mathbb{E}
                     \left (
          \| u(s) \|^2
          \right ) ds .
        \ee
        By \eqref{est4 p8} and Lemma \ref{est1}
        we find that there exist
        $c_1=c_1(\tau)>0$ and
         $T_1=T_1(\tau, D)\ge 1$ such that
        for all $t\ge T_1$,
          $$
    \mathbb{E}
    \left (   \|
    \rho_n u  (\tau, \tau -t, u_{\tau -t}   ) \|^2
        \right )
     \le e^{-\eta  t   }
         \mathbb{E} \left ( \|
          u_ { \tau -t } \|^2
         \right )
   +2
       \int_ { -\infty } ^ {\tau }
      e^{\eta (s-\tau) }
      \| \rho_n^2  \phi_1 (s)\|_{L^1
      (\R^n)}     ds
       $$
         $$
      +
        \eta^{-1}       \int_ { -\infty } ^ {\tau }
      e^{\eta (s-\tau) }
     \| \rho_n \phi_g (s)\|^2  ds
             +
          2
                      \int_ { -\infty } ^ {\tau }
      e^{\eta (s-\tau) }
       \|\rho_n  \theta (s) \|^2_{L^2(\R^n, l^2)} ds
       $$
       $$
      + 8 \|\rho_n \kappa\|^2 \| \beta \|^2_{l^2}
      \eta^{-1}
      +  2c_1 n^{-1} \|\nabla \rho \|_{L^\infty(\R^n) }
      $$
      \be\label{est4 p9}
         +
         \left(
          8 \|\rho_n \kappa\|^2 \| \beta \|^2_{l^2}
         + 2\| \rho_n^2 \psi_1 \|_{L^1(\R^n)}
         +
         \| \rho_n^2 \psi_g \|_{L^1(\R^n)}
          \right )c_1.
         \ee

  Since  $ \mathcal{L}_{ u_{\tau-t}  }  \in   D ( \tau-t )$,
  we have
 $$
\lim_{t \to \infty}
 e^{-\eta t }
         \mathbb{E} \left ( \|   u_ { \tau -t} \|^2
         \right )
           \le  e^{- \eta \tau }
          \lim_{t \to \infty}
 e^{\eta (\tau-t) }\| D(\tau -t)\|^2_{\mathcal{P}_2 ( H ) } =0,
         $$
         and hence for every $\delta>0$,
         there
         exists $T_2=T_2(\delta, \tau, D ) \ge T_1$
         such that for all $t\ge T_2$
\be\label{est4 p10}
e^{-\eta t }
         \mathbb{E} \left ( \|  u_ { \tau -t} \|^2
         \right )
         <\delta.
         \ee
         For the second term on the right-hand side
         of \eqref{est4 p9} we have
         $$
         \int_{-\infty}^\tau\int_{\R^n}
         e^{\eta (s-\tau)} |\phi_1 (s,x)|dxds
         \le
          \int_{-\infty}^\tau
         e^{\eta (s-\tau)} \|\phi_1 (s)\|_{L^1(\R^n)}
          ds
          \le \eta^{-1} \| \phi_1 \|_{L^\infty(\R, L^1(\R^n))}
          <\infty,
         $$
         and hence as $n \to \infty$,
        $$ 2
       \int_ { -\infty } ^ {\tau }
      e^{\eta (s-\tau) }
      \| \rho_n^2  \phi_1 (s)\|_{L^1
      (\R^n)}     ds
      \le  2
 \int_{-\infty}^\tau\int_{|x|\ge {{\frac 12} n}}
         e^{\eta (s-\tau)} |\phi_1 (s,x)|dxds
         \to 0,
         $$
 which shows that there exists $N_1=N_1 (\delta, \tau)
 \in \N$ such that
 for all $n \ge N_1$,
\be\label{est4 p11}
2
       \int_ { -\infty } ^ {\tau }
      e^{\eta (s-\tau) }
      \| \rho_n^2  \phi_1 (s)\|_{L^1
      (\R^n)}     ds
 <\delta.
 \ee
 For the third term on the right-hand side of \eqref{est4 p9},
 by \eqref{phi_g} we have, as $n\to \infty$,
    $$
        \eta^{-1}       \int_ { -\infty } ^ {\tau }
      e^{\eta (s-\tau) }
     \| \rho_n \phi_g (s)\|^2  ds
             +
          2
                      \int_ { -\infty } ^ {\tau }
      e^{\eta (s-\tau) }
       \|\rho_n  \theta (s) \|^2_{L^2(\R^n, l^2)} ds
       $$
       $$
       \le
          \eta^{-1} e^{-\eta \tau  }
             \int_ { -\infty } ^ {\tau }
             \int_{|x|\ge { {\frac 12} n}}
      e^{\eta  s  }|g(s,x)|^2 dx ds
      $$
      $$
             +
          2 e^{-\eta \tau  }
                      \int_ { -\infty } ^ {\tau }
                       \int_{|x|\ge { {\frac 12} n}}
      e^{\eta  s }
      \sum_{k=1}^\infty
      |\theta_k (s,x)|^2dxds
      \to 0,
 $$
 and hence there exists $N_2=N_2(\delta,\tau)
 \ge N_1$ such that
 for all $n\ge N_2$,
\be\label{est4 p12}
        \eta^{-1}       \int_ { -\infty } ^ {\tau }
      e^{\eta (s-\tau) }
     \| \rho_n \phi_g (s)\|^2  ds
             +
          2
                      \int_ { -\infty } ^ {\tau }
      e^{\eta (s-\tau) }
       \|\rho_n  \theta (s) \|^2_{L^2(\R^n, l^2)} ds
       <\delta.
  \ee
 For the last three terms on the right-hand side
 of \eqref{est4 p9} we have, as $n\to \infty$,
  $$
       8 \|\rho_n \kappa\|^2 \| \beta \|^2_{l^2}
      \eta^{-1}
      +  2c_1 n^{-1} \|\nabla \rho \|_{L^\infty(\R^n) }
      $$
   $$
         +
         \left(
          8 \|\rho_n \kappa\|^2 \| \beta \|^2_{l^2}
         + 2\| \rho_n^2 \psi_1 \|_{L^1(\R^n)}
         +
         \| \rho_n^2 \psi_g \|_{L^1(\R^n)}
          \right )c_1
          $$
          $$
          \le
       8  \| \beta \|^2_{l^2}
      \eta^{-1}   \int_{|x|\ge { {\frac 12} n}}
       \kappa^2 (x) dx
        +  2c_1 n^{-1} \|\nabla \rho \|_{L^\infty(\R^n) }
      $$
   $$
         +
         \left(
           8| \beta \|^2_{l^2}
          \int_{|x|\ge { {\frac 12} n}}
       \kappa^2 (x) dx
          + 2     \int_{|x|\ge { {\frac 12} n}}
       \|\psi_1 (x)|   dx
         +     \int_{|x|\ge { {\frac 12} n}}
       |\psi_g (x) |
          \right )c_1
          \to 0,
          $$
 and hence there exists $N_3=N_3(\delta, \tau)
 \ge N_2$ such that
  $$
       8 \|\rho_n \kappa\|^2 \| \beta \|^2_{l^2}
      \eta^{-1}
      +  2c_1 n^{-1} \|\nabla \rho \|_{L^\infty(\R^n) }
      $$
  \be\label{est4 p13}
         +
         \left(
          8 \|\rho_n \kappa\|^2 \| \beta \|^2_{l^2}
         + 2\| \rho_n^2 \psi_1 \|_{L^1(\R^n)}
         +
         \| \rho_n^2 \psi_g \|_{L^1(\R^n)}
          \right )c_1 <\delta.
        \ee

 It
  follows from \eqref{est4 p9}-\eqref{est4 p13} that
  for all $t\ge T_2$ and $n\ge N_3$,
          $$
           \mathbb{E}
    \left (  \int_{|x|\ge n}
   | u  (\tau, \tau -t, u_{\tau -t}   ) (x)|^2
   dx
   \right )
   \le
    \mathbb{E}
    \left (   \|
    \rho_n u  (\tau, \tau -t, u_{\tau -t}   ) \|^2
        \right )
        <4 \delta,
        $$
        as desired.
   \end{proof}

   Next, we derive
   the  uniform estimates of solutions
   of \eqref{sde3}-\eqref{sde4}
   in $L^4(\Omega, \mathcal{F}, H)$.

\begin{lem}\label{est5}
If  {\bf  (H1)}-{\bf (H4)},
\eqref{lamc} and \eqref{phi_ga}
hold,
 then for every $ \tau \in \mathbb{R} $
and
$ D = \{  D ( t ) : t \in \mathbb{R}   \}     \in \mathcal{D}   $,
there exists
$ T = T ( \tau, D ) >0 $
such that for all $ t \ge T $,
the solution   $u $   of   \eqref{sde3}-\eqref{sde4} satisfies,
\be\label{est5 1}
\E  \left (  \|  u  ( \tau, \tau - t, u_{\tau -t}  )  \|^4
    \right )
\le   M_3
     + M_3  \int_{ -\infty }^\tau
     e^{  2\eta  ( s - \tau ) }
  \left (
     \| \phi_g (s)\|^4     +
       \| \theta (s) \|^4_{L^2(\R^n, l^2)}
       \right )  ds,
\ee
 where
$ u_{ \tau - t }  \in L^4  ( \Omega, \mathcal{F}_{ \tau - t },
H )  $
with
$ \mathcal{L}_{ u_{ \tau - t } } \in    D ( \tau - t )$,
  $\eta>0$ is the same number as in \eqref{etac},
  and
$ M_3>0 $  is a   constant independent of $ \tau $
 and $D$.
\end{lem}

\begin{proof}
 Let  $u(t)= u(t, \tau, u_\tau)$ be the solution
of \eqref{sde3}-\eqref{sde4}
with initial data $u_\tau$
at initial time $ \tau$.

By \eqref{est1 p1} and It\^{o}'s formula, we get for
all  $ t\ge \tau $,
 $$
     e^{2\eta t} \| u  (t) \|^4
      + 4 \int_{ \tau} ^t
                 e^{2 \eta   s}  \|     u (s) \|^2
                 \| \nabla    u (s)
                 \|^2 ds
                 + 2 (2 \lambda-\eta)
                  \int_ { \tau}^t e^{ 2 \eta s}
                  \|  u  ( s )  \|^4 ds
                  $$
                  $$
      + 4 \int_ { \tau} ^t
       e^{2\eta s}   \|     u (s) \|^2
      \int_{\mathbb{R} ^n}
        f( s, x, u  (s,x), \mathcal{L}_{u  (s)} ) u  (s,x) dx ds
      $$
      $$
  = e^{2 \eta \tau}\| u_ { \tau} \|^4
      + 4\int_ { \tau}^t
      e^{2 \eta s}    \|     u (s) \|^2
      (g(s, u  (s), \mathcal{L}_{u (s)} ),  u  (s) ) ds
     $$
     $$ +   2 \int_ { \tau}^t
      e^{2 \eta s}   \|     u (s) \|^2
      \| \sigma(s, u  (s), \mathcal{L}_{u  (s)} ) \|_{L_2(l^2,H)}^2 ds
      $$
      $$
      +     4 \int_ { \tau}^t
      e^{2 \eta s}   \|     u^* (s)
      \sigma(s, u  (s), \mathcal{L}_{u  (s)} ) \|_{L_2(l^2,\R)}^2
      ds
      $$
        \be \label{est5 p1}
     + 4  \int_ { \tau}^t e^{2\eta s}
      \|     u (s) \|^2
          \left( u  (s),
         \sigma(s, u  (s), \mathcal{L}_{u  (s)}) dW(s)
                                      \right),
 \ee
  $    \mathbb{P}$-almost surely,
  where $u^*(s)$  is the element in $H^*$
  identified by the Riesz representation theorem.
   Let
$
\tau_m
=\inf
\{
t\ge \tau:  \| u(t) \| > m
\}.
$
By \eqref{est1 p1} we have
for
all  $ t\ge \tau $,
 $$
     e^{2\eta (t\wedge \tau_m)} \| u   (t\wedge \tau_m) \|^4
      + 4 \int_{ \tau} ^ {t\wedge \tau_m}
                 e^{2 \eta   s}  \|     u (s) \|^2
                 \| \nabla    u (s)
                 \|^2 ds
                 + 2 (2 \lambda-\eta)
                  \int_ { \tau}^ {t\wedge \tau_m} e^{ 2 \eta s}
                  \|  u  ( s )  \|^4 ds
                  $$
                  $$
      + 4 \int_ { \tau} ^ {t\wedge \tau_m}
       e^{2\eta s}   \|     u (s) \|^2
      \int_{\mathbb{R} ^n}
        f( s, x, u  (s,x), \mathcal{L}_{u  (s)} ) u  (s,x) dx ds
      $$
      $$
  = e^{2 \eta \tau}\| u_ { \tau} \|^4
      + 4\int_ { \tau}^ {t\wedge \tau_m}
      e^{2 \eta s}    \|     u (s) \|^2
      (g(s, u  (s), \mathcal{L}_{u (s)} ),  u  (s) ) ds
     $$
     $$ +   2 \int_ { \tau}^ {t\wedge \tau_m}
      e^{2 \eta s}   \|     u (s) \|^2
      \| \sigma(s, u  (s), \mathcal{L}_{u  (s)} ) \|_{L_2(l^2,H)}^2 ds
      $$
      $$
      +     4 \int_ { \tau}^ {t\wedge \tau_m}
      e^{2 \eta s}   \|     u^* (s)
      \sigma(s, u  (s), \mathcal{L}_{u  (s)} ) \|_{L_2(l^2,\R)}^2
      ds
      $$
        \be \label{est5 p2a}
     + 4  \int_ { \tau}^ {t\wedge \tau_m}  e^{2\eta s}
      \|     u (s) \|^2
          \left( u  (s),
         \sigma(s, u  (s), \mathcal{L}_{u  (s)}) dW(s)
                                      \right).
 \ee
Taking the expectation of \eqref{est5 p2a} we get,
for
all  $ t\ge \tau $,
 $$
     \E \left (
     e^{2\eta (t\wedge \tau_m)} \| u   (t\wedge \tau_m) \|^4
     \right )
          + 2 (2 \lambda-\eta)
          \E \left (
                  \int_ { \tau}^ {t\wedge \tau_m} e^{ 2 \eta s}
                  \|  u  ( s )  \|^4 ds
                  \right )
                  $$
                  $$
      \le
      \E \left (e^{2 \eta \tau}\| u_ { \tau} \|^4
      \right )
      - 4
      \E \left (
      \int_ { \tau} ^ {t\wedge \tau_m}
       e^{2\eta s}   \|     u (s) \|^2
      \int_{\mathbb{R} ^n}
        f( s, x, u  (s,x), \mathcal{L}_{u  (s)} ) u  (s,x) dx ds
        \right )
      $$
      $$
      + 4
      \E \left (
      \int_ { \tau}^ {t\wedge \tau_m}
      e^{2 \eta s}    \|     u (s) \|^2
      (g(s, u  (s), \mathcal{L}_{u (s)} ),  u  (s) ) ds
      \right )
     $$
         \be \label{est5 p2}
         +   6
     \E \left (
     \int_ { \tau}^ {t\wedge \tau_m}
      e^{2 \eta s}   \|     u (s) \|^2
      \| \sigma(s, u  (s), \mathcal{L}_{u  (s)} ) \|_{L_2(l^2,H)}^2 ds   \right ).
      \ee
 For the
 second   term on the right-hand side of
\eqref{est5 p2} , by \eqref{f2}, we have
$$
- 4
      \E \left (
      \int_ { \tau} ^ {t\wedge \tau_m}
       e^{2\eta s}   \|     u (s) \|^2
      \int_{\mathbb{R} ^n}
        f( s, x, u  (s,x), \mathcal{L}_{u  (s)} ) u  (s,x) dx ds
        \right )
      $$
$$
\le
-4\alpha_1
 \mathbb{E}
      \left (
      \int_ { \tau} ^ {t\wedge \tau_m}
      e^{2\eta s}
       \| u(s) \|^2
      \| u(s) \|^p_{L^p(\R^n)} ds
      \right )
      +
      4\mathbb{E}
      \left (
      \int_ { \tau} ^ {t\wedge \tau_m}
      e^{2 \eta s} \| \phi_1 (s)\|_{L^\infty
      (\R^n)}
      \| u(s) \|^4  ds
      \right )
      $$
      $$
      +4
      \mathbb{E}
      \left (
      \int_ { \tau} ^ {t\wedge \tau_m}
      e^{2\eta s}  \| u(s) \|^2
      \left (
      \| \phi_1 (s)\|_{L^1
      (\R^n)}
       +
        \| \psi_1 \|_{L^1
      (\R^n)}
       \mathbb{E}
      \left (
      \| u(s) \|^2
      \right )
      \right )  ds
      \right )
      $$
      $$ \le
      4
      \int_ { \tau} ^ {t }
      e^{2 \eta s} \| \phi_1 (s)\|_{L^\infty
      (\R^n)}
      \mathbb{E}
      \left (\| u(s) \|^4
      \right )  ds
      $$
      $$
      +4
      \int_ { \tau} ^ {t }
      e^{2\eta s}
      \| \phi_1 (s)\|_{L^1
      (\R^n)}
      \mathbb{E}
      \left (
      \| u(s) \|^2
      \right )
      + 4
      \int_ { \tau} ^ {t }
      e^{2\eta s}
        \| \psi_1 \|_{L^1
      (\R^n)}
      \left (
       \mathbb{E}
      \left (
      \| u(s) \|^2
      \right )
      \right )^2  ds
      $$
       $$
          \le
       \int_ { \tau} ^ {t }
      e^{2 \eta s}
      \left (4 \| \phi_1 (s)\|_{L^\infty
      (\R^n)}
      +4  \| \psi_1 \|_{L^1
      (\R^n)}
      + {\frac 13} \eta
      \right )
      \mathbb{E}
      \left (\| u(s) \|^4
      \right )  ds
      $$
        \be\label{est5 p3}
      + 12 \eta^{-1}
        \int_ { \tau} ^ {t }
      e^{2 \eta s}
      \| \phi_1 (s)\|^2 _{L^1
      (\R^n)} ds.
      \ee
     For the third term on the
   right-hand side of
   \eqref{est5 p2},
   by \eqref{g3} we have
       $$
      4
      \E \left (
      \int_ { \tau}^ {t\wedge \tau_m}
      e^{2 \eta s}    \|     u (s) \|^2
      (g(s, u  (s), \mathcal{L}_{u (s)} ),  u  (s) ) ds
      \right )
      $$
      $$
      \le 2
        \mathbb{E}
                     \left (
                     \int_ { \tau}^ {t\wedge \tau_m}
      e^{2\eta s}
      \left ( \eta \| u(s)\|^4
      + \eta^{-1} \|\phi_g (s) \|^2\| u(s)\|^2
      \right ) ds
      \right )
       $$
       $$
      +2
        \mathbb{E}
                     \left (
                     \int_ { \tau}^ {t\wedge \tau_m}
      e^{2\eta s}   \| u(s)\|^2
     \| \psi_g \|_{L^1(\R^n)}
      \mathbb{E}
                     \left (
                     \| u(s)\|^2
                     \right )
            ds
     \right )
     $$
     $$
      + 2
        \mathbb{E}
                     \left (
                     \int_ { \tau}^ {t\wedge \tau_m}
      e^{2\eta s}
     (2\| \phi_7 (s )\|_{L^\infty (\R^n)}
     +\|\psi_g\|_{L^\infty (\R^n)})
           \| u(s)\|^4
            ds
     \right )
      $$
       $$
      \le 2
        \mathbb{E}
                     \left (
                     \int_ { \tau}^ {t }
      e^{2\eta s}
      \left ({\frac 43}  \eta \| u(s)\|^4
      + {\frac 34}  \eta^{-3} \|\phi_g (s) \|^4
      \right ) ds
      \right )
       $$
       $$
      +2
        \mathbb{E}
                     \left (
                     \int_ { \tau}^ {t }
      e^{2\eta s}   \| u(s)\|^2
     \| \psi_g \|_{L^1(\R^n)}
      \mathbb{E}
                     \left (
                     \| u(s)\|^2
                     \right )
            ds
     \right )
     $$
     $$
      + 2
        \mathbb{E}
                     \left (
                     \int_ { \tau}^ {t }
      e^{2\eta s}
     (2\| \phi_7 (s )\|_{L^\infty (\R^n)}
     +\|\psi_g\|_{L^\infty (\R^n)})
           \| u(s)\|^4
            ds
     \right )
      $$
      $$
      \le
              {\frac 32}
               \eta^{-3}      \int_ { \tau}^ {t }
      e^{2 \eta s}
     \| \phi_g (s)\|^4  ds
     $$
        \be\label{est5 p4}
     +
             2  \int_ { \tau}^ {t }
      e^{2 \eta s}
      \left (\eta +  \| \psi_g \|_{L^1 (\R^n)}
      + 2
     \| \phi_7 (s)\|_{L^\infty (\R^n)}
     +\|\psi_g \|_{L^\infty (\R^n)}
          \right )
        \mathbb{E}
                     \left ( \| u(s)\|^2
                     \right )
            ds .
            \ee
      For the last term on the right-hand side
     of \eqref{est5 p2},
     by \eqref{s4} we have
     $$  6
     \E \left (
     \int_ { \tau}^ {t\wedge \tau_m}
      e^{2 \eta s}   \|     u (s) \|^2
      \| \sigma(s, u  (s), \mathcal{L}_{u  (s)} ) \|_{L_2(l^2,H)}^2 ds   \right )
      $$
         $$
         \le
        12
          \mathbb{E}
                     \left (
                      \int_ { \tau}^ {t\wedge \tau_m}
      e^{2\eta s}  \|u(s)\|^2
       \| \theta (s) \|^2_{L^2(\R^n, l^2)} ds
       \right )
       $$
       $$
      + 48 \|\kappa\|^2 \| \beta \|^2_{l^2}
        \mathbb{E}
                     \left (
                      \int_ { \tau}^ {t\wedge \tau_m}
      e^{2 \eta s}   \|u(s)\|^2
      \left  (1+  \mathbb{E}
                     \left ( \|u(s) \|^2
                     \right )  \right ) ds \right )
                     $$
      $$
         + 24\|\kappa\|^2_{L^\infty(\R^n)}
          \| \gamma \|^2_{l^2}
           \mathbb{E}
                     \left (
                      \int_ { \tau}^ {t\wedge \tau_m}
      e^{2 \eta s}
          \| u(s) \|^4 ds
          \right  )
           $$
           $$
         \le
        12
          \mathbb{E}
                     \left (
                      \int_ { \tau}^ {t }
      e^{2\eta s}  \|u(s)\|^2
       \| \theta (s) \|^2_{L^2(\R^n, l^2)} ds
       \right )
       $$
       $$
      + 48 \|\kappa\|^2 \| \beta \|^2_{l^2}
        \mathbb{E}
                     \left (
                      \int_ { \tau}^ {t }
      e^{2 \eta s}   \|u(s)\|^2
      \left  (1+  \mathbb{E}
                     \left ( \|u(s) \|^2
                     \right )  \right ) ds \right )
                     $$
      $$
         + 24\|\kappa\|^2_{L^\infty(\R^n)}
          \| \gamma \|^2_{l^2}
           \mathbb{E}
                     \left (
                      \int_ { \tau}^ {t }
      e^{2 \eta s}
          \| u(s) \|^4 ds
          \right  )
           $$
            $$
         \le
        {\frac 13}\eta
                      \int_ { \tau}^ {t }
      e^{2\eta s}
       \mathbb{E}
                     \left (
      \|u(s)\|^4
      \right ) ds
      + 108 \eta^{-1}   \int_ { \tau}^ {t }
      e^{2\eta s}
       \| \theta (s) \|^4_{L^2(\R^n, l^2)} ds
       $$
       $$
      + 48 \|\kappa\|^2 \| \beta \|^2_{l^2}
                      \int_ { \tau}^ {t }
      e^{2 \eta s}    \mathbb{E}
                     \left (
                     \|u(s)\|^4
                     \right ) ds
                     +
                      48 \|\kappa\|^2 \| \beta \|^2_{l^2}
                      \int_ { \tau}^ {t }
      e^{2 \eta s}    \mathbb{E}
                     \left (
                     \|u(s)\|^2
                     \right ) ds
                     $$
      $$
         + 24\|\kappa\|^2_{L^\infty(\R^n)}
          \| \gamma \|^2_{l^2}
                      \int_ { \tau}^ {t }
      e^{2 \eta s}
            \mathbb{E}
                     \left (\| u(s) \|^4
                     \right )
                     ds
           $$
           $$
           \le  \left (
           {\frac 12}\eta
           +48 \|\kappa\|^2 \| \beta \|^2_{l^2}
         + 24\|\kappa\|^2_{L^\infty(\R^n)}
          \| \gamma \|^2_{l^2}\right )
              \int_ { \tau}^ {t }
      e^{2 \eta s}
            \mathbb{E}
                     \left (\| u(s) \|^4
                     \right )
                     ds
           $$
    \be\label{est5 p5}
           +  108 \eta^{-1}   \int_ { \tau}^ {t }
      e^{2\eta s}
       \| \theta (s) \|^4_{L^2(\R^n, l^2)} ds
       + 1728\eta^{-2} \|\kappa\|^4 \| \beta \|^4_{l^2}
       e^{2\eta t} .
       \ee

 It follows from
 \eqref{est5 p2}-\eqref{est5 p5}
 that for all $t\ge \tau$,
 $$
     \E \left (
     e^{2\eta (t\wedge \tau_m)} \| u   (t\wedge \tau_m) \|^4
     \right )
          + 2 (2 \lambda-\eta)
          \E \left (
                  \int_ { \tau}^ {t\wedge \tau_m} e^{ 2 \eta s}
                  \|  u  ( s )  \|^4 ds
                  \right )
                  $$
        $$\le
         \mathbb{E} \left (e^{2\eta \tau} \| u_ { \tau} \|^4
         \right )
  + 12 \eta^{-1}
       \int_ { \tau} ^ {t }
      e^{ 2 \eta s}
      \| \phi_1 (s)\|^2_{L^1
      (\R^n)}     ds
   $$
      $$
      +
       {\frac 32} \eta^{-3}      \int_ { \tau}^ {t }
      e^{2 \eta s}
     \| \phi_g (s)\|^4  ds
             +
                108 \eta^{-1}   \int_ { \tau}^ {t }
      e^{2\eta s}
       \| \theta (s) \|^4_{L^2(\R^n, l^2)} ds
       + 1728\eta^{-2} \|\kappa\|^4 \| \beta \|^4_{l^2}
       e^{2\eta t} .
      $$
      $$
      +
      2
      \int_ { \tau} ^ {t }
      e^{2 \eta s}
      \left (  2\| \phi_1 (s)\|_{L^\infty
      (\R^n)}
      + 2 \| \psi_1 \|_{L^1
      (\R^n)}
      +{\frac 16} \eta
      \right )
      \mathbb{E}
      \left (
      \| u(s) \|^4
      \right )   ds
 $$
  $$
     +
           2    \int_ { \tau}^ {t }
      e^{2\eta s}
      \left (  \eta + \| \psi_g \|_{L^1 (\R^n)}
      + 2
     \| \phi_7 (s )\|_{L^\infty (\R^n)}
     +\| \psi_g   \|_{L^\infty (\R^n)}
          \right )
        \mathbb{E}
                     \left ( \| u(s)\|^4
                     \right )
            ds
            $$
     \be\label{est5 p6}
         +
         \left({\frac 12}\eta
         + 48 \|\kappa\|^2 \| \beta \|^2_{l^2}
         +
         24 \|\kappa\|^2_{L^\infty(\R^n)}
          \| \gamma \|^2_{l^2}
          \right )
                      \int_ { \tau}^ {t }
      e^{2 \eta s}             \mathbb{E}
                     \left (
          \| u(s) \|^4
          \right ) ds .
        \ee
      Taking the limit of \eqref{est5 p6}
      as $m \to \infty$, by Fatou's lemma we obtain
      for all $t\ge \tau$,
 $$
     \E \left (  \| u   (t ) \|^4
     \right )
          + 2 (2 \lambda-\eta)
          \E \left (
                  \int_ { \tau}^ {t } e^{ 2 \eta (s-t) }
                  \|  u  ( s )  \|^4 ds
                  \right )
                  $$
        $$\le
         \mathbb{E} \left (e^{2\eta (\tau-t) } \| u_ { \tau} \|^4
         \right )
  + 12 \eta^{-1}
       \int_ { \tau} ^ {t }
      e^{ 2 \eta  (s-t)}
      \| \phi_1 (s)\|^2_{L^1
      (\R^n)}     ds
   $$
      $$
      +
       {\frac 32} \eta^{-3}      \int_ { \tau}^ {t }
      e^{2 \eta  (s-t) }
     \| \phi_g (s)\|^4  ds
             +
                108 \eta^{-1}   \int_ { \tau}^ {t }
      e^{2\eta  (s-t)}
       \| \theta (s) \|^4_{L^2(\R^n, l^2)} ds
       + 1728\eta^{-2} \|\kappa\|^4 \| \beta \|^4_{l^2}
      $$
      $$
      +
      2
      \int_ { \tau} ^ {t }
      e^{2 \eta  (s-t)}
      \left (  2\| \phi_1 (s)\|_{L^\infty
      (\R^n)}
      + 2 \| \psi_1 \|_{L^1
      (\R^n)}
      +{\frac 16} \eta
      \right )
      \mathbb{E}
      \left (
      \| u(s) \|^4
      \right )   ds
 $$
  $$
     +
           2    \int_ { \tau}^ {t }
      e^{2\eta  (s-t)}
      \left (  \eta + \| \psi_g \|_{L^1 (\R^n)}
      + 2
     \| \phi_7 (s )\|_{L^\infty (\R^n)}
     +\| \psi_g   \|_{L^\infty (\R^n)}
          \right )
        \mathbb{E}
                     \left ( \| u(s)\|^4
                     \right )
            ds
            $$
     \be\label{est5 p6a}
         +
         \left({\frac 12}\eta
         + 48 \|\kappa\|^2 \| \beta \|^2_{l^2}
         +
         24 \|\kappa\|^2_{L^\infty(\R^n)}
          \| \gamma \|^2_{l^2}
          \right )
                      \int_ { \tau}^ {t }
      e^{2 \eta  (s-t)}             \mathbb{E}
                     \left (
          \| u(s) \|^4
          \right ) ds .
        \ee
      By \eqref{etac} and
      \eqref{est5 p6a} we get
         for all $t\ge \tau$,
         $$
     \E \left (  \| u   (t,\tau, u_\tau ) \|^4
     \right )
          \le
         \mathbb{E} \left (e^{2\eta (\tau-t) } \| u_ { \tau} \|^4
         \right )
  + 12 \eta^{-1}
       \int_ { \tau} ^ {t }
      e^{ 2 \eta  (s-t)}
      \| \phi_1 (s)\|^2_{L^1
      (\R^n)}     ds
   $$
     \be\label{est5 p7}
       +
       {\frac 32} \eta^{-3}      \int_ { \tau}^ {t }
      e^{2 \eta  (s-t) }
     \| \phi_g (s)\|^4  ds
             +
                108 \eta^{-1}   \int_ { \tau}^ {t }
      e^{2\eta  (s-t)}
       \| \theta (s) \|^4_{L^2(\R^n, l^2)} ds
       + 1728\eta^{-2} \|\kappa\|^4 \| \beta \|^4_{l^2} .
   \ee

        Replacing $\tau$ and $t$
   in \eqref{est5 p7} by $\tau -t$
   and  $\tau$, respectively ,
         we get
            for all $t\ge 0$,
         $$
     \E \left (  \| u   (\tau,\tau-t, u_{\tau -t} ) \|^4
     \right )
          \le
       e^{-2\eta t   }  \mathbb{E} \left (
          \| u_ { \tau -t} \|^4
         \right )
  + 12 \eta^{-1}
       \int_ { \tau -t} ^ {\tau }
      e^{ 2 \eta  (s-\tau)}
      \| \phi_1 (s)\|^2_{L^1
      (\R^n)}     ds
   $$
  $$
       +
       {\frac 32} \eta^{-3}       \int_ { \tau -t} ^ {\tau }
      e^{2 \eta  (s-\tau) }
     \| \phi_g (s)\|^4  ds
             +
                108 \eta^{-1}    \int_ { \tau -t} ^ {\tau }
      e^{2\eta  (s-\tau)}
       \| \theta (s) \|^4_{L^2(\R^n, l^2)} ds
       $$
          \be\label{est5 p8}
       + 1728\eta^{-2} \|\kappa\|^4 \| \beta \|^4_{l^2} .
   \ee

   Since
  $ \mathcal{L}_{ u_{\tau-t}  }  \in   D ( \tau-t )$
  and $D\in \mathcal{D}$, we have
 $$
\lim_{t \to \infty}
 e^{-2 \eta t  }
         \mathbb{E} \left ( \| u_ { \tau -t} \|^4
         \right )
           \le  e^{-2 \eta  \tau }
          \lim_{t \to \infty}
 e^{2 \eta (\tau-t) }\| D(\tau -t)\|^4_{\mathcal{P}_4 ( H ) } =0,
         $$
       and hence
  there exists $T=T(\tau, D ) >0 $ such that for all $t\ge T$,
  $$
 e^{-2 \eta t  }
         \mathbb{E} \left ( \| u_ { \tau -t} \|^4
         \right )      \le 1,
         $$
         which along with
         \eqref{est5 p8}
         implies that
         for all $t\ge T$,
         $$
     \E \left (  \| u   (\tau,\tau-t, u_{\tau -t} ) \|^4
     \right )
          \le
      1+
6\eta^{-2}
      \| \phi_1  \|^2_{L^\infty(\R, L^1
      (\R^n)) }
       + 1728\eta^{-2} \|\kappa\|^4 \| \beta \|^4_{l^2}
   $$
  $$
       +
       {\frac 32} \eta^{-3}       \int_ { -\infty } ^ {\tau }
      e^{2 \eta  (s-\tau) }
     \| \phi_g (s)\|^4  ds
             +
                108 \eta^{-1}    \int_ { -\infty } ^ {\tau }
      e^{2\eta  (s-\tau)}
       \| \theta (s) \|^4_{L^2(\R^n, l^2)} ds,
       $$
       which  completes the proof.
    \end{proof}

\section{Existence of Pullback Measure Attractors}

In this section, we
prove  the existence and uniqueness of $\mathcal{D}$-pullback measure attractors
of \eqref{sde3}-\eqref{sde4}
in $ \mathcal{P}_4 ( H ) $. To that end,
we first define a non-autonomous dynamical system
in
$ \mathcal{P}_4 ( H ) $.

Given $   \tau, t \in \R$
with $t\ge \tau$, define
$P^*_{\tau, t}: {\mathcal{P}}_4 (H)
\to {\mathcal{P}}_4 (H)$ by, for every
$\mu\in {\mathcal{P}}_4 (H)$,
\be\label{md1}
P^*_{\tau, t} \mu
={\mathcal{L}}_{u(t, \tau, u_\tau)},
\ee
where
$u(t, \tau, u_\tau)$
is the solution of \eqref{sde3}-\eqref{sde4}
with
$u_\tau\in L^4(\Omega,
\mathcal{F}_\tau,
H)$
such that  ${\mathcal{L}}_{u_\tau}
=\mu$.
In terms of \eqref{md1},
for every $t\in \R^+$
and $\tau \in \R$, define
$\Phi (t, \tau): {\mathcal{P}}_4 (H)
\to {\mathcal{P}}_4 (H)$ by, for every
$\mu\in {\mathcal{P}}_4 (H)$,
\be\label{md2}
 \Phi (t,\tau) \mu=
 P^*_{\tau,  \tau +t} \mu.
\ee
By the uniqueness of solutions
of \eqref{sde3}-\eqref{sde4}, we have,
for all $t,s\in \R^+$, $\tau \in \R$
and $u_\tau \in L^4(\Omega, \mathcal{F}_\tau,
H)$,
$$
u(t+s+\tau, \tau, u_\tau)
= u(t+s+\tau, s+\tau, u(s+\tau, \tau, u_\tau) ),
$$
and hence
for
 all $t,s\in \R^+$, $\tau \in \R$
and $\mu \in  {\mathcal{P}}_4(H)$,
$$
\Phi (t+s, \tau)\mu =
\Phi (t, s+\tau) ( \Phi(s, \tau) \mu).
$$

To prove $\Phi$ is a non-autonomous
dynamical system on
$({\mathcal{P}}_4(H), d_{{\mathcal{P}}(H) })$, it remains
to show $\Phi$ is weakly continuous
over bounded subsets of ${\mathcal{P}}_4(H)$,
which is given below.

\begin{lem}\label{ma1}
Suppose  {\bf  (H1)}-{\bf (H4)}
hold.  Let
$u_0, u_{0,n}  \in
  L^4(\Omega, {\mathcal{F}}_\tau,
H)$ such that
$ \E \left (\|u_0\|^4 \right )
\le R$ and
$ \E \left (\|u_{0,n} \|^4 \right )
\le R$
  for  some $R>0$.
If $\mathcal{L}_{u_{0,n}}
 \to \mathcal{L}_{u_0}$   weakly,   then
 for every $ \tau \in \R $
 and $t\ge \tau$,
 $\mathcal{L}_{u(t,\tau, u_{0,n})}
 \to \mathcal{L}_{u(t,\tau, u_0) }$ weakly.
 \end{lem}

\begin{proof}
Since $\mathcal{L}_{u_{0,n}}
 \to \mathcal{L}_{u_0}$ weakly,
 by the Skorokhov theorem,
 there exist a probability space
 $(\widetilde{\Omega}, \widetilde{\mathcal{F}},
 \widetilde{\mathbb{P}} )$
 and random variables
 $ \widetilde{u}_0 $
 and
 $ \widetilde{u}_{0,n} $
 defined in
 $(\widetilde{\Omega}, \widetilde{\mathcal{F}},
 \widetilde{\mathbb{P}} )$
 such that
 the distributions of
 $ \widetilde{u}_0 $
 and
 $ \widetilde{u}_{0,n} $
 coincide with that
 of $ {u}_0 $
 and
 $ {u}_{0,n} $, respectively.
 Furthermore,
 $ \widetilde{u}_{0,n}
 \to  \widetilde{u}_0 $
 $\widetilde{\mathbb{P}}$-almost surely.
Note that
 $  \widetilde{u}_{0} ,
 \widetilde{u}_{0,n} $
 and $W$
 can be considered as
 random variables defined
 in the  product space
 $(\Omega \times \widetilde{\Omega},
 \mathcal{F} \times \widetilde{\mathbb{F}},
  {\mathbb{P}} \times \widetilde{\mathbb{P}})$.
  So we may consider the
   solutions of
  the stochastic equation
  in  the product space with initial data
  $  \widetilde{u}_{0}$ and
 $\widetilde{u}_{0,n} $, instead of the solutions
 in  $(\Omega  ,
 \mathcal{F}  ,
  {\mathbb{P}}  )$
  with initial data
  $   {u}_{0}$ and
 $ {u}_{0,n} $.
 However, for simplicity,
 we will not distinguish
 the new random variables from the original
 ones, and just consider the solutions of the equation
 in the original space.
  Since
  $\widetilde{u}_{0,n}
  \to \widetilde{u}_{0 } $
  $({\mathbb{P}} \times \widetilde{\mathbb{P}})$-almost
  surely,
  without loss of generality,
  we   simply
  assume that
  $ {u}_{0,n}
  \to {u}_{0 } $
  ${\mathbb{P}}$-almost
  surely.

 Let
  $ u_n (t,\tau) = u  ( t,\tau,  u_{0,n} ) $,
   $ u (t,\tau) = u  ( t,\tau,  u_{0} ) $ and
  $ v_n (t,\tau) = u  ( t,\tau,  u_{0,n} )
 - u  ( t,\tau, u_0 ) $.
 Then by \eqref{sde3}-\eqref{sde4} we have,
  for all $ t  \ge \tau $,
$$
dv_n (t) -  \Delta v_n (t) dt  +  \lambda  v_n (t) dt
+
 f  \left( t,   u_n (t),
 \mathcal{L}_{u_n (t) }  \right)
         -  f \left( t,   u ( t )    \mathcal{L}_{u  (t) }  \right)
     dt
$$
$$
=
g \left( t,   u_n (t),
 \mathcal{L}_{u_n (t) }  \right)
         - g \left( t,   u ( t )    \mathcal{L}_{u  (t) }  \right)
     dt
     +
     \left (
      \sigma   ( t,  u_n( t ) , \mathcal{L}_{u_n ( t )}  )
      -
      \sigma   ( t,  u ( t ) , \mathcal{L}_{u  ( t )}  )
      \right )
         dW(t).
         $$

By It\^{o}'s formula, we have for all $ t \ge \tau $,
$$
 \| v_n ( t )  \|^2
+ 2 \int_\tau^t  \| \nabla v_n (s)  \|^2  ds
  +  2 \lambda   \int_\tau^t   \|  v_n (s) \|^2  ds
  $$
  $$
  +2
     \int_\tau^t
        \int_{\R^n}
        \left (
         f  \left( s,   u_n ( s ) ,  \mathcal{L}_{u_n ( s )}  \right)
           -   f  \left( s,   u ( s ) ,  \mathcal{L}_{u ( s )}  \right)
               \right)
             v_n(s) dx
          ds
$$
$$
 =  \| u_{0,n} -u_0 \|^2
 +2
     \int_\tau^t
        \int_{\R^n}
        \left (
        g \left( s,   u_n ( s ) ,  \mathcal{L}_{u_n ( s )}  \right)
           -   g \left( s,   u ( s ) ,  \mathcal{L}_{u ( s )}  \right)
               \right)
             v_n(s) dx
          ds
          $$
          $$
 +
  \int_\tau^t
        \|  \sigma   ( s,  u_n ( s ) , \mathcal{L}_{u_n ( s )}  )
           - \sigma  ( s, u ( s ) ,
           \mathcal{L}_{u  ( s )}   )
         \|_{ L_2 ( \ell^2,  H ) }  ^2
     ds
$$
 \be \label{ma1 p1}
  + 2 \int_\tau^t
        \left ( v_n (s),
         \left ( \sigma   ( s ,
         u_n ( s ) , \mathcal{L}_{u_n ( s )}   )
           - \sigma  ( s, u ( s ) ,
           \mathcal{L}_{u  ( s )}    )
           \right )  dW \right ).
  \ee
  By \eqref{f3} and \eqref{f4} we  have
     $$
  -2
     \int_\tau^t
        \int_{\R^n}
        \left (
         f  \left( s,   u_n ( s ) ,  \mathcal{L}_{u_n ( s )}  \right)
           -   f  \left( s,   u ( s ) ,  \mathcal{L}_{u ( s )}  \right)
               \right)
             v_n (s) dx
          ds
$$
  $$
  =
  -2
     \int_\tau^t
        \int_{\R^n}
        \left (
         f  \left( s,   u_n ( s ) ,  \mathcal{L}_{u_n ( s )}  \right)
           -   f  \left( s,   u ( s ) ,  \mathcal{L}_{u_n ( s )}  \right)
               \right)
             v_n(s) dx
          ds
$$
 $$
  -2
     \int_\tau^t
        \int_{\R^n}
        \left (
         f  \left( s,   u ( s ) ,  \mathcal{L}_{u_n ( s )}  \right)
           -   f  \left( s,   u ( s ) ,  \mathcal{L}_{u ( s )}  \right)
               \right)
             v_n(s) dx
          ds
$$
 \be\label{ma1 p2}
  \le
  \int_\tau ^t
  \left (
  (2 \| \phi_4 (s)\|_{L^\infty (\R^n)}
  + \|\phi_3 (s)\|_{L^\infty (\R^n)}
  ) \| v_n (s) \|^2
  +\E \left (
  \| v_n (s) \|^2 \right )
  \| \phi_3 (s) \|_{L^1(\R^n)}
  \right )
  ds.
\ee
By \eqref{g2} we have
$$2
     \int_\tau^t
        \int_{\R^n}
        \left (
        g \left( s,   u_n ( s ) ,  \mathcal{L}_{u_n ( s )}  \right)
           -   g \left( s,   u ( s ) ,  \mathcal{L}_{u ( s )}  \right)
               \right)
             v(s) dx
          ds
          $$
 \be\label{ma1 p3}
  \le
  \int_\tau^t
  \left (
  3\| \phi_7 (s) \|_{L^\infty(\R^n)} \| v_n (s) \|^2
  +
      \E \left (
  \| v_n (s) \|^2 \right )
  \| \phi_7 (s) \|_{L^1(\R^n)}
  \right )
  ds.
\ee
By \eqref{s5} we have
    $$
   \int_\tau^t
        \|  \sigma   ( s,  u_n ( s ) , \mathcal{L}_{u_n ( s )}  )
           - \sigma  ( s, u ( s ) ,
           \mathcal{L}_{u  ( s )}   )
         \|_{ L_2 ( \ell^2,  H ) }  ^2
     ds
$$
\be\label{ma1 p4}
\le
2\|L_\sigma \|^2_{l^2}\int_\tau^t
\left (
\|\kappa\|^2_{L^\infty (\R^n)}
\|v_n (s) \|^2
+ \| \kappa\|^2
  \E \left (
  \| v_n (s) \|^2 \right )
\right )ds.
\ee

It follows from
\eqref{ma1 p1}-\eqref{ma1 p4}
that
for all $ t \ge \tau $,
$$
 \E \left (  \| v_n ( t )  \|^2
 \right )
 \le
  \E \left (\| u_{0,1} -u_{0,2} \|^2
  \right )
  $$
  $$
  +
  \left (
  2\|\phi_4 \|_{L^\infty(\R, L^\infty(\R^n))}
  +\|\phi_3 \|_{L^\infty(\R, L^\infty(\R^n)\cap
  L^1(\R^n)  )}
  +3 \|\phi_7 \|_{L^\infty(\R, L^\infty(\R^n)
  \cap L^1(\R^n) )}
  \right ) \int_\tau^t
   \E \left (
  \| v_n (s) \|^2 \right )
 ds
          $$
\be\label{ma1 p5}
          +
           2\|L_\sigma \|^2_{l^2}
  \left (\|\kappa\|^2_{L^\infty (\R^n)}
  + \| \kappa \|^2
  \right )\int_\tau^t
   \E \left (
  \| v_n (s) \|^2 \right ) ds.
  \ee
        By \eqref{ma1 p5}
        and  Gronwall's  lemma, we obtain,
        for  all $t\ge \tau$,
         \be\label{ma1 p6}
 \E \left ( \| v_n ( t, \tau )  \|^2
 \right )
 \le
  \E \left ( \| u_{0,n} -u_0 \|^2
  \right ) e^{c_1 (t-\tau)},
\ee
where $c_1>0$ is a constant independent of
$n, \tau$  and $t$.

Since $\E \left ( \|u_{0,n}\|^4 \right )
\le R$, we see that
the sequence
$\{ u_{0,n} \}_{n=1}^\infty$
is uniformly integrable in $L^2(\Omega, H)$.
Then  using the  assumption that
$u_{0,n} \to u_0$ $P$-almost surely,
we obtain from Vitali's theorem that
 $  u_{0,n}  \to u_0 $ in   $L^2(\Omega, H)$,
 which along with
  \eqref{ma1 p6}
  shows that
  $u(t,\tau, u_{0,n})
  \to u(t,\tau, u_0)$
  in $L^2(\Omega, H)$ and hence
  also in distribution.
\end{proof}

By Lemma \ref{ma1}, we find that the mapping
$\Phi$ given by \eqref{md2}
is a non-autonomous dynamical system
in $(\mathcal{P}_4 (H), d_{\mathcal{P} (H)})$
which is continuous over
bounded subsets of
$(\mathcal{P}_4 (H)$.
Next, we prove the existence
of $\mathcal{D}$-pullback absorbing sets
of $\Phi$.

\begin{lem}\label{ma2}
If  {\bf  (H1)}-{\bf (H4)},
\eqref{lamc} and \eqref{phi_ga}
hold, then $\Phi$
   has a closed
$ \mathcal{D}$-pullback absorbing set
$ K  = \{    K  ( \tau ): \tau \in \mathbb{R}    \}
\in \mathcal{D} $
which is given by, for each $\tau\in\mathbb{R}$,
\be \label{ma2 1}
K  ( \tau ) =
\left \{
\mu \in {\mathcal{P}}_4 (H):
\int_H \| \xi\|^4 \mu (d \xi)
\le M_4 (\tau)
\right \},
\ee
where
 $$
 M_4(\tau)
 =   M_3
     + M_3  \int_{ -\infty }^\tau
     e^{  2\eta  ( s - \tau ) }
  \left (
     \| \phi_g (s)\|^4     +
       \| \theta (s) \|^4_{L^2(\R^n, l^2)}
       \right )  ds,
$$
and
  $\eta>0$ is the same number as in \eqref{etac},
$ M_3>0 $  is  the same   constant
as in Lemma \ref{est5}
which is  independent of $ \tau $.
\end{lem}

\begin{proof}
First, note that $K(\tau)$ is a closed subset of
$\mathcal{P}_4 (H)$. On the other hand,
for every  $ \tau \in \mathbb{R} $
 and $ D \in \mathcal{D} $,
 by   \eqref{md2}
and Lemma \ref{est5}, we infer that
there exists $ T = T ( \tau, D ) >0$
such that for all $ t \ge T $,
\be\label{ma2 p1}
\Phi ( t, \tau - t )
D ( \tau - t )
\subseteq
K  ( \tau ).
\ee
Finally, by \eqref{phi_ga}
and \eqref{ma2 1} we have,
as $\tau \to -\infty$,
$$
  e^{ 2\eta  \tau }
     \| K  ( \tau )  \|_{   \mathcal{P}_4  (H )   }^4
  =
  e^{ 2\eta  \tau } M_3
     + M_3
       \int_{ -\infty }^\tau    e^{  2\eta s   }
       \left (
     \| \phi_g (s)\|^4     +
       \| \theta (s) \|^4_{L^2(\R^n, l^2)}
       \right )  ds
     \to 0,
       $$
      which shows that $K=\{K(\tau):
       \tau \in \R\} \in \mathcal{D}$.
    Therefore,  by \eqref{ma2 p1} we find that
    $K$ as given by \eqref{ma2 1} is
    a closed
    $\mathcal{D}$-pullback absorbing set
    for $\Phi$.
\end{proof}

Next, we establish the
$\mathcal{D}$-pullback asymptotic compactness of
$\Phi$.

\begin{lem}
\label{ma3}
If  {\bf  (H1)}-{\bf (H4)},
\eqref{lamc} and \eqref{phi_g}-\eqref{phi_ga}
hold, then $\Phi$
 is $\mathcal{D}$-pullback asymptotically compact in
  $  (\mathcal{P}_4 ( H), d_{\mathcal{P} (H)}) $.
\end{lem}

\begin{proof}
Given  $ \tau \in \mathbb{R} $, $ D  \in  \mathcal{D} $,
 $ t_n \to  \infty $
 and $\mu_n   \in D ( \tau - t_n )  $,
 we will prove
 the sequence
 $\{\Phi (t_n,
 \tau -t_n  )\mu_n \}_{n=1}^\infty$
 has a convergent
 subsequence  in
 $  (\mathcal{P}_4 ( H), d_{\mathcal{P} (H)}) $.

Given  $ v_n  \in L^4  ( \Omega,
\mathcal{F}_{ \tau - t_n }, H ) $
with $ \mathcal{L}_{ v_n } =\mu_n $,
we
consider the solution
$u(\tau, \tau-t_n, v_n)$ of
\eqref{sde3} with initial data
$v_n$ at initial time $\tau -t_n$.
  we first prove
  the distributions
  of
  $\{u(\tau, \tau-t_n, v_n)\}_{n=1}^\infty$
  are
     tight in $H$.

     Let $\rho: \R^n \to [0,1]$ be the smooth cut-off
     function given by \eqref{rho},
     and $\rho_m (x) = \rho \left ({\frac xm}
     \right )$
     for every $m\in \N$ and $x\in \R^n$.
     Then  the solution $u$ can be decomposed as:
     $$
     u(\tau, \tau-t_n, v_n)
     =
     \rho_m u(\tau, \tau-t_n, v_n)
     +(1-\rho_m) u(\tau, \tau-t_n, v_n).
     $$
     Note that there exists $c_1=c_1 (\rho)>0$ such that
     $$
     \sup_{x\in \R^n}
     |\nabla \rho (x)| \le c_1,
     $$
   and hence for all $m, n\in \N$,
   \be\label{ma3 p1}
   \|(1- \rho_m) u(\tau, \tau-t_n, v_n)\|^2_{H^1(\R^n)}
   \le c_2\| u(\tau, \tau-t_n, v_n)\|^2_{H^1(\R^n)},
   \ee
 where $c_2= 1+ c_1^2  $.

  By Lemma \ref{est3} we see that
  there exists $N_1=N_1 (\tau, D)\in \N$
  such that for all $n\ge N_1$,
\be\label{ma3 p2}
\mathbb{E}
\left (
 \|  u   ( \tau, \tau - t_n, v_n )
                   \|_{ H^1 (\R^n) } ^2
         \right )
\le  c_3  ,
\ee
 $ c_3=c_3  ( \tau )>0 $ is a constant depending
only  on $ \tau $, but not on $n$ or  $D$.
By \eqref{ma3 p1} and \eqref{ma3 p2} we have
for all $m\in \N$ and $n\ge N_1$,
 \be\label{ma3 p3}
  \E \left (
   \| (1-\rho_m ) u(\tau, \tau-t_n, v_n)\|^2_{H^1(\R^n)}
   \right )
   \le c_2c_3.
    \ee
 By \eqref{ma3 p3} and
 Chebyshev's inequality, we obtain that
  for all $m\in \N$  and
  $n\ge N_1$,
$$
 \mathbb{P}\left (
         \{ \| (1-\rho_m ) u   ( \tau, \tau - t_n, v_n   )
         \|_{H^1(\R^n)} >R
         \}
   \right )
   \le {\frac {c_2c_3}{R^2}}.
   $$
  Therefore,
   for
     every $\delta>0$,
   there exists
   $R(\delta, \tau)>0$ such that
   for all $m\in \N$ and $n\ge N_1$,
   \be\label{ma3 p4}
  \mathbb{P}\left (
         \{ \| (1-\rho_m)  u   ( \tau, \tau - t_n,  v_n   )
         \|_{H^1(\R^n)} >R(\delta, \tau)
         \}
   \right )
 <   \delta.
   \ee
   Let
   $$Z_\delta = \{
   u\in H^1(\R^n):  \| u \|_{H^1(
   \R^n)}
   \le R(\delta, \tau);
   \  \ u(x) =0 \ \text{for a.e. } |x|>m
   \}.$$
   Then $Z_\delta$ is a compact subset of $H$.
   By \eqref{ma3 p4} we have
    for all $m\in \N$ and $n\ge N_1$,
\be\label{ma3 p5}
 \mathbb{P}
 \left (
 \left \{
 (1-\rho_m)  u   ( \tau, \tau - t_n,  v_n   )
 \in Z_\delta
 \right \}
 \right ) > 1-\delta.
\ee
Since $\delta>0$ is arbitrary, by \eqref{ma3 p5}
we find that for every
$m\in \N$,
\be\label{ma3 p6}
\{
\mathcal{L}_{
(1-\rho_m)  u   ( \tau, \tau - t_n,  v_n   )}
\}_{n=1}^\infty
\ \text{is   tight in } H,
\ee
where
 $\mathcal{L}_{
(1-\rho_m)  u   ( \tau, \tau - t_n,  v_n   )}$
is the distribution of
$
(1-\rho_m)  u   ( \tau, \tau - t_n,  v_n   ) $ in $H$.

Next, we prove
$\{
\mathcal{L}_{
   u   ( \tau, \tau - t_n,  v_n   )}
\}_{n=1}^\infty
$ is   tight in  $ H$ by the uniform tail-ends estimates.
Indeed, by Lemma \ref{est4} we infer that
for every $\varepsilon>0$,
there exist
$N_2=N_2(\varepsilon, \tau, D)\in \N$ and
$m_0=m_0(\eps, \tau)\in \N$ such that
for all $n\ge N_2$,
\be\label{ma3 p10}
\E
\left (
\int_{|x|>{\frac 12} m_0} |u(\tau, \tau -t_n, v_n) (x)|^2 dx
\right )
<{\frac 19} \eps^2.
\ee
By \eqref{ma3 p10} we get,
for all $n\ge N_2$,
\be\label{ma3 p11}
\E
\left (
\| \rho_{m_0}  u(\tau, \tau -t_n, v_n)\|^2
\right )
<{\frac 19} \eps^2.
\ee
By \eqref{ma3 p6}, we  see that
$\{
\mathcal{L}_{
(1-\rho_{m_0})  u   ( \tau, \tau - t_n,  v_n   )}
\}_{n=N_2}^\infty$
 is   tight in  $H$, and hence
 there exist
 $n_1,\cdots , n_l \ge N_2 $ such that
 \be\label{ma3 p12}
 \{
\mathcal{L}_{
 (1-\rho_{m_0})  u   ( \tau, \tau - t_n,  v_n   )}
\}_{n=N_2}^\infty
\subseteq
\bigcup_{j=1}^{l}
 B\left (\mathcal{L}_{
 (1-\rho_{m_0})  u   ( \tau, \tau - t_{n_j},  v_{n_j}   )},
 \ {\frac 13} \eps
 \right ),
 \ee
where
$
B\left (\mathcal{L}_{
 (1-\rho_{m_0})  u   ( \tau, \tau - t_{n_j},  v_{n_j}   )},
 \ {\frac 13} \eps
 \right )
 $ is the
 ${\frac 13} \eps$-neighborhood
 of $ \mathcal{L}_{
 (1-\rho_{m_0})  u   ( \tau, \tau - t_{n_j},  v_{n_j}   )}$
 in the space
 $(\mathcal{P}_4 (H), d_{\mathcal{P} (H)})$.
  We claim:
  \be\label{ma3 p13}
 \{
\mathcal{L}_{    u   ( \tau, \tau - t_n,  v_n   )}
\}_{n=N_2}^\infty
\subseteq
\bigcup_{j=1}^{l}
 B\left (  \mathcal{L}_{
  u   ( \tau, \tau - t_{n_j},  v_{n_j}   )},
 \   \eps
 \right ).
 \ee

Given $n\ge N_2$ by
\eqref{ma3 p12} we know that
there exist $j\in \{1,\cdots, l\}$ such that
  \be\label{ma3 p14}
 \mathcal{L}_{
 (1-\rho_{m_0})  u   ( \tau, \tau - t_n,  v_n   )}
 \in
 B\left (\mathcal{L}_{
 (1-\rho_{m_0})  u   ( \tau, \tau - t_{n_j},  v_{n_j}   )},
 \ {\frac 13} \eps
 \right ).
 \ee
 By \eqref{ma3 p11} and \eqref{ma3 p14}
 we have
 $$
 d_{\mathcal{P}(H) }
 \left (
 \mathcal{L}_{   u   ( \tau, \tau - t_n,  v_n   )},
 \
 \mathcal{L}_{  u   ( \tau, \tau - t_{n_j},  v_{n_j}   )}
 \right )
 $$
 $$
 =
 \sup_{\psi \in L_b(H), \|\psi\|_{L_b}
 \le 1}
 \left |
 \int_H
 \psi   d \mathcal{L}_{   u   ( \tau, \tau - t_n,  v_n   )}
 -\int_H
 \psi  d\mathcal{L}_{  u   ( \tau, \tau - t_{n_j},  v_{n_j}   )}
 \right |
 $$
 $$
 =
 \sup_{\psi \in L_b(H), \|\psi\|_{L_b}
 \le 1}
 \left |
 \E \left (
 \psi    (  u   ( \tau, \tau - t_n,  v_n   ) )
 \right )
 -\E \left (
 \psi  (  u   ( \tau, \tau - t_{n_j},  v_{n_j}   ))
 \right )
 \right |
 $$
 $$
 \le
 \sup_{\psi \in L_b(H), \|\psi\|_{L_b}
 \le 1}
 \left |
 \E \left (
 \psi    (  u   ( \tau, \tau - t_n,  v_n   ) )
 \right )
 -\E \left (
 \psi  ( (1-\rho_{m_0} ) u   ( \tau, \tau - t_{n},  v_{n}   ))
 \right )
 \right |
 $$
 $$
+
 \sup_{\psi \in L_b(H), \|\psi\|_{L_b}
 \le 1}
 \left |
 \E \left (
 \psi    ( (1-\rho_ {m_0})  u   ( \tau, \tau - t_n,  v_n   ) )
 \right )
 -\E \left (
 \psi  ( (1-\rho_ {m_0}) u   ( \tau, \tau - t_{n_j},  v_{n_j}   ))
 \right )
 \right |
 $$
 $$
+
 \sup_{\psi \in L_b(H), \|\psi\|_{L_b}
 \le 1}
 \left |
 \E \left (
 \psi  ( (1-\rho_ {m_0}) u   ( \tau, \tau - t_{n_j},  v_{n_j}   ))
 \right )
 -
 \E \left (
 \psi  (   u   ( \tau, \tau - t_{n_j},  v_{n_j}   ))
 \right )
 \right |
 $$
 $$
 \le
  \E \left (
   \| \rho_ {m_0}  u   ( \tau, \tau - t_{n},  v_{n}   )\|
 \right )
 +\E \left (
   \| \rho_ {m_0}  u   ( \tau, \tau - t_{n_j},  v_{n_j}   )\|
 \right )
 $$
 $$
+
d_{\mathcal{P} (H)}
     \left (   \mathcal{L}_{
      (1-\rho_{m_0} )  u   ( \tau, \tau - t_n,  v_n   ) )}
, \  \mathcal{L}_{
(1-\rho_{m_0}
 ) u   ( \tau, \tau - t_{n_j},  v_{n_j}   )}
 \right )
 $$
 $$
 <{\frac 13} \eps
 + {\frac 13} \eps +
 {\frac 13} \eps =\eps,
 $$
  which yields \eqref{ma3 p13}.
   Since $\eps>0$ is arbitrary,
   by \eqref{ma3 p13} we see that
   the sequence
    $\{ \mathcal{L}_{
       u   ( \tau, \tau - t_n,  v_n   )  }
       \}_{n=1}^\infty$
 is tight in $ {\mathcal{P} (H)} $,
 which implies that
 there exists $\nu \in \mathcal{P} (H)$ such that,
 up to a subsequence,
 \be\label{ma3 p20}
 \mathcal{L}_{
       u   ( \tau, \tau - t_n,  v_n   )  }
 \to \nu \ \text{weakly}.
 \ee

 It remains to show $\nu \in \mathcal{P}_4 (H)$.
 Let $K=\{K(\tau): \tau \in \R \}$
 be the closed $\cald$-pullback
 absorbing set of $\Phi$ given by \eqref{ma2 1}.
 Then there exists
 $N_3=N_3 (\tau, D)\in \N$ such that for all
 $n\ge N_3$,
 \be\label{ma3 p21}
 \mathcal{L}_{  u(\tau, \tau -t_n, v_n)}
 \in  K(\tau).
\ee
 Since $K(\tau)$ is closed
 with respect to the weak topology
 of  $ \mathcal{P} (H)$,
 by \eqref{ma3 p20}-\eqref{ma3 p21}
 we obtain $\nu \in K(\tau)$ and thus
 $\nu \in \mathcal{P}_4 (H)$.
 This completes the proof.
\end{proof}

We now prove   the existence and uniqueness of $D$-pullback measure attractors of
$\Phi$.

\begin{thm}\label{main_e}
If  {\bf  (H1)}-{\bf (H4)},
\eqref{lamc} and \eqref{phi_g}-\eqref{phi_ga}
hold, then $\Phi$
    has a unique $\mathcal{D}$-pullback
    measure attractor
$\mathcal{A}  = \{  \mathcal{A} ( \tau )
: \tau \in \mathbb{R}   \}  \in \mathcal{D}$.

Moreover,
If $f, g,  \sigma$, $\theta$ and $\phi_g$
are  all periodic functions in time with period $T$,
then   so is  the measure attractor
 $\mathcal{A} $; that is, for all
 $\tau \in \R$,
 $  \mathcal{A} ( \tau  +T)
 =
 \mathcal{A} ( \tau  ) $.
 \end{thm}

\begin{proof}
The existence
and uniqueness
of the measure attractor $\mathcal{A}$
follow from   Proposition \ref{exatt}
based on Lemmas  \ref{ma1},
\ref{ma2}  and
 \ref{ma3}.

 If $f, g$
 and $ \sigma$ are
  periodic functions in time with period $T$,
  then
  $\Phi$ is also periodic with period $T$.
  If  $\theta$ and $\phi_g$ are periodic
  in time with period $T$, then by   \eqref{ma2 1}
  we see the absorbing set
  $K$ is also periodic with period $T$; that is,
  $K(\tau +T) =K(\tau)$
  for all $\tau \in \R$.
  Then the periodicity of $\mathcal{A}$
  follows from Proposition \ref{exatt} again.
   \end{proof}

\section{Singleton attractors,
invariant measures and periodic measures}

In this section, under further conditions, we
prove  the measure attractor of \eqref{sde3}-\eqref{sde4}
is actually a singleton.
Based on this result, we obtain the existence
and uniqueness of invariant measures
and periodic measures  for
the autonomous and
periodic systems, respectively.

In the sequel, we assume
that $\lambda$ satisfies
\eqref{lamc} as well as the following condition:
$$
 \lambda >
   \|\phi_4 \|_{L^\infty(\R, L^\infty(\R^n))}
  +{\frac 12} \|\phi_3 \|_{L^\infty(\R, L^\infty(\R^n)\cap
  L^1(\R^n)  )}
   +{\frac 32} \|\phi_7 \|_{L^\infty(\R, L^\infty(\R^n)
  \cap L^1(\R^n) )}
  $$
  \be\label{lamca}
  +
            \|L_\sigma \|^2_{l^2}
  \left (\|\kappa\|^2_{L^\infty (\R^n)}
  + \| \kappa \|^2
  \right ).
  \ee
By \eqref{lamc} and \eqref{lamca} we find
that there exists a sufficiently small
$\eta\in (0,1)$ such that
both \eqref{etac} and the following inequality
are fulfilled:
 $$
2\lambda  - \eta >
  2\|\phi_4 \|_{L^\infty(\R, L^\infty(\R^n))}
  +\|\phi_3 \|_{L^\infty(\R, L^\infty(\R^n)\cap
  L^1(\R^n)  )}
  +3 \|\phi_7 \|_{L^\infty(\R, L^\infty(\R^n)
  \cap L^1(\R^n) )}
  $$
  \be\label{etaca}
  +
           2\|L_\sigma \|^2_{l^2}
  \left (\|\kappa\|^2_{L^\infty (\R^n)}
  + \| \kappa \|^2
  \right ).
  \ee

 Under  \eqref{lamca}, we prove
 any two solutions of \eqref{sde3}-\eqref{sde4}
 pullback converge to each other
 in  $L^2(\Omega, H)$.

\begin{lem}\label{im1}
Suppose  {\bf  (H1)}-{\bf (H4)}
and \eqref{lamca}
hold.  Then  for every $\tau \in \R$,
$t\in \R^+$ and
$u_{0,1}, u_{0,2}  \in
  L^2(\Omega, {\mathcal{F}}_{\tau-t},
H)$, the solutions of \eqref{sde3} satisfy
     \be\label{im1 1}
 \E \left ( \| u  ( \tau, \tau-t, u_{0,1} )
 -
 u  ( \tau, \tau-t, u_{0,2} ) \|^2
 \right )
 \le e^{- \eta  t }
  \E \left ( \| u_{0,1} -u_{0,2} \|^2
  \right )  ,
\ee
where $\eta>0$ is the same number as in \eqref{etaca}.
   \end{lem}

\begin{proof}
     Let
  $ u_1 (t,\tau) = u  ( t,\tau,  u_{0,1} ) $,
   $ u_2 (t,\tau) = u  ( t,\tau,  u_{0,2} ) $ and
  $ v  (t,\tau) = u  ( t,\tau,  u_{0,1} )
 - u  ( t,\tau, u_{0,2} ) $.
  By \eqref{ma1 p1}
  and It\^{o}'s formula, we have for all $ t \ge \tau $,
$$
 e^{\eta t}
 \| v ( t )  \|^2
+ 2 \int_\tau^t  e^{\eta s}\| \nabla v  (s)  \|^2  ds
  +  (2 \lambda  -\eta)
   \int_\tau^t   e^{\eta s}  \|  v (s) \|^2  ds
  $$
  $$
  +2
     \int_\tau^t  e^{\eta s}
        \int_{\R^n}
        \left (
         f  \left( s,   u_1 ( s ) ,  \mathcal{L}_{u_1 ( s )}  \right)
           -   f  \left( s,   u_2 ( s ) ,  \mathcal{L}_{u_2 ( s )}  \right)
               \right)
             v (s)
          dx ds
$$
$$
 =  e^{\eta \tau} \| u_{0,1} -u_{0,2} \|^2
 +2
     \int_\tau^t  e^{\eta s}
        \int_{\R^n}
        \left (
        g \left( s,   u_1 ( s ) ,  \mathcal{L}_{u_1 ( s )}  \right)
           -   g \left( s,   u_2 ( s ) ,  \mathcal{L}_{u_2 ( s )}  \right)
               \right)
             v (s)
          dxds
          $$
          $$
 +
  \int_\tau^t e^{\eta s}
        \|  \sigma   ( s,  u_1 ( s ) , \mathcal{L}_{u_1 ( s )}  )
           - \sigma  ( s, u_2 ( s ) ,
           \mathcal{L}_{u _2  ( s )}   )
         \|_{ L_2 ( \ell^2,  H ) }  ^2
     ds
$$
 \be \label{im1 p1}
  + 2 \int_\tau^t  e^{\eta s}
        \left ( v (s),
         \left ( \sigma   ( s ,
         u_1 ( s ) , \mathcal{L}_{u_1 ( s )}   )
           - \sigma  ( s, u_2 ( s ) ,
           \mathcal{L}_{u_2  ( s )}    )
           \right )  dW \right ).
  \ee
  By the proof of \eqref{ma1 p5} we obtain,
  for all $ t \ge \tau $,
$$
 \E \left ( e^{\eta t} \| v ( t )  \|^2
 \right )
 + (2\lambda -\eta)
 \int_0^t
 e^{\eta s} \| v ( s)  \|^2 ds
 \le
  \E \left ( e^{\eta \tau}  \| u_{0,1} -u_{0,2} \|^2
  \right )
  $$
  $$
  +
  \left (
  2\|\phi_4 \|_{L^\infty(\R, L^\infty(\R^n))}
  +\|\phi_3 \|_{L^\infty(\R, L^\infty(\R^n)\cap
  L^1(\R^n)  )}
  \right ) \int_\tau^t  e^{\eta s}
   \E \left (
  \| v (s) \|^2 \right )
 ds
          $$
          $$
          + 3 \|\phi_7 \|_{L^\infty(\R, L^\infty(\R^n)
  \cap L^1(\R^n) )}
   \int_\tau^t  e^{\eta s}
   \E \left (
  \| v (s) \|^2 \right )
          $$
\be\label{im1 p2}
          +
           2\|L_\sigma \|^2_{l^2}
  \left (\|\kappa\|^2_{L^\infty (\R^n)}
  + \| \kappa \|^2
  \right )\int_\tau^t  e^{\eta s}
   \E \left (
  \| v (s) \|^2 \right ) ds.
  \ee
      By \eqref{etaca} and \eqref{im1 p2}
         we obtain,
        for  all $t\ge \tau$,
         \be\label{im1 p3}
 \E \left ( \| v  ( t, \tau )  \|^2
 \right )
 \le e^{\eta (\tau -t)}
  \E \left ( \| u_{0,1} -u_{0,2} \|^2
  \right )  .
\ee
Replacing $t$ and $\tau$ by $\tau$ and
$\tau -t$ in \eqref{im1 p3}, respectively,
we obtain for all $t\ge 0$,
$$
 \E \left ( \| u  ( \tau, \tau-t, u_{0,1} )
 -
 u  ( \tau, \tau-t, u_{0,2} ) \|^2
 \right )
 \le e^{- \eta  t }
  \E \left ( \| u_{0,1} -u_{0,2} \|^2
  \right )  ,
  $$
  which is the same as \eqref{im1 1}.
 \end{proof}

Next,  we   apply
  Lemma \ref{im1} to prove
  the measure attractor of \eqref{sde3}
  is a singleton.

\begin{thm}\label{main_s}
If  {\bf  (H1)}-{\bf (H4)},
\eqref{lamc}, \eqref{phi_g}-\eqref{phi_ga}
and \eqref{lamca}
hold, then:
\begin{enumerate}
\item [(i)]
The  $\mathcal{D}$-pullback
    measure attractor
$\mathcal{A}$ of
$\Phi$  is a singleton;  that is,
$\mathcal{A}= \{   \mu (\tau)
: \tau \in \mathbb{R}   \}$.

 \item[(ii)]
If   $f, g,  \sigma$, $\theta$ and $\phi_g$
are  all periodic functions in time with period $T$,
then the stochastic equation
\eqref{sde3} has a unique periodic measure
with periodic $T$ in $H$.

\item[(iii)]
If   $f, g,  \sigma$, $\theta$ and $\phi_g$
are  all  time independent,
then   the stochastic equation
\eqref{sde3} has a unique invariant  measure
 in $H$.
 \end{enumerate}
 \end{thm}

\begin{proof} (i).
Let
$\mathcal{A}= \{   \mathcal{A}   (\tau)
: \tau \in \mathbb{R}   \}$
be the unique  $\mathcal{D}$-pullback
    measure attractor of $\Phi$.
    We will prove   $\mathcal{A}   (\tau) $
    consists of only one point for every
    $\tau \in \R$. Let
    $\mu $ and $\nu$ be arbitrary points in
    $\mathcal{A}   (\tau) $. We need to prove
    $\mu =\nu$.

    Let  $\{t_n\}_{n=1}^\infty$ be a sequence of
    positive numbers such that $t_n \to \infty$.
    By the invariance of $\mathcal{A}$, for every
    $n\in \N$, there exist
    $\mu_{\tau -t_n},
    \nu_{\tau-t_n} \in \mathcal{A} (\tau -t_n)$
    such that
    $$
    \mu =\Phi(t_n, \tau -t_n, \mu_{\tau -t_n})
    \quad
    \text{ and}
    \quad
    \nu =\Phi(t_n, \tau -t_n, \nu_{\tau -t_n}).
    $$
    Let $u_{0,1},  u_{0,2}\in
    L^4(\Omega, \mathcal{F}_{\tau -t_n}, H)$
    such that
    $\mathcal{L}_{u_{0,1}}
    =\mu_{\tau -t_n}$  and
     $\mathcal{L}_{u_{0,2}}
    =\nu_{\tau -t_n}$.
    Then we have
     $$
    \mu =\mathcal{L}_{
    u(\tau,  \tau -t_n, u_{0,1} )}
    \quad \text{ and}
    \quad
    \nu  =\mathcal{L}_{
    u(\tau,  \tau -t_n, u_{0,2} )},
    $$
    where
   $u(\tau,  \tau -t_n, u_{0,1} )$
   and
   $u(\tau,  \tau -t_n, u_{0,2} )$
   are the solutions of \eqref{sde3}
   with initial data $u_{0,1}$ and
   $u_{0,2}$ at initial time $\tau-t_n$,
   respectively.
   By Lemma \ref{im1} we get
   $$
   d_{\mathcal{P} (H)}
   (\mu, \nu)
    =\sup_{\psi\in L_b (H), \|\psi\|_{L_b}
   \le 1}
   \left |
   \int_H \psi d\mathcal{L}_{
    u(\tau,  \tau -t_n, u_{0,1} )}
    - \int_H \psi \mathcal{L}_{
    u(\tau,  \tau -t_n, u_{0,2} )}
   \right |
   $$
    $$
    =\sup_{\psi\in L_b (H), \|\psi\|_{L_b}
   \le 1}
   \left |
   \E \left (
    \psi (
    u(\tau,  \tau -t_n, u_{0,1} )) \right )
    -\E \left (
    \psi (
    u(\tau,  \tau -t_n, u_{0,2} )) \right )
   \right |
   $$
    $$
   \le
   \E \left (
    \|
    u(\tau,  \tau -t_n, u_{0,1} )
   -
    u(\tau,  \tau -t_n, u_{0,2} )   \|
    \right )
    $$
    $$
    \le
     e^{- {\frac 12} \eta  t }
 \left (  \E \left ( \| u_{0,1} -u_{0,2} \|^2
  \right )
  \right )^{\frac 12}
   \le
    e^{- {\frac 12} \eta  t }
 \left (  2\E \left ( \| u_{0,1}\|^2
 +\|u_{0,2} \|^2
  \right )
  \right )^{\frac 12}
  $$
\be\label{im2 p1}
  \le
  2 e^{- {\frac 12} \eta  t } \|\mathcal{A} (\tau -t_n)\|
  _{\mathcal{P}_2(H)}.
 \ee
    Since $\mathcal{A} \in \mathcal{D}$, we
    have
   $$
   \lim_{t\to \infty}
    e^{ 2 \eta (\tau - t )} \|\mathcal{A} (\tau -t )\|^4
  _{\mathcal{P}_4(H)} =0,
  $$
  which implies  the right-side of \eqref{im2 p1}
  converges to zero as $t\to \infty$,
  and hence
  $\mu=\nu$.

  (ii).
  If   $f, g,  \sigma$, $\theta$ and $\phi_g$
are  all periodic functions in time with period $T$,
then by Theorem \ref{main_e}, we know
the measure attractor $\mathcal{A}$ is periodic
with period $T$, which along with (i)
shows that
$\mathcal{A} =\{ \mu (\tau): \tau
\in \R\}$
 and  $\mu(\tau +T)
=\mu (\tau)$ for all $\tau \in \R$.
Therefore,
 $\mu (0)$ is a periodic measure
of \eqref{sde3} with period $T$.

  (iii).
  If   $f, g,  \sigma$, $\theta$ and $\phi_g$
are  all
time independent,
then
by (ii) we know that
$\mu(\tau), \tau \in \R$,
 is $T$-periodic
for any $T>0$, and hence $\mu$ must be
an invariant measure.
  This completes the proof.
  \end{proof}

\section{Upper semicontinuity of pullback measure attractors}

In this section,
we discuss the limiting behavior of
pullback measure attractors
of \eqref{sdep1}
as $\eps \to 0$.
To that end, we assume  that
all the functions
$f^\eps$, $g^\eps$ and $\sigma_k^\eps$
in
\eqref{sdep1}
satisfy  the conditions
{\bf  (H1)}-{\bf (H4)}
uniformly with respect to $\eps\in (0,1)$.
Throughout this section, we assume
that
the functions
$f$, $g$ and $\sigma_k$
in
\eqref{sdel1}
satisfy the conditions
corresponding
to  {\bf  (H1)}-{\bf (H4)}
without  the terms involving   variables
$\mu, \mu_1$ or $\mu_2$.
Furthermore, we assume that
all  $\eps \in (0,1)$,  $t, u
    \in \mathbb{R}$,
    $x\in \R^n$  and $\mu \in \mathcal {P}_2(H)$,
\be\label{limc1}
  \left |
  f^\eps(t, x, u, \mu )-
  f(t, x, u )
  \right |
  \le
  \eps \phi_9 (t,x)
  ( |u| + \sqrt{\mu (\|\cdot \|^2) },
\ee
\be\label{limc2}
  \left |
   g^\eps(t, x, u, \mu )- g
  (t, x, u )
  \right |
  \le
  \eps \phi_9 (t,x)
  ( |u| + \sqrt{\mu (\|\cdot \|^2) },
\ee
and for every $k\in \N$,
\be\label{limc3}
  \left |
  \sigma_k^\eps(t,  u, \mu )-
 \sigma_k(t,  u)
  \right |
  \le
  \eps \phi_{10} (t) \widetilde{L} _{\sigma, k}
  ( |u| + \sqrt{\mu (\|\cdot \|^2) },
\ee
where
$\sum_{k=1}^\infty
\widetilde{L}^2 _{\sigma, k}<\infty$,
$\phi_9
\in L^1_{loc}(\R,
L^\infty (\R^n)
\cap  L^1(\R^n))$
and $\phi_{10} \in
\in L^2_{loc}(\R,
\R)$.

The next lemma is concerned with
the convergence of solutions
of \eqref{sdep1} as $\eps \to 0$.

\begin{lem}\label{up1}
If \eqref{limc1}-\eqref{limc3} hold, then
for every $\tau \in \R$,
$T>0$ and $R>0$,  the solutions of \eqref{sdep1}
and \eqref{sdel1} satisfy:
 $$
\lim\limits_{ \varepsilon  \to 0 }\
    \sup\limits_{ \|u_\tau\|_{L^2
    (\Omega, H)} \le R} \  \
    \sup\limits_{ \tau\le t\le \tau +T}
   \E \left (
    \|u^\eps (t, \tau, u_\tau)
    -u (t, \tau, u_\tau) \|^2
    \right )
 = 0.
$$
  \end{lem}

\begin{proof}
Let $ u^\varepsilon  ( t )
 =  u^\varepsilon ( t, \tau, u_\tau )   $,
 $ u  ( t )  =  u ( t, \tau, u_\tau )   $
 and
 $ v (t) = u^\varepsilon  ( t ) - u  ( t ) $.
By \eqref{sdep1} and \eqref{sdel1},
we get for all $ t >\tau $,
\begin{equation*}
\begin{split}
& dv(t)   -  \Delta v(t) dt  +  \lambda  v(t) dt
+   \Big(  f^\varepsilon \left( t, x, u^\varepsilon ( t ) ,  \mathcal{L}_{u^\varepsilon ( t )}  \right)
         -  f \left( t, x, u ( t )  \right)
    \Big)  dt
\\
&
=\Big(  g^\varepsilon \left( t, x, u^\varepsilon ( t ) ,  \mathcal{L}_{u^\varepsilon ( t )}  \right)
         - g \left( t, x, u ( t )  \right)
    \Big)  dt
+
\sum_{k=1}^\infty \kappa (x)
\Big(  \sigma^\varepsilon_k  ( t, u^\varepsilon ( t ) ,
\mathcal{L}_{u^\varepsilon ( t )}  )
         - \sigma_k  ( t, u ( t )  )
    \Big)  dW_k(t).
\end{split}
\end{equation*}
By It\^{o}'s formula, we have
for all $ t \ge \tau $,
\begin{equation}\label{up1 p1}
\begin{split}
\| v ( t )  \|^2
=& -  2 \int_\tau^t  \| \nabla v(s)  \|^2  ds
   -  2 \lambda   \int_\tau^t   \|  v(s) \|^2  ds
\\
&  - 2  \int_\tau^t
        \int_{\R^n}
        \left (
         f^\varepsilon \left( s, x, u^\varepsilon ( s,x ) ,
           \mathcal{L}_{u^\varepsilon ( s )}  \right)
               -  f \left( s, x, u ( s,x )  \right)
            v(s,x)
        \right ) dx  ds
        \\
&  + 2  \int_\tau^t
        \int_{\R^n}
        \left (
         g^\varepsilon \left( s, x, u^\varepsilon ( s,x ) ,
           \mathcal{L}_{u^\varepsilon ( s )}  \right)
               -  g \left( s, x, u ( s,x )  \right)
            v(s,x)
        \right ) dx  ds
  \\
&  +  \sum_{k=1}^\infty
\int_\tau^t
        \| \kappa
        ( \sigma^\varepsilon_k  ( s, u^\varepsilon ( s ) , \mathcal{L}_{u^\varepsilon ( s )}  )
           - \sigma_k  ( s, u ( s )  ) )
        \|  ^2
     ds
\\
&  + 2 \int_\tau^t
        \Big(  v(s)
              , \sum_{k=1}^\infty
              \kappa
              \left( \sigma^\varepsilon_k  ( s, u^\varepsilon ( s ) , \mathcal{L}_{u^\varepsilon ( s )}  )
                     - \sigma_k  ( s, u ( s )  )
                \right)   dW_k(s)
        \Big).
\end{split}
\end{equation}
For the third term on the right-hand side of \eqref{up1 p1},
by \eqref{f3}-\eqref{f4}  and \eqref{limc1}we have
$$
- 2  \int_\tau^t
        \int_{\R^n}
        \left (
         f^\varepsilon \left( s, x, u^\varepsilon ( s,x ) ,
           \mathcal{L}_{u^\varepsilon ( s )}  \right)
               -  f \left( s, x, u ( s,x )  \right)
            v(s,x)
        \right ) dx  ds
        $$
        $$
        = - 2  \int_\tau^t
        \int_{\R^n}
        \left (
         f^\varepsilon \left( s, x, u^\varepsilon ( s,x ) ,
           \mathcal{L}_{u^\varepsilon ( s )}  \right)
               -
               f^\varepsilon \left( s, x, u  ( s,x ) ,
           \mathcal{L}_{u ^\eps ( s )}  \right)
            v(s,x)
        \right ) dx  ds
        $$
         $$
          - 2  \int_\tau^t
        \int_{\R^n}
        \left (
         f^\varepsilon \left( s, x, u  ( s,x ) ,
           \mathcal{L}_{u^\eps ( s )}  \right)
               -
               f^\varepsilon \left( s, x, u  ( s,x ) ,
           \mathcal{L}_{u  ( s )}  \right)
            v(s,x)
        \right ) dx  ds
        $$
         $$
        - 2  \int_\tau^t
        \int_{\R^n}
        \left (
         f^\varepsilon \left( s, x, u  ( s,x ) ,
           \mathcal{L}_{u  ( s)}  \right)
               -
               f  \left( s, x, u  ( s,x )
              \right)
            v(s,x)
        \right ) dx  ds
        $$
        $$
        \le 2
        \int_\tau^t
        \int_{\R^n}\phi_4 (s,x) |v(s,x)|^2 dxds
        +2 \int_\tau^t
        \int_{\R^n}
        \phi_3  (s,x) |v(s,x)|  \sqrt{
        \E \left (
        \| v(s)\|^2
        \right )} dxds
        $$
        $$
        +2\eps
       \int_\tau^t
        \int_{\R^n}\phi_9 (s,x)
        ( |u(s,x)| |v(s,x)|
        +|v(s,x)|
        \sqrt{
        \E \left (
        \| u(s)\|^2
        \right )})  dxds
         $$
          $$
        \le 2
         \int_\tau^t
         \|\phi_4(s)\|_{
        L^\infty (\R^n)}
          \|v(s )\|^2  ds
          $$
          $$
        + \int_\tau^t
        \left (
        \|\phi_3  (s)\|_{L^\infty (\R^n)}
         \|v(s)\|^2
         + \|\phi_3  (s)\|_{L^1 (\R^n)}
         \E \left (
        \| v(s)\|^2
        \right )
        \right )  ds
        $$
        $$
        + \eps
       \int_\tau^t
         \left (
          \|\phi_9 (s) \|_{L^\infty
         (\R^n)}
         (2 \|v(s)\|^2
         + \| u(s) \|^2)
         + \|\phi_9 (s) \|_{L^1
         (\R^n)}
         \E \left (
        \| u(s)\|^2
        \right )
           \right )ds
         $$
          $$
        \le
         \int_\tau^t
         \left (
         2\|\phi_4(s)\|_{
        L^\infty (\R^n)}
        +
       \|\phi_3(s)\|_{
        L^\infty (\R^n)}
        +
        2 \eps \|\phi_9 (s)\|_{
        L^\infty (\R^n)}
        \right )
          \|v(s )\|^2  ds
          $$
         $$
         +    \int_\tau^t
          \|\phi_3 (s) \|_{L^1
         (\R^n)}
         \E \left (
        \| v(s)\|^2
        \right )
           ds
           $$
           \be\label{up1 p2}
           +  \eps  \int_\tau^t
          \left (
           \|\phi_9 (s) \|_{L^\infty
         (\R^n)} \|u(s)\|^2
         +\|\phi_9 (s) \|_{L^1
         (\R^n)}
         \E \left (
        \| u(s)\|^2
        \right )
           \right )ds.
           \ee

For the forth term on the right-hand side of \eqref{up1 p1},
by  \eqref{g2} and \eqref{limc2}  we get
         $$   2  \int_\tau^t
        \int_{\R^n}
        \left (
         g^\varepsilon \left( s, x, u^\varepsilon ( s,x ) ,
           \mathcal{L}_{u^\varepsilon ( s)}  \right)
               -  g \left( s, x, u ( s,x )  \right)
            v(s,x)
        \right ) dx  ds
        $$
         $$ =   2  \int_\tau^t
        \int_{\R^n}
        \left (
         g^\varepsilon \left( s, x, u^\varepsilon ( s,x ) ,
           \mathcal{L}_{u^\varepsilon ( s )}  \right)
               -  g^\eps
                \left( s, x, u ( s,x ) ,
                \mathcal{L}_{u  ( s )}
                 \right)
            v(s,x)
        \right ) dx  ds
        $$
         $$ +
           2  \int_\tau^t
        \int_{\R^n}
        \left (
         g^\varepsilon \left( s, x, u  ( s,x ) ,
           \mathcal{L}_{u  ( s )}  \right)
               -  g
                \left( s, x, u ( s,x )
                 \right)
            v(s,x)
        \right ) dx  ds
        $$
          $$ \le
             2  \int_\tau^t
        \int_{\R^n}
        \phi_7 (s,x)
       \left ( |v(s,x)|^2 +|v(s,x)|
        \sqrt{\E\left (
        \|v(s)\|^2
        \right )
        }
        \right ) dx ds
        $$
         $$
        +2\eps
       \int_\tau^t
        \int_{\R^n}\phi_9 (s,x)
        ( |u(s,x)| |v(s,x)|
        +|v(s,x)|
        \sqrt{
        \E \left (
        \| u(s)\|^2
        \right )})  dxds
         $$
         $$
        \le 2
         \int_\tau^t
         \|\phi_7(s)\|_{
        L^\infty (\R^n)}
          \|v(s )\|^2  ds
          $$
          $$
        + \int_\tau^t
        \left (
        \|\phi_7  (s)\|_{L^\infty (\R^n)}
         \|v(s)\|^2
         + \|\phi_7  (s)\|_{L^1 (\R^n)}
         \E \left (
        \| v(s)\|^2
        \right )
        \right )  ds
        $$
        $$
        + \eps
       \int_\tau^t
         \left (
          \|\phi_9 (s) \|_{L^\infty
         (\R^n)}
         (2 \|v(s)\|^2
         + \| u(s) \|^2)
         + \|\phi_9 (s) \|_{L^1
         (\R^n)}
         \E \left (
        \| u(s)\|^2
        \right )
           \right )ds
         $$
          $$
        \le
         \int_\tau^t
         \left (
         3\|\phi_7(s)\|_{
        L^\infty (\R^n)}
         +
        2 \eps \|\phi_9 (s)\|_{
        L^\infty (\R^n)}
        \right )
          \|v(s )\|^2  ds
          $$
         $$
         +    \int_\tau^t
          \|\phi_7 (s) \|_{L^1
         (\R^n)}
         \E \left (
        \| v(s)\|^2
        \right )
           ds
           $$
           \be\label{up1 p3}
           +  \eps  \int_\tau^t
          \left (
           \|\phi_9 (s) \|_{L^\infty
         (\R^n)} \|u(s)\|^2
         +\|\phi_9 (s) \|_{L^1
         (\R^n)}
         \E \left (
        \| u(s)\|^2
        \right )
           \right )ds.
           \ee

For the fifth term on the right-hand side of \eqref{up1 p1},
by
\eqref{s2} and \eqref{limc3}  we obtain
   $$
    \sum_{k=1}^\infty
\int_\tau^t
        \| \kappa
        ( \sigma^\varepsilon_k  ( s, u^\varepsilon ( s ) , \mathcal{L}_{u^\varepsilon ( s )}  )
           - \sigma_k  ( s, u ( s )  ) )
        \|  ^2
     ds
     $$
     $$
     \le
    2 \sum_{k=1}^\infty
\int_\tau^t
        \| \kappa
        ( \sigma^\varepsilon_k  ( s, u^\varepsilon ( s ) , \mathcal{L}_{u^\varepsilon ( s )}  )
           - \sigma^\eps_k  ( s, u ( s )  ,
           \mathcal{L}_{u ( s )} ) )
        \|  ^2
     ds
     $$
      $$
    +2
    \sum_{k=1}^\infty
\int_\tau^t
        \| \kappa
        ( \sigma^\varepsilon_k  ( s, u  ( s ) , \mathcal{L}_{u  ( s )}  )
           - \sigma _k  ( s, u ( s )   ))
        \|  ^2
     ds
     $$
     $$
     \le
     4\|L_\sigma\|^2_{l^2}
 \int_\tau^t
 \left (
 \| \kappa \|^2_{L^\infty (\R^n)}
 \|v(s)\|^2
 +\|\kappa \|^2
 \E \left (
 \|v(s)\|^2
 \right )
 \right )
 ds
 $$
    \be\label{up1 p4}
   +
     4\eps^2\sum_{k=1}^\infty
     {\widetilde{L}}_{\sigma, k}^2
   \int_\tau^t   | \phi_{10} (s)  |^2
 \left (
 \| \kappa \|^2_{L^\infty (\R^n)}
 \|u(s)\|^2
 +\|\kappa \|^2
 \E \left (
 \|u(s)\|^2
 \right )
 \right )
 ds  .
   \ee

     It follows from
     \eqref{up1 p1}-\eqref{up1 p4}
     that for all $t\ge \tau$ and $\eps\in (0,1)$,
    $$
\mathbb{E}
  \left (
        \|  v ( t ) \|^2
  \right )
\le
    \int_\tau^t
    c_1(s)  \mathbb{E}  \left (    \|  v(s) \|^2    \right )  ds
   +   \eps
   \int_\tau^t
    c_2(s)  \mathbb{E}  \left (    \|  u (s)\|^2
    \right ) ds,
$$
where $c_1,  c_2 \in L^1_{loc} (\R, \R^+)$,
which along with \eqref{apri_est} shows that
for all
$t\in  [\tau, \tau +T]$ and $\eps\in (0,1)$,
$$
\mathbb{E}
  \left (
        \|  v ( t ) \|^2
  \right )
\le
    \int_\tau^t
    c_1(s)  \mathbb{E}  \left (    \|  v(s) \|^2    \right )  ds
   +  \eps c_3
   (1+ \|u_\tau\|^2_{L^2(\Omega, H)})
   \int_\tau ^{\tau +T}  c_2 (s) ds,
 $$
where $c_3= c_3 (\tau, T)>0$ is a constant
independent of $\eps$.
Then by Gronwall's inequality, we get
for all   $\eps\in (0,1)$,
$t\in  [\tau, \tau +T]$ and
$u_\tau$ with $\|u_\tau\|_{L^2(\Omega, H)}
\le R$,
 $$
\mathbb{E}
  \left (
        \|  u^\eps (t, \tau,  u_\tau)
        -u  (t, \tau,  u_\tau)
         \|^2
  \right )
\le
    \eps  c_3
   (1+  R^2) e^{\int_\tau^{\tau +T}  c_1(s) ds}
   \int_\tau ^{\tau +T}  c_2 (s) ds,
 $$
 which concludes the proof.
 \end{proof}

We now write
 the
non-autonomous dynamical system associated
with  \eqref{sdep1}
as $\Phi^\eps$
to indicate its dependence
on $\eps$,  and  use $\Phi$ for  the
 dynamical system associated
with  \eqref{sdel1}.
The $\mathcal{D}$-pullback measure
attractors of $\Phi^\eps$ and
$\Phi$ are denoted by
$\mathcal{A}^\eps$  and
$\mathcal{A}$, respectively.

 As an immediate consequence of Lemma \ref{up1}, we have
the following convergence result
for  $\Phi^\eps$.

\begin{cor}\label{up2}
If \eqref{limc1}-\eqref{limc3} hold, then
for every  $t\in \R^+$,
$\tau \in \R$
  and $R>0$,
  $\Phi^\eps$ satisfies:
   $$
  \lim_{\eps \to 0}\
  \sup_{  \mu (\|\cdot\|^2) \le R}\
  d_{
  {\mathbb{W}_2}}
  \left (\Phi^\eps (t,\tau)\mu,
  \ \Phi (t,\tau)\mu
  \right )
  =0,
  $$
  and hence
$$
  \lim_{\eps \to 0}\
  \sup_{  \mu (\|\cdot\|^2) \le R}\
  d_{
  {\mathcal{P}(H) }}
  \left (\Phi^\eps (t,\tau)\mu,
  \ \Phi (t,\tau)\mu
  \right )
  =0.
  $$
  \end{cor}

Next, we prove   the upper semi-continuity of $\mathcal{D}$-pullback measure attractors
of \eqref{sdep1} as $ \varepsilon \rightarrow 0 $.

\begin{thm}\label{main_up}
If \eqref{limc1}-\eqref{limc3} hold, then
  for every $ \tau \in \mathbb{R} $,
$$
\lim\limits_{ \varepsilon \rightarrow 0 }
   d_{ \mathcal{P}(H) ) }
      \Big(   \mathcal{A}^\varepsilon ( \tau ),   \mathcal{A}_0 ( \tau )   \Big)
=0,
$$
where $ d_{ \mathcal{P}(H) ) }$
is the Hausdorff semi-distance of sets  in
terms of $  d_{ \mathcal{P}(H) ) }$.
\end{thm}

\begin{proof}
Since $f^\eps$,
$g^\eps$ and $\sigma_k^\eps$
satisfy { (\bf H1)-(\bf H4)} uniformly with respect to
$\eps \in (0,1)$, we find that
the family
$K=\{K(\tau): \tau
\in \R\} $ given by
\eqref{ma2 1} is a
closed $\mathcal{D}$-pullback
absorbing set
of  $\Phi$ and
  $\Phi^\eps$
  for all  $\eps\in (0,1)$.
 In particular, for every $\tau \in \R$,
 $
 \mathcal{A} (\tau)$
 and  $\mathcal{A}^\eps  (\tau)$
 are   subsets of $K(\tau)$
 for all $\eps \in (0,1)$.

Since $\mathcal{A} $ is
the $\mathcal{D}$-pullback
measure  attractor of $\Phi$ in
$ (\mathcal{P}_4 (H ), d_{\mathcal{P}(H)})$,
we see that
for every $\delta>0$, there exists $ T = T ( \delta, \tau,
K) >0 $ such that
\be \label{main_up p1}
d_{ \mathcal{P} ( H ) }
\Big( \Phi ( T, \tau-T)  K( \tau-T  ), \mathcal{A} ( \tau )
\Big)
<  \frac{1}{2}  \delta.
\ee
 By  Corollary \ref{up2} we have
 $$
 \lim\limits_{ \varepsilon  \to 0 }
    \sup\limits_{ \mu \in   K ( \tau - T ) }
         d_{ \mathcal{P} (H ) }
\Big(   \Phi^\varepsilon ( T, \tau - T ) \mu,
 \Phi  ( T, \tau - T ) \mu    \Big)
 = 0,
$$
and hence, there exists
  $ 0 < \varepsilon_1 < 1 $
   such that for all $ 0< \varepsilon < \varepsilon_1$,
\be \label{main_up p2}
\sup\limits_{ \mu \in   K ( \tau - T ) }
    d_{ \mathcal{P}  ( H ) }
       \Big(    \Phi^\varepsilon ( T, \tau - T ) \mu,
       \Phi ( T, \tau - T ) \mu     \Big)
<   \frac{1}{2}  \delta.
 \ee
Since   $ \mathcal{A}
^\varepsilon ( \tau-T ) \subseteq K  ( \tau-T ) $,
by  \eqref{main_up p1} and \eqref{main_up p2}, we get
 \be \label{main_up p3}
d_{ \mathcal{P} ( H ) }
\Big( \Phi ( T, \tau-T)
\mathcal{A}^\eps( \tau-T  ), \mathcal{A} ( \tau )
\Big)
<  \frac{1}{2}  \delta.
\ee
 and for all
 $  0< \varepsilon < \varepsilon_1$,
 \be \label{main_up p4}
\sup\limits_{ \mu \in   {\mathcal{A}}^\eps
 ( \tau - T ) }
    d_{ \mathcal{P}  ( H ) }
       \Big(    \Phi^\varepsilon ( T, \tau - T ) \mu,
       \Phi ( T, \tau - T ) \mu     \Big)
<   \frac{1}{2}  \delta.
  \ee
 By \eqref{main_up p3}-\eqref{main_up p4},
 we have for all $\varepsilon < \varepsilon_1$,
\be \label{main_up p5}
\sup\limits_{ \mu \in   \mathcal{A}^\varepsilon ( \tau - T ) }
    d_{ \mathcal{P} ( H ) }
       \Big(    \Phi^\varepsilon ( T, \tau - T ) \mu, \mathcal{A}_0 ( \tau )   \Big)
<     \delta.
\ee
By \eqref{main_up p5}
and
  the invariance of $ \mathcal{A}^\varepsilon $,
  we get  for all $\varepsilon < \varepsilon_1$,
  $$\sup\limits_{ \mu \in   \mathcal{A}^\varepsilon ( \tau  ) }
    d_{ \mathcal{P}  (H ) }
       \Big( \mu, \mathcal{A}  ( \tau )   \Big)
<     \delta,
$$
which completes the proof.
 \end{proof}

By Theorem \ref{main_s} and Theorem \ref{main_up}
we immediately obtain the convergence
of invariant measures and periodic measures
of \eqref{sdep1} as $\eps \to 0$
when the measure attractors are   singletons.

\begin{cor}\label{ip_con}
If  {\bf  (H1)}-{\bf (H4)},
\eqref{lamc}, \eqref{phi_g}-\eqref{phi_ga},
 \eqref{lamca}
 and \eqref{limc1}-\eqref{limc3}
hold, then:
\begin{enumerate}
  \item[(i)]
If
$f^\eps, g^\eps,  \sigma^\eps$,
$f, g,  \sigma$,   $\theta$ and  $\phi_g$
are  all periodic functions in time with period $T$,
then the  unique $T$-periodic   measure
of \eqref{sdep1} converges to that
of \eqref{sdel1} as $\eps \to 0$.

\item[(ii)]
If
$f^\eps, g^\eps,  \sigma^\eps$,
$f, g,  \sigma$,   $\theta$ and  $\phi_g$
are  all  time independent,
 then the  unique  invariant   measure
of \eqref{sdep1} converges to that
of \eqref{sdel1} as $\eps \to 0$.
  \end{enumerate}
 \end{cor}

\bibliography{References}

\end{document}